\documentclass{aic}

\usepackage[british]{babel}
\usepackage[utf8]{inputenc}
\usepackage{csquotes}

\usepackage{tabularx}
\usepackage{array}
    \newcolumntype{C}{>{$}c<{$}}
    \newcolumntype{L}{>{$}l<{$}}
    \newcolumntype{R}{>{$}r<{$}}
    \newcolumntype{M}[1]{>{\centering\arraybackslash \(}m{#1}<{\)}}
\usepackage{wrapfig}
\usepackage{subcaption}
\numberwithin{equation}{section}
\usepackage{relsize}
\usepackage{mathtools}	
\usepackage{accents}
\usepackage{multirow}
\usepackage{graphicx}
\usepackage{amscd}
\usepackage{epic, eepic}
\usepackage{url}
\usepackage{xcolor}
	\definecolor{SkyBlue}{cmyk}{0.62,0,0.12,0}
	\definecolor{BurntOrange}{cmyk}{0,0.51,1,0}
	\definecolor{ForestGreen}{cmyk}{0.91,0,0.88,0.12}
	\definecolor{CarnationPink}{cmyk}{0,0.63,0,0}
    \definecolor{OliveGreen}{cmyk}{0.64,0,0.95,0.40}
\usepackage{todonotes}
\usepackage{enumerate}
\usepackage{enumitem}
\usepackage{comment}
\usepackage{hhline}
\addto\extrasbritish{%
}
\usepackage{makecell}
\usepackage{tcolorbox}

\usepackage{doi}
\usepackage[numbers,sort&compress]{natbib}
\usepackage{tikz}
    \tikzset{
    labelbox/.style={
        anchor=center,
        fill=white,
        fill opacity=0.6,
        text opacity=1,
        inner sep=3pt,
        rounded corners=10pt
    },
    sparse dotted/.style={
        dash pattern=on 0.5pt off 6pt
    }
}
\usepackage{tikz-cd}
    \tikzcdset{
      cells={font=\everymath\expandafter{\the\everymath\displaystyle}},
    }
\usepackage{ifthen}
\usepackage{pgfplots}
\usepackage{mathrsfs}
    \pgfplotsset{compat=1.15}
    \usetikzlibrary{arrows}
    \usetikzlibrary{patterns}
    \usetikzlibrary{decorations.pathreplacing}
    \usetikzlibrary{shapes.geometric}

\usepackage{chngcntr}
    \counterwithin{table}{section}
    \counterwithin{figure}{section}

\input{figures/def_patterns.tex}

\newcommand\PG{{\rm PG}}
\newcommand\F{{\mathbb F}}
\newcommand\AG{{\rm AG}}

\newcommand\cF{{\mathcal F}}
\newcommand\cG{{\mathcal G}}

\newcommand\cL{{\mathcal L}}

\newcommand\cP{{\mathcal P}}

\newcommand\cR{{\mathcal R}}
\newcommand\cS{{\mathcal S}}

\newcommand\EE{{\mathbb E}}
\renewcommand{\P}{\mathbb P}

\newcommand{\zz}{\mathbb{Z}}
\newcommand{\rr}{\mathbb{R}}
\newcommand{\Z}{\mathbb{Z}}
\newcommand{\eps}{\varepsilon}

\newcommand{\R} {\mathbb R}

\newcommand{\N} {\mathbb N}

\newcommand{\leteq}{\coloneqq}
\newcommand{\eqlet}{\eqqcolon}
\newcommand{\cardinality}[1]{|#1|}
\newcommand{\Cardinality}[1]{\left|#1\right|}

\newcommand{\from}{\colon}

\theoremstyle{plain}
\newtheorem{theorem}{Theorem}[section]
\newtheorem{prop}[theorem]{Proposition}
\newtheorem{proposition}[theorem]{Proposition}

\newtheorem{observation}[theorem]{Observation}
\newtheorem{observations}[theorem]{Observations}
\newtheorem{lemma}[theorem]{Lemma}

\newtheorem{corollary}[theorem]{Corollary}

\newtheorem{construction}[theorem]{Construction}

\theoremstyle{definition}
\newtheorem{defi}[theorem]{Definition}

\newtheorem{definition}[theorem]{Definition}

\newtheorem{remark}[theorem]{Remark}

\newcommand{\indicator}[1]{\mathbf{1}_{\{#1\}}}
\newcommand{\bigO}{O}
\newcommand{\littleO}{o}
\newcommand{\when}{,& \text{ if\quad}}
\newcommand{\otherwise}{,& \text{ otherwise}}

\allowdisplaybreaks

\aicAUTHORdetails{%
  title = {Randomised Algebraic Constructions for the No-$(k+1)$-in-Line Problem}, 
  author = {Benedek Kov{\'a}cs, Zolt{\'a}n L{\'o}r{\'a}nt Nagy, and D{\'a}vid R. Szab{\'o}},
  plaintextauthor = {Benedek Kovacs, Zoltan Lorant Nagy, David R. Szabo},
  plaintexttitle = {Randomised Algebraic Constructions for the No-(k+1)-in-Line Problem}, 
  keywords = {no-(k+1)-in-line problem, algebraic constructions, probabilistic method},
}   

\aicEDITORdetails{%
   year={2026},
   number={7},
   received={5 September 2025},   
   published={18 September 2026},  
   doi={10.19086/aic.2026.7},      
}   

\begin{document}

\begin{frontmatter}[classification=text]

\title{Randomised Algebraic Constructions for the No-$(k+1)$-in-Line Problem} 

\author[benedek]{Benedek Kov{\'a}cs\thanks{Supported by the EKÖP-25 University Research Scholarship Program of the Ministry for Culture and Innovation from the source of the National Research, Development and Innovation Fund, the University Excellence Fund of Eötvös Loránd University, and EXCELLENCE NKKP\_24 grant number 151504.}}
\author[zoli]{Zolt{\'a}n L{\'o}r{\'a}nt Nagy\thanks{Supported by the University Excellence Fund of Eötvös Loránd University and the János Bolyai Research Scholarship of the Hungarian Academy of Sciences.}}
\author[david]{D{\'a}vid R. Szab{\'o}\thanks{Supported by the University Excellence Fund of Eötvös Loránd University.}}

\begin{abstract} The {no-$(k+1)$-in-line} problem seeks the
 maximum number $f_k(n)$ of points that can be selected from an $n \times n$ square lattice such that no $k+1$ of them are collinear. The problem was first posed more than $100$ years ago for the special case $k=2$ and has remained open ever since. The general problem was recently resolved in the case $k$ is not small compared to $n$, as Kovács, Nagy and Szabó
 proved  that the upper bound $kn$ can be attained, provided that $k>C\sqrt{n\log{n}}$ for an absolute constant $C$.
 
 In this paper, we show that  $\left(1-\tfrac{2}{k}\right)kn \leq f_k(n)\leq kn$ and $\left(1-\tfrac{3}{k}\right)kn \leq f_k(n)\leq kn$ hold for every  even $k$ and odd $k$, respectively, provided that $n$ is large enough. This is asymptotically tight as $k\to \infty$. Previously, only $f_k(n)=\Omega(kn)$ was known due to Lefmann.

We present further improvements on the lower bounds for constant values of $k$ when $k<23$ holds. All these bounds are based on randomised algebraic constructions.
\end{abstract}
\end{frontmatter}

\section{Introduction}

A set of points in the plane is said to be in \emph{general position} if no three of the points lie on a common line.  Following the riddle of Dudeney  concerning the placement of 16 chess pieces on the chessboard in general position, the no-three-in-line problem was formulated, asking to determine the maximum number of grid points of an $n \times n$ grid so that the points are in general position. For the history and background of this celebrated problem, we refer to the excellent book of Brass, Moser, and Pach \cite{Brass2005} and that of Eppstein \cite{Eppstein:2018}. 
It is easy to see using the pigeonhole principle that $2n$ is an upper bound. This theoretical upper bound can be attained when $n$ is small \cite{anderson1979update, flammenkamp1998progress, prellberg2026constraint}.
Surprisingly, although it has been studied extensively, no progress has been reported concerning the lower bound, up to our best knowledge, since the result of 
Hall, Jackson, Sudbery, and Wild \cite{hall1975some}. They showed fifty years ago that one can choose $1.5n - o(n)$ points in general position from the $n\times n$ grid. 

Several conjectures were formulated concerning
 the \textit{asymptotical value} of the answer to the no-three-in-line problem, and it is far from clear whether it is  $1.5n$, $2n$  or somewhere in between \cite{Brass2005, Eppstein:2018,  green100, Guy/Kelly:1968}.

Our main focus is to continue the investigation of the generalization of this problem.

	\begin{defi}
		Let $f_k(n)$ denote the maximum size of a point set chosen from the points  of  the $n\times n$ grid $\cG=[1, n]\times [1, n]$ in which at most $k$ points lie on any line.
	\end{defi}

Using this notation, the aforementioned cornerstone result is the following.

    \begin{theorem}[Hall, Jackson, Sudbery, and Wild \cite{hall1975some}]\label{hall}
         \[ 1.5n-o(n) \le f_2(n)\le 2n.\]
    \end{theorem}
Their construction relies on a well-chosen subset of the modular hyperbola $\{(x,y)\in \mathbb{Z}^2: xy\equiv c \pmod p\}$ where $p$ denotes a prime number close to $n/2$, and $c$ is a nonzero (mod $p$) residue. This construction, which we call the \textit{HJSW construction} for short, will play a central role in our paper and will be discussed in detail in \autoref{sec:prelim}.

Later Lefmann, using graph-theoretic tools, proved that for fixed $k$, $f_{k}(n)=\Omega(kn)$ holds for the general case \cite{lefmann2012extensions}.
Here we note that actually, simply taking the disjoint union of $\left\lfloor k/2\right\rfloor$ distinct constructions of HJSW type where $p\in \left[\left(\frac12-o(1)\right)n, \frac12n\right]$ is fixed and $c$ varies, we can obtain the following stronger observation for free from \autoref{hall}.
  
\begin{prop}[Multiplicativity] \label{multip}
     Let $k\geq 2$ be a fixed integer. Then as $n\to\infty$, we have 
    \begin{align*}
        (\tfrac{3}{4}-o(1))\cdot kn &\leq f_k(n) \when 2\mid k, \\
        \left(\tfrac{3}{4}-\tfrac{3}{4k}-o(1)\right)\cdot kn &\leq f_k(n) \when 2\nmid k.
    \end{align*}
 
\end{prop}

However, the present authors were recently able to show much stronger results, as we determined \textit{the precise value} of $f_k(n)$ when $k$ is not too small.

\begin{theorem}[Kovács, Nagy, Szabó \cite{kovacs2025settling}]\label{main_old} 
	For every $C>12.5$, there exists an integer $N_C$ such that 
	whenever $n\geq N_C$ and $k\geq C\sqrt{n\log n}$, we have 
	\[f_k(n)=kn.\]
\end{theorem}

In this paper, we have two main objectives. First, we wish to extend the range of $k=k(n)$ where $f_k(n)=(1-o(1))kn$ holds. In this direction, our main result is the following.

\begin{theorem}\label{k_asymp}
     For every $k\geq 2$, there exists $N_k$ such that whenever $n\geq N_k$, we have 
    \begin{align*}
        kn-2n<\left(1-\frac{2}{k} + \frac{6}{k\sqrt{\pi k + 14}} - \frac{2}{k(k+2)}\right) kn &\leq f_k(n)\leq kn \when 2\mid k, \\
        kn-3n<\left(1-\frac{3}{k} + \frac{10}{k\sqrt{\pi k + 7}} - \frac{6}{k(k+1)}\right)kn &\leq f_k(n)\leq kn \when 2\nmid k.
    \end{align*}
\end{theorem}
We prove a slightly stronger result in \autoref{cor:asym}. 
Note that for fixed $k$, \autoref{main_old} determines the exact value of $f_k(n)$ for small values of $n$, while \autoref{k_asymp} gives a near-optimal estimate for all large enough values of $n$.

While  \autoref{hall} and its consequence, \autoref{multip} relied on a purely algebraic construction, namely the line intersection patterns of conics over a finite field, the recently proved exact result of \autoref{main_old} was essentially purely probabilistic in nature, as its construction was built  from  carefully constructed blockwise uniform random bipartite graphs.  To obtain \autoref{k_asymp}, we apply the combination of the two techniques. 

\autoref{multip} and the construction of Hall et al. \cite{hall1975some} indicate that obtaining good lower bounds might depend on the parity of $k$. We investigate the case $k=3$ in detail, in addition to the cases when $k$ is a small constant, and improve upon the previous general lower bound with a different approach.

\begin{theorem}\label{k=3}
    $f_3(n)\ge  1.973 n$ for every sufficiently large $n$.
\end{theorem}

\begin{theorem}\label{k_small} 
    $f_k(n)\ge (C_k-o(1))kn$ holds for $k\in \{2,\dots,22\}$ with coefficients $C_k$ presented in  \autoref{tab:summary}.
    \begin{table}[!htb]
	\centering
		\begin{tabular}{C||R@{${}\approx{}$}Ll}
			k & \multicolumn{2}{C}{C_k}  & proof \\ \hline\hline
            \\   
			2 & \frac{3}{4} & 0.75000 & \S\ref{sec:prelim} or \S\ref{sec:smallK}\\
			3 & \frac{26521}{40320} & 0.65776 & \S\ref{sec:k=3} and \S\ref{sec:full_k=3}\\
			4 & \frac{3}{4} & 0.75000 & \S\ref{sec:constructionGeneral} or \S\ref{sec:smallK} \\
			5 & \frac{41}{60} & 0.68333 & \S\ref{sec:constructionGeneral} or  \S\ref{sec:smallK}\\
			6 & \frac{59}{72} & 0.81944 & \S\ref{sec:smallK}\\
			7 & \frac{171}{224} & 0.76339 & \S\ref{sec:constructionGeneral}\\
			8 & \frac{109}{128} & 0.85156 & \S\ref{sec:constructionGeneral} or \S\ref{sec:smallK}\\
            9 & \frac{193}{240} & 0.80417 & \S\ref{sec:constructionGeneral}\\
			10 & \frac{559}{640} & 0.87344 & \S\ref{sec:constructionGeneral} or \S\ref{sec:smallK}\\
			11 & \frac{5}{6} & 0.83333 & \S\ref{sec:constructionGeneral}\\
		\end{tabular}
		\quad
		\begin{tabular}{C||R@{${}\approx{}$}Ll}
			k & \multicolumn{2}{C}{C_k} & proof \\ \hline\hline	
			12 & \frac{229}{256} & 0.89453 & \S\ref{sec:constructionGeneral}\\
			13 & \frac{4969}{5824} & 0.85319 & \S\ref{sec:constructionGeneral}\\
			14 & \frac{1621}{1792} & 0.90458 & \S\ref{sec:constructionGeneral}\\
			15 & \frac{13379}{15360} & 0.87103 & \S\ref{sec:constructionGeneral}\\
			16 & \frac{11255}{12288} & 0.91593 & \S\ref{sec:constructionGeneral}\\
			17 & \frac{34559}{39168} & 0.88233 & \S\ref{sec:constructionGeneral}\\
			18 & \frac{2837}{3072} & 0.92350 & \S\ref{sec:constructionGeneral}\\
			19 & \frac{43457}{48640} & 0.89344 & \S\ref{sec:constructionGeneral}\\
			20 & \frac{4763}{5120} & 0.93027 & \S\ref{sec:constructionGeneral}\\
			21 & \frac{17775}{19712} & 0.90173 & \S\ref{sec:constructionGeneral}\\
			22 & \frac{42307}{45056} & 0.93899 & \S\ref{sec:constructionGeneral}\\
			\end{tabular}
	\caption{The values of $C_k$, together with the Section number where the corresponding bound is proved.}
	\label{tab:summary}
\end{table}
\end{theorem}
In this statement, the bound $k\leq 22$ is chosen arbitrarily. For any concrete integer $k$, one could compute a lower bound slightly better than in \autoref{k_asymp} as detailed in \autoref{rem:explicitVsStirling}.

Let us compare these results to the field analogue of the problem, when the maximum number of points is to be determined in finite projective planes $\PG(2,q)$ (or finite affine planes $\AG(2,q)$) over a finite field $\F_q$ of order $q$, provided that every line intersects the point set in at most $k$ points. Such point sets of size $K$ are called $(K, k)$-arcs in finite projective geometries. In general, for  small $k$, i.e., $k<q/2$, the best constructions are obtained from the union of $\lfloor k/2 \rfloor $ conics \cite{ball2005bounds}.
 We mention that in the case $k=3$, there is a construction of size $q + \lfloor2\sqrt{q}\rfloor$, (which only slightly exceeds the lower bound for $k=2$,)
while the known
upper bound is $K\le 2q + 1$. Ball and Hirschfeld cite that any progress on determining a constant $c$ such that the
upper bound $n/q < c< 2$ for $q$ large enough, or a construction where $n/q > c> 1$
for infinitely many $q$ will be rewarded by a cheque for 10,000 Hungarian forints from
Prof. A. Blokhuis \cite{ball2005bounds}. This underlines the different behavior of the cases $k$ even and $k$ odd.

\paragraph{Outline of the paper} 
Our paper is organized as follows. In \autoref{sec:prelim}, we set the main notation and the general tools for the proofs. We also recall the main ideas and properties of the construction for $k=2$  by Hall et al. \cite{hall1975some}. It relies on a construction over $\F_p\times \F_p$ with $p\le n$ and the fact that for every $n$ there is a prime number $p=p_n\le n$ with $p_n/n\to 1$ (see \cite{Baker/Harman/Pintz:2001}) which enables us to embed the construction to $[1,n]\times [1,n]$ and get the same asymptotical result as for prime numbers. This theme will be applied in our constructions as well: the size of our randomised constructions will be parametrized by an arbitrary odd prime number, and this construction will  be embedded into the grid $[1,n]\times [1,n]$.

Then in \autoref{sec:constructionGeneral}, we use this construction as a building block to prove \autoref{k_asymp}. The more technical part of the argument, consisting mainly of the required computations, is carried out in the next section.  The general lower bound of  \autoref{k_asymp} can be slightly improved for sporadic cases of $k$ via the combination of two further  building blocks, relying on the  construction for $k=2$  by Hall et al., and 
\autoref{sec:smallK} is dedicated to the refinement of the techniques detailed in \autoref{sec:constructionGeneral} in order to get improved bounds on $f_k(n)$ when $k$ is a small constant. 
Using a slightly different technique, we improve the lower bound in the case $k=3$ and sketch the proof of \autoref{k=3} in \autoref{sec:k=3} to complete the table of \autoref{k_small}. The exact calculations are carried out in \autoref{sec:full_k=3}.
Finally, we discuss some open and related problems in \autoref{sec:conclusion}.

\paragraph{Notes added} Following the appearance of the current paper, Kwan and Grebennikov were able to extend \autoref{main_old} and prove the following general result:

\begin{theorem}[Kwan, Grebennikov \cite{grebennikov2025no}]\label{kwan} 
For   $n\ge k\geq 10^{37}$, we have 
	\[f_k(n)=kn.\]
\end{theorem}

This result thus gives an exact answer whenever $k$ is large (independently of $n$). Our result mainly focuses on the behavior of $f_k(n)$ when $k$ is rather small.
\section{Preliminaries}\label{sec:prelim}

Throughout the paper, $\F_p=\{0,1,\dots,p-1\}$ denotes the set of integers modulo $p$, and $\F_p^{*}=\F_p\setminus \{0\}$. We let $\mathbb{N}$ denote the set of nonnegative integers.

\subsection{Grids and secants}	
	Let us recall the following definitions and notations from the previous paper by the authors \cite{kovacs2025settling}.
	The notation $[a,b]$ refers to the integer interval $\{x\in \zz: a\le x\le b\}$. A \emph{grid} is a set $\cG=I_1\times I_2$ where $I_1=[a_1,b_1]$ and $I_2=[a_2,b_2]$ are integer intervals. We say that $\cG$ is an $n_1\times n_2$ grid, where $n_j=|I_j|=b_j-a_j+1$ for $j=1,2$. We call $\cG$ a \emph{square grid} if $n_1=n_2$.
	
	A secant $\ell$ of a grid $\cG$ is a line in the Euclidean plane that meets $\cG$ in at least two points. We call a line $\ell$ \emph{generic} if it is not horizontal or vertical. Note that each generic secant $\ell$ has a unique direction vector $\mathbf{v}=(v_x, v_y)$ such that $v_x\in \zz^{+}$, $v_y\in \zz\setminus \{0\}$ and $\gcd(v_x, v_y)=1$. The \emph{modulus} of $\ell$ is $M=\max(|v_x|,|v_y|)$.
	
	Given an integer $k\ge 2$, we call $S\subseteq \cG$ a \emph{no-$(k+1)$-in-line set} if $|S\cap \ell|\le k$ for every Euclidean line $\ell$ (or equivalently, for every secant $\ell$ of $\cG$).
	
\subsection{The no-three-in-line construction of Hall, Jackson, Sudbery and Wild}
\label{sec:HJSW-2}
	
	In this subsection, we recall the main construction of Hall, Jackson, Sudbery and Wild in \cite{hall1975some}, which gives a no-three-in-line set of size $3(p-1)$ in a $2p\times 2p$ grid for any prime $p\ge 3$. We start by recalling their most important definitions and lemmas, as our new constructions are also based on these.
	
	\begin{defi}[Congruent points in $\Z\times\Z$]
		Two points $(x_1,y_1)$ and $(x_2,y_2)$ are called \emph{congruent mod $p$}, or simply \emph{congruent}, if $x_1\equiv x_2\pmod{p}$ and $y_1\equiv y_2\pmod{p}$. A \emph{class} is the set of all points in $\zz\times \zz$ congruent to a certain point. Note that classes naturally correspond to elements of $\F_p\times \F_p$. The \emph{base point} of a class is its unique element in $[0,p-1]\times [0,p-1]$.
	\end{defi}
	
	\begin{definition}[Modular hyperbola]\label{def:modularHperbola}
		Let $c\in \F_p^{*}$. Define the \emph{modular hyperbola} to be  $$H(c,p):=\{(x,y)\in \zz\times \zz: xy\equiv c ~(\mathrm{mod}~p)\}.$$
	\end{definition}
	
	A point $(x,y)\in \zz\times \zz$ will be called \emph{normal} if $p\nmid x$ and $p\nmid y$. In \cite{hall1975some}, Hall et al. introduce the following way to partition the normal points of $\zz\times \zz$ into subgrids of size $\frac{p-1}{2}\times \frac{p-1}{2}$, which we will call \emph{blocks}.
	
	\begin{definition}[Blocks]\label{def:blocks}
		
		Define the following four $\frac{p-1}{2}\times \frac{p-1}{2}$ grids:
		\begin{align*}
		A_{0,0} &= \left[1, \frac{p-1}{2}\right] \times \left[\frac{p+1}{2}, p-1\right],\\
		B_{0,0} &= \left[\frac{p+1}{2}, p-1\right] \times \left[\frac{p+1}{2}, p-1\right],\\
		C_{0,0} &= \left[1, \frac{p-1}{2}\right] \times \left[1, \frac{p-1}{2}\right],\\
		D_{0,0} &= \left[\frac{p+1}{2}, p-1\right] \times \left[1, \frac{p-1}{2}\right].
		\end{align*}
		For $r,s\in \zz$, let $A_{r,s}=A_{0,0}+(rp,sp)$, and similarly define $B_{r,s}$, $C_{r,s}$, $D_{r,s}$. \autoref{fig:blocks_hjsw} illustrates these sets for $p=11$.
        
        We let \[A=\bigcup\limits_{r,s\in \zz} A_{r,s},\] which is the union of all classes that have their base point in $A_{0,0}$; and we similarly define $B$, $C$, $D$. Note that the disjoint union $A\sqcup B\sqcup C\sqcup D$ is equal to the set of all normal points of $\zz\times \zz$.
	\end{definition}

    \begin{figure}[htb!]
    \centering
        \null\hfill
        \begin{subfigure}[t]{0.33\textwidth}
        	\definecolor{ffqqqq}{rgb}{1.,0.,0.}
\definecolor{qqqqff}{rgb}{0.,0.,1.}
\definecolor{ffqqff}{rgb}{1.,0.,1.}
\definecolor{qqffqq}{rgb}{0.,1.,0.}
\definecolor{lightgggrey}{rgb}{0.7,0.7,0.7}
\resizebox{\textwidth}{!}{
\begin{tikzpicture}[line cap=round,line join=round,>=triangle 45,x=0.25cm,y=0.25cm]
\draw [step=1,gray,line width=0.3pt, opacity=0.5] (-5,0) grid (16,21);

\draw [line width=1.5pt] (0,0)--(0,21);
\draw [line width=1.5pt] (11,0)--(11,21);
\draw [line width=1.5pt] (-5,0)--(16,0);
\draw [line width=1.5pt] (-5,11)--(16,11);

\draw [fill=blue, color=blue, line width=1pt, fill opacity=0.3] (-5,1)--(-5,5)--(-1,5)--(-1,1)--cycle;
\draw [fill=yellow, color=yellow, line width=1pt, fill opacity=0.3] (-5,6)--(-5,10)--(-1,10)--(-1,6)--cycle;
\draw [fill=blue, color=blue, line width=1pt, fill opacity=0.3] (-5,12)--(-5,16)--(-1,16)--(-1,12)--cycle;
\draw [fill=yellow, color=yellow, line width=1pt, fill opacity=0.3] (-5,17)--(-5,21)--(-1,21)--(-1,17)--cycle;

\draw [fill=green, color=green, line width=1pt, fill opacity=0.3] (5,1)--(5,5)--(1,5)--(1,1)--cycle;
\draw [fill=red, color=red, line width=1pt, fill opacity=0.3] (5,6)--(5,10)--(1,10)--(1,6)--cycle;
\draw [fill=green, color=green, line width=1pt, fill opacity=0.3] (5,12)--(5,16)--(1,16)--(1,12)--cycle;
\draw [fill=red, color=red, line width=1pt, fill opacity=0.3] (5,17)--(5,21)--(1,21)--(1,17)--cycle;

\draw [fill=blue, color=blue, line width=1pt, fill opacity=0.3] (6,1)--(6,5)--(10,5)--(10,1)--cycle;
\draw [fill=yellow, color=yellow, line width=1pt, fill opacity=0.3] (6,6)--(6,10)--(10,10)--(10,6)--cycle;
\draw [fill=blue, color=blue, line width=1pt, fill opacity=0.3] (6,12)--(6,16)--(10,16)--(10,12)--cycle;
\draw [fill=yellow, color=yellow, line width=1pt, fill opacity=0.3] (6,17)--(6,21)--(10,21)--(10,17)--cycle;

\draw [fill=green, color=green, line width=1pt, fill opacity=0.3] (12,1)--(12,5)--(16,5)--(16,1)--cycle;
\draw [fill=red, color=red, line width=1pt, fill opacity=0.3] (12,6)--(12,10)--(16,10)--(16,6)--cycle;
\draw [fill=green, color=green, line width=1pt, fill opacity=0.3] (12,12)--(12,16)--(16,16)--(16,12)--cycle;
\draw [fill=red, color=red, line width=1pt, fill opacity=0.3] (12,17)--(12,21)--(16,21)--(16,17)--cycle;

\begin{scriptsize}

\foreach \x in {-5,-1,1,5,6,10,12,16}
{
\draw (\x,-0.78) node[anchor=mid] {$\x$};
}
\foreach \y in {0,1,5,6,10,11,12,16,17,21}
{
\draw (-6.2,\y) node[anchor=mid] {$\y$};
}
\end{scriptsize}

\draw (3,8) node[anchor=center] {$A_{0,0}$};
\draw (14,8) node[anchor=center] {$A_{1,0}$};
\draw (3,19) node[anchor=center] {$A_{0,1}$};
\draw (14,19) node[anchor=center] {$A_{1,1}$};

\draw (-3,8) node[anchor=center] {$B_{-1,0}$};
\draw (8,8) node[anchor=center] {$B_{0,0}$};
\draw (-3,19) node[anchor=center] {$B_{-1,1}$};
\draw (8,19) node[anchor=center] {$B_{0,1}$};

\draw (3,3) node[anchor=center] {$C_{0,0}$};
\draw (14,3) node[anchor=center] {$C_{1,0}$};
\draw (3,14) node[anchor=center] {$C_{0,1}$};
\draw (14,14) node[anchor=center] {$C_{1,1}$};

\draw (-3,3) node[anchor=center] {$D_{-1,0}$};
\draw (8,3) node[anchor=center] {$D_{0,0}$};
\draw (-3,14) node[anchor=center] {$D_{-1,1}$};
\draw (8,14) node[anchor=center] {$D_{0,1}$};

\end{tikzpicture}
}
        \end{subfigure}%
        \hfill
        \begin{subfigure}[t]{0.33\textwidth}
        	\definecolor{ffqqqq}{rgb}{1.,0.,0.}
\definecolor{qqqqff}{rgb}{0.,0.,1.}
\definecolor{ffqqff}{rgb}{1.,0.,1.}
\definecolor{qqffqq}{rgb}{0.,1.,0.}
\definecolor{lightgggrey}{rgb}{0.7,0.7,0.7}
\resizebox{\textwidth}{!}{
\begin{tikzpicture}[line cap=round,line join=round,>=triangle 45,x=0.25cm,y=0.25cm]

\draw [step=1,gray,line width=0.3pt, opacity=0.5] (-5,0) grid (16,21);

\draw [line width=1.5pt] (0,0)--(0,21);
\draw [line width=1.5pt] (11,0)--(11,21);
\draw [line width=1.5pt] (-5,0)--(16,0);
\draw [line width=1.5pt] (-5,11)--(16,11);

\draw [fill=blue, color=blue, line width=1pt, fill opacity=0.3] (-5,1)--(-5,5)--(-1,5)--(-1,1)--cycle;
\draw [fill=yellow, color=yellow, line width=1pt, fill opacity=0.3] (-5,6)--(-5,10)--(-1,10)--(-1,6)--cycle;
\draw [fill=blue, color=blue, line width=1pt, fill opacity=0.3] (-5,12)--(-5,16)--(-1,16)--(-1,12)--cycle;
\draw [fill=yellow, color=yellow, line width=1pt, fill opacity=0.3] (-5,17)--(-5,21)--(-1,21)--(-1,17)--cycle;

\draw [fill=green, color=green, line width=1pt, fill opacity=0.3] (5,1)--(5,5)--(1,5)--(1,1)--cycle;
\draw [fill=red, color=red, line width=1pt, fill opacity=0.3] (5,17)--(5,21)--(1,21)--(1,17)--cycle;

\draw [fill=blue, color=blue, line width=1pt, fill opacity=0.3] (6,1)--(6,5)--(10,5)--(10,1)--cycle;
\draw [fill=yellow, color=yellow, line width=1pt, fill opacity=0.3] (6,17)--(6,21)--(10,21)--(10,17)--cycle;

\draw [fill=green, color=green, line width=1pt, fill opacity=0.3] (12,1)--(12,5)--(16,5)--(16,1)--cycle;
\draw [fill=red, color=red, line width=1pt, fill opacity=0.3] (12,6)--(12,10)--(16,10)--(16,6)--cycle;
\draw [fill=green, color=green, line width=1pt, fill opacity=0.3] (12,12)--(12,16)--(16,16)--(16,12)--cycle;
\draw [fill=red, color=red, line width=1pt, fill opacity=0.3] (12,17)--(12,21)--(16,21)--(16,17)--cycle;

\begin{scriptsize}

\foreach \x in {-5,-1,1,5,6,10,12,16}
{
\draw (\x,-0.78) node[anchor=mid] {$\x$};
}
\foreach \y in {0,1,5,6,10,11,12,16,17,21}
{
\draw (-6.2,\y) node[anchor=mid] {$\y$};
}
\end{scriptsize}





\end{tikzpicture}
}
        \end{subfigure}
        \hfill\null
    \caption{The left side illustrates the blocks in \autoref{def:blocks} in the case $p=11$, and the right side shows which of these are kept in the set $T_2$ in \autoref{hallcon}.}
    \label{fig:blocks_hjsw}
    \end{figure}
	
	Recall the following observation from \cite{hall1975some}:
	
	\begin{observation}\label{obs:atmosttwoclasses}
		For any $c\in \F_p^{*}$ and any Euclidean line $\ell$, the intersection points of $H(c,p)$ and $\ell$ lie in at most two classes.
	\end{observation}
    
	For certain types of lines, Hall et al. \cite{hall1975some} describe their possible intersection patterns with $H(c,p)$ in more detail:
	
	\begin{observations}\label{obs:meet_classes}
		\begin{enumerate}[label=(\alph*)]
			\item A line of slope $0$ or $\infty$ meets $H(c,p)$ in at most one class.
			\item A line of slope $+1$ meets $H(c,p)$ in one of the following sets:
			\begin{itemize}
				\item the empty set,
				\item one class in $A$ or $D$,
				\item two classes, with both in $A$, both in $D$, or one in $B$ and the other in $C$.
			\end{itemize}
		    In the last case, the two classes of the intersection points are $(x,y)$ and $(-y,-x)$ for some $x,y\in \F_p^{*}$.
			\item A line of slope $-1$ meets $H(c,p)$ in one of the following sets:
			\begin{itemize}
				\item the empty set,
				\item one class in $B$ or $C$,
				\item two classes, with both in $B$, both in $C$, or one in $A$ and the other in $D$.
			\end{itemize}
			In the last case, the two classes of the intersection points are $(x,y)$ and $(y,x)$ for some $x,y\in \F_p^{*}$.
	\end{enumerate}
\end{observations}

In the cases when a line of slope $+1$ or $-1$ meets $H(c,p)$ in two classes, the base points of the two classes are reflections of each other in the lines $x+y=p$ or $x=y$, respectively. The proof makes use of the symmetry of the modular hyperbola with respect to switching the coordinates and to multiplying each coordinate by $-1$.

The construction of Hall et al. is a subset of the grid 
\begin{equation}
    \cG=\cG(p):=\left[-\frac{p-1}{2}, \frac{3p-1}{2}\right]\times [0,2p-1],
    \label{eq:G}
\end{equation}
 which has the following key property:

\begin{observation}\label{obs:region_property}
	In $\cG$, no three congruent points lie on the same line, and any pair of congruent points determine a line of slope $0$, $\infty$, or $\pm 1$.
\end{observation}
\begin{proof}
	The region contains $4$ points of each class, laid out in the vertices of a square, from which the statement follows.
\end{proof}

Now we are ready to recall the construction of Hall et al., and the proof of their main result, giving a no-$(k+1)$-in-line set for $k=2$:

\begin{construction}[HJSW construction \cite{hall1975some}]\label{hallcon} 
	Take the set $T_2=T_2(p)\subseteq \cG(p)$ defined as the union of the following $12$ blocks, three of each type:
	$$T_2 := \big(A_{0,1}\sqcup A_{1,0}\sqcup A_{1,1}\big) \sqcup 
    \big(B_{-1,0}\sqcup B_{-1,1}\sqcup B_{0,1}\big) \sqcup 
    \big(C_{0,0}\sqcup C_{1,0}\sqcup C_{1,1}\big) \sqcup 
    \big(D_{-1,0}\sqcup D_{-1,1}\sqcup D_{0,0}\big).$$
	Fix any value $c\in \F_p^{*}$, and take the modular hyperbola $H(c,p)$. The HJSW construction is the set \[S_2(c,p):=H(c,p)\cap T_2(p).\]
\end{construction}

The set $T_2(p)$ is illustrated in \autoref{fig:blocks_hjsw}.

\begin{theorem}[\cite{hall1975some}]\label{hjsw_th}
For any odd prime $p$ and $c\in \F_p^*$, the set $S_2(c,p)$ defined in 	\autoref{hallcon}  is a no-three-in-line set of size $3(p-1)$, hence
$f_2(2p)\ge 3(p-1)$.
\end{theorem}
\begin{proof}  As $H(c,p)$ altogether contains $p-1$ classes of normal points, and $T_2$ contains 3 points of every such class, $|S_2(c,p)|=3(p-1)$.
	
	We prove that $S_2:=S_2(c,p)\subseteq \cG$ is a no-three-in-line set.
	Suppose that line $\ell$ contains at least $3$ points of $S_2$. By  \autoref{obs:atmosttwoclasses}, $\ell\cap S_2$ contains at least two congruent points, so by  \autoref{obs:region_property}, $\ell$ has slope $0$, $\infty$, or $\pm 1$.
	
	If $\ell$ is horizontal or vertical, then it meets $H(c,p)$ in at most one class. But this is a contradiction, since $T_2$ has no three congruent points on a line.
	
	If $\ell$ has slope $\pm 1$, and it meets $H(c,p)$ in at most one class, then we similarly get a contradiction. So suppose that $\ell\cap H(c,p)$ has points from two classes, and distinguish two cases based on the slope of $\ell$.
	
	\begin{description}[style=nextline,leftmargin=1em]
		\item[Case 1: $\ell$ has slope $+1$.] In $T_2$, the only pairs of congruent points lying on a line of slope $+1$ consist of one point from $B_{-1,0}$ and one from $B_{0,1}$, or one point from $C_{0,0}$ and one from $C_{1,1}$. However, in these cases the line does not intersect $T_2\cap C$, or $T_2\cap B$, respectively, so by  \autoref{obs:meet_classes}(b), the other class in $\ell\cap H(c,p)$ does not contain a point of $T_2$, giving a contradiction.
		
		\item[Case 2: $\ell$ has slope $-1$.] In this case, we get an analogous contradiction using the block pairs $\{D_{-1,1},D_{0,0}\}$ and $\{A_{0,1},A_{1,0}\}$, considering  \autoref{obs:meet_classes}(c).
	\end{description}
	
	So $S_2$ is indeed a no-three-in-line set.
\end{proof}

\subsection{Extensions of the HJSW construction for \texorpdfstring{$k=3$}{k=3} and \texorpdfstring{$k=4$}{k=4}}
\label{sec:HJSW-34}

We now show a way to supplement the HJSW construction $S_2$ with additional points to obtain a no-$4$-in-line set $S_3$ of at least $\frac72(p-1)$ points, and a no-$5$-in-line set $S_4$ of $4(p-1)$ points. Although these sets turn out to be suboptimal, they will be useful building blocks later on.

To define this new construction, we introduce a new subdivision of $\zz\times \zz$, which resembles a chessboard rotated by $45^{\circ}$. Let $\zz\times \zz=E\sqcup F$, where

$$E=\left\{(x,y)\in\Z\times \Z: (p\mid y+x) \text{~or~} (p\mid y-x) \text{~or~} \left(\left\lfloor\frac{y+x}{p}\right\rfloor + \left\lfloor\frac{y-x}{p}\right\rfloor \equiv 0\pmod{2}\right)\right\},$$

$$F=\left\{(x,y)\in\Z\times \Z: (p\nmid y+x) \text{~and~} (p\nmid y-x) \text{~and~} \left(\left\lfloor\frac{y+x}{p}\right\rfloor + \left\lfloor\frac{y-x}{p}\right\rfloor \equiv 1\pmod{2}\right)\right\}.$$

 \autoref{fig:efsubdivision} shows this subdivision, where $E$ is shaded brown (and also includes the boundary lines with $p\mid y+x$ or $p\mid y-x$), while $F$ is shaded light blue. The subdivision is compatible with translations by multiples of $p$ along both axes, so every class lies fully in $E$ or in $F$.

The subdivision has the following symmetry property:
\begin{observation}\label{obs:efproperties}
For any $(x,y)\in \F_p\times \F_p$ with $x,~y,~y+x,~y-x\ne 0,$ the four points obtained by reflecting $(x,y)$ in some subset of the two lines $y+x=0$ and $y-x=0$ are categorized as follows: either $(x,y), (-x,-y)$ are in $E$ and $(y,x), (-y,-x)$ are in $F$, or vice versa. See the dark blue points of  \autoref{fig:efsubdivision} for an example.
\end{observation}

\begin{figure}[htb!]
	\centering
	
	\definecolor{ffqqqq}{rgb}{1.,0.,0.}
\definecolor{qqqqff}{rgb}{0.,0.,1.}
\definecolor{ffqqff}{rgb}{1.,0.,1.}
\definecolor{qqffqq}{rgb}{0.,1.,0.}
\definecolor{lightgggrey}{rgb}{0.7,0.7,0.7}
\definecolor{ecolor}{rgb}{0.8,0.8,0.8}
\definecolor{fcolor}{rgb}{0.3,0.3,0.3}
\resizebox{0.5\textwidth}{!}{
\begin{tikzpicture}[line cap=round,line join=round,>=triangle 45,x=0.25cm,y=0.25cm]
\draw [step=1,gray,line width=0.3pt, opacity=0.5] (-5.5,0) grid (16.5,22);

\fill [SkyBlue, fill opacity=0.2] (5.5,5.5)--(11,11)--(5.5,16.5)--(0,11)--cycle;
\fill [SkyBlue, fill opacity=0.2] (0,0)--(11,0)--(5.5,5.5)--cycle;
\fill [SkyBlue, fill opacity=0.2] (0,22)--(11,22)--(5.5,16.5)--cycle;
\fill [SkyBlue, fill opacity=0.2] (-5.5,5.5)--(-5.5,16.5)--(0,11)--cycle;
\fill [SkyBlue, fill opacity=0.2] (11,11)--(16.5,16.5)--(16.5,5.5)--cycle;
\fill [SkyBlue, fill opacity=0.2] (11,0)--(16.5,0)--(16.5,5.5)--cycle;
\fill [SkyBlue, fill opacity=0.2] (11,22)--(16.5,22)--(16.5,16.5)--cycle;
\fill [SkyBlue, fill opacity=0.2] (-5.5,22)--(0,22)--(-5.5,16.5)--cycle;
\fill [SkyBlue, fill opacity=0.2] (-5.5,0)--(0,0)--(-5.5,5.5)--cycle;

\draw [line width=1.5pt] (0,0)--(0,22);
\draw [line width=1.5pt] (11,0)--(11,22);
\draw [line width=1.5pt] (-5.5,0)--(16.5,0);
\draw [line width=1.5pt] (-5.5,11)--(16.5,11);
\draw [line width=1.5pt] (-5.5,22)--(16.5,22);

\draw [line width=0.9pt] (0,0)--(16.5,16.5);
\draw [line width=0.9pt] (11,0)--(16.5,5.5);
\draw [line width=0.9pt] (-5.5,5.5)--(11,22);
\draw [line width=0.9pt] (-5.5,16.5)--(0,22);

\draw [line width=0.9pt] (-5.5,16.5)--(11,0);
\draw [line width=0.9pt] (0,22)--(16.5,5.5);
\draw [line width=0.9pt] (11,22)--(16.5,16.5);
\draw [line width=0.9pt] (-5.5,5.5)--(0,0);

\draw [line width=0.9pt] (-5.5,16.5)--(0,22);

\draw [fill=brown, color=brown, line width=1pt, fill opacity=0.3] (0,0)--(5.5,5.5)--(0,11)--(-5.5,5.5)--cycle;
\draw [fill=brown, color=brown, line width=1pt, fill opacity=0.3] (11,0)--(16.5,5.5)--(11,11)--(5.5,5.5)--cycle;
\draw [fill=brown, color=brown, line width=1pt, fill opacity=0.3] (0,11)--(5.5,16.5)--(0,22)--(-5.5,16.5)--cycle;
\draw [fill=brown, color=brown, line width=1pt, fill opacity=0.3] (11,11)--(16.5,16.5)--(11,22)--(5.5,16.5)--cycle;

\draw [fill=blue, line width=0pt] (1,8) circle (0.07cm);
\draw [fill=blue, line width=0pt] (8,1) circle (0.07cm);
\draw [fill=blue, line width=0pt] (3,10) circle (0.07cm);
\draw [fill=blue, line width=0pt] (10,3) circle (0.07cm);

\draw [fill=blue, line width=0pt] (12,8) circle (0.07cm);
\draw [fill=blue, line width=0pt] (14,10) circle (0.07cm);

\draw [fill=blue, line width=0pt] (-3,1) circle (0.07cm);
\draw [fill=blue, line width=0pt] (-1,3) circle (0.07cm);

\draw [fill=blue, line width=0pt] (1,19) circle (0.07cm);
\draw [fill=blue, line width=0pt] (8,12) circle (0.07cm);
\draw [fill=blue, line width=0pt] (3,21) circle (0.07cm);
\draw [fill=blue, line width=0pt] (10,14) circle (0.07cm);

\draw [fill=blue, line width=0pt] (12,19) circle (0.07cm);
\draw [fill=blue, line width=0pt] (14,21) circle (0.07cm);

\draw [fill=blue, line width=0pt] (-3,12) circle (0.07cm);
\draw [fill=blue, line width=0pt] (-1,14) circle (0.07cm);

\begin{scriptsize}
\draw (0,-0.78) node[anchor=mid] {$0$};
\draw (11,-0.78) node[anchor=mid] {$p$};

\draw (-6.5,0) node[anchor=mid] {$0$};
\draw (-6.5,11) node[anchor=mid] {$p$};
\draw (-6.5,22) node[anchor=mid] {$2p$};
\end{scriptsize}
\end{tikzpicture}
}
    
    \vspace{-0.4cm}
    \caption{The sets $E$ (brown) and $F$ (light blue). The points of a set of four classes illustrating  \autoref{obs:efproperties} are marked as dark blue.}
    \label{fig:efsubdivision}
\end{figure}

Recall the HJSW construction $S_2=H(c,p)\cap T_2$ for a modular hyperbola $H(c,p)$, which was no-$3$-in-line. We are now ready to define our modified constructions $S_k$ for $k\in \{3,4\}$, which will also be no-$(k+1)$-in-line.

In the grid $\cG$, the normal points form $16$ blocks, with $12$ of these forming the set $T_2$. We let $M$ be the union of the four remaining blocks in the middle: $M=A_{0,0}\sqcup B_{0,0}\sqcup C_{0,1}\sqcup D_{0,1}$.

\begin{construction}[Modular hyperbola construction in the grid]\label{const:S3S4}
	In the grid $\cG$ from \eqref{eq:G}, let $T_3\leteq T_2 \sqcup \left(M\cap E\right)$ and $T_4\leteq T_2 \sqcup M$. 
    For any $c\in \F_p^{*}$, define  
    \begin{align*}
        S_3&\leteq S_3(c,p)\leteq H(c,p)\cap T_3,\\
        S\leteq S(c,p)\leteq S_4&\leteq S_4(c,p)\leteq H(c,p)\cap T_4, 
	\end{align*}
 where the modular hyperbola $H(c,p)\subset \Z^2$ is given in \autoref{def:modularHperbola}.
\end{construction}

Note that the subscripts $2$, $3$ and $4$ indicate the maximal intersection sizes of the respective sets with respect to lines. In the calculations of \autoref{sec:constructionGeneral} and \autoref{sec:E}, we will use $S$ instead of $S_4$ for simplicity.

In  \autoref{fig:t4t3t2}, the sets $T_2\subseteq T_3\subseteq T_4$ are depicted simultaneously.

\begin{figure}[htb!]
\centering
    \begin{subfigure}[t]{0.33\textwidth}
    	\definecolor{ffqqqq}{rgb}{1.,0.,0.}
\definecolor{qqqqff}{rgb}{0.,0.,1.}
\definecolor{ffqqff}{rgb}{1.,0.,1.}
\definecolor{qqffqq}{rgb}{0.,1.,0.}
\definecolor{lightgggrey}{rgb}{0.7,0.7,0.7}
\resizebox{\textwidth}{!}{
\begin{tikzpicture}[line cap=round,line join=round,>=triangle 45,x=0.25cm,y=0.25cm]
\draw [step=1,gray,line width=0.3pt, opacity=0.5] (-5,0) grid (16,21);

\draw [line width=1.5pt] (0,0)--(0,21);
\draw [line width=1.5pt] (11,0)--(11,21);
\draw [line width=1.5pt] (-5,0)--(16,0);
\draw [line width=1.5pt] (-5,11)--(16,11);

\draw [fill=blue, color=blue, line width=1pt, fill opacity=0.3] (-5,1)--(-5,5)--(-1,5)--(-1,1)--cycle;
\draw [fill=yellow, color=yellow, line width=1pt, fill opacity=0.3] (-5,6)--(-5,10)--(-1,10)--(-1,6)--cycle;
\draw [fill=blue, color=blue, line width=1pt, fill opacity=0.3] (-5,12)--(-5,16)--(-1,16)--(-1,12)--cycle;
\draw [fill=yellow, color=yellow, line width=1pt, fill opacity=0.3] (-5,17)--(-5,21)--(-1,21)--(-1,17)--cycle;

\draw [fill=green, color=green, line width=1pt, fill opacity=0.3] (5,1)--(5,5)--(1,5)--(1,1)--cycle;
\draw [fill=red, color=red, line width=1pt, fill opacity=0.3] (5,6)--(5,10)--(1,10)--(1,6)--cycle;
\draw [fill=green, color=green, line width=1pt, fill opacity=0.3] (5,12)--(5,16)--(1,16)--(1,12)--cycle;
\draw [fill=red, color=red, line width=1pt, fill opacity=0.3] (5,17)--(5,21)--(1,21)--(1,17)--cycle;

\draw [fill=blue, color=blue, line width=1pt, fill opacity=0.3] (6,1)--(6,5)--(10,5)--(10,1)--cycle;
\draw [fill=yellow, color=yellow, line width=1pt, fill opacity=0.3] (6,6)--(6,10)--(10,10)--(10,6)--cycle;
\draw [fill=blue, color=blue, line width=1pt, fill opacity=0.3] (6,12)--(6,16)--(10,16)--(10,12)--cycle;
\draw [fill=yellow, color=yellow, line width=1pt, fill opacity=0.3] (6,17)--(6,21)--(10,21)--(10,17)--cycle;

\draw [fill=green, color=green, line width=1pt, fill opacity=0.3] (12,1)--(12,5)--(16,5)--(16,1)--cycle;
\draw [fill=red, color=red, line width=1pt, fill opacity=0.3] (12,6)--(12,10)--(16,10)--(16,6)--cycle;
\draw [fill=green, color=green, line width=1pt, fill opacity=0.3] (12,12)--(12,16)--(16,16)--(16,12)--cycle;
\draw [fill=red, color=red, line width=1pt, fill opacity=0.3] (12,17)--(12,21)--(16,21)--(16,17)--cycle;

\begin{scriptsize}

\foreach \x in {-5,-1,1,5,6,10,12,16}
{
\draw (\x,-0.78) node[anchor=mid] {$\x$};
}
\foreach \y in {0,1,5,6,10,11,12,16,17,21}
{
\draw (-6.2,\y) node[anchor=mid] {$\y$};
}
\end{scriptsize}

\draw (3,8) node[anchor=center] {$A_{0,0}$};
\draw (14,8) node[anchor=center] {$A_{1,0}$};
\draw (3,19) node[anchor=center] {$A_{0,1}$};
\draw (14,19) node[anchor=center] {$A_{1,1}$};

\draw (-3,8) node[anchor=center] {$B_{-1,0}$};
\draw (8,8) node[anchor=center] {$B_{0,0}$};
\draw (-3,19) node[anchor=center] {$B_{-1,1}$};
\draw (8,19) node[anchor=center] {$B_{0,1}$};

\draw (3,3) node[anchor=center] {$C_{0,0}$};
\draw (14,3) node[anchor=center] {$C_{1,0}$};
\draw (3,14) node[anchor=center] {$C_{0,1}$};
\draw (14,14) node[anchor=center] {$C_{1,1}$};

\draw (-3,3) node[anchor=center] {$D_{-1,0}$};
\draw (8,3) node[anchor=center] {$D_{0,0}$};
\draw (-3,14) node[anchor=center] {$D_{-1,1}$};
\draw (8,14) node[anchor=center] {$D_{0,1}$};

\end{tikzpicture}
}
        \caption*{$T_4$}
    \end{subfigure}%
    \begin{subfigure}[t]{0.33\textwidth}
    	\definecolor{ffqqqq}{rgb}{1.,0.,0.}
\definecolor{qqqqff}{rgb}{0.,0.,1.}
\definecolor{ffqqff}{rgb}{1.,0.,1.}
\definecolor{qqffqq}{rgb}{0.,1.,0.}
\definecolor{lightgggrey}{rgb}{0.7,0.7,0.7}
\resizebox{\textwidth}{!}{
\begin{tikzpicture}[line cap=round,line join=round,>=triangle 45,x=0.25cm,y=0.25cm]

\draw [step=1,gray,line width=0.3pt, opacity=0.5] (-5,0) grid (16,21);

\draw [line width=1.5pt] (0,0)--(0,21);
\draw [line width=1.5pt] (11,0)--(11,21);
\draw [line width=1.5pt] (-5,0)--(16,0);
\draw [line width=1.5pt] (-5,11)--(16,11);

\draw [fill=blue, color=blue, line width=1pt, fill opacity=0.3] (-5,1)--(-5,5)--(-1,5)--(-1,1)--cycle;
\draw [fill=yellow, color=yellow, line width=1pt, fill opacity=0.3] (-5,6)--(-5,10)--(-1,10)--(-1,6)--cycle;
\draw [fill=blue, color=blue, line width=1pt, fill opacity=0.3] (-5,12)--(-5,16)--(-1,16)--(-1,12)--cycle;
\draw [fill=yellow, color=yellow, line width=1pt, fill opacity=0.3] (-5,17)--(-5,21)--(-1,21)--(-1,17)--cycle;

\draw [fill=green, color=green, line width=1pt, fill opacity=0.3] (5,1)--(5,5)--(1,5)--(1,1)--cycle;
\draw [fill=red, color=red, line width=1pt, fill opacity=0.3] (5,6)--(1,10)--(1,6)--cycle;
\draw [fill=green, color=green, line width=1pt, fill opacity=0.3] (5,16)--(1,16)--(1,12)--cycle;
\draw [fill=red, color=red, line width=1pt, fill opacity=0.3] (5,17)--(5,21)--(1,21)--(1,17)--cycle;

\draw [fill=blue, color=blue, line width=1pt, fill opacity=0.3] (6,1)--(6,5)--(10,5)--(10,1)--cycle;
\draw [fill=yellow, color=yellow, line width=1pt, fill opacity=0.3] (6,6)--(10,10)--(10,6)--cycle;
\draw [fill=blue, color=blue, line width=1pt, fill opacity=0.3] (6,16)--(10,16)--(10,12)--cycle;
\draw [fill=yellow, color=yellow, line width=1pt, fill opacity=0.3] (6,17)--(6,21)--(10,21)--(10,17)--cycle;

\draw [fill=green, color=green, line width=1pt, fill opacity=0.3] (12,1)--(12,5)--(16,5)--(16,1)--cycle;
\draw [fill=red, color=red, line width=1pt, fill opacity=0.3] (12,6)--(12,10)--(16,10)--(16,6)--cycle;
\draw [fill=green, color=green, line width=1pt, fill opacity=0.3] (12,12)--(12,16)--(16,16)--(16,12)--cycle;
\draw [fill=red, color=red, line width=1pt, fill opacity=0.3] (12,17)--(12,21)--(16,21)--(16,17)--cycle;

\begin{scriptsize}

\foreach \x in {-5,-1,1,5,6,10,12,16}
{
\draw (\x,-0.78) node[anchor=mid] {$\x$};
}
\foreach \y in {0,1,5,6,10,11,12,16,17,21}
{
\draw (-6.2,\y) node[anchor=mid] {$\y$};
}
\end{scriptsize}





\end{tikzpicture}
}
        \caption*{$T_3$}
    \end{subfigure}%
    \begin{subfigure}[t]{0.33\textwidth}
    	\definecolor{ffqqqq}{rgb}{1.,0.,0.}
\definecolor{qqqqff}{rgb}{0.,0.,1.}
\definecolor{ffqqff}{rgb}{1.,0.,1.}
\definecolor{qqffqq}{rgb}{0.,1.,0.}
\definecolor{lightgggrey}{rgb}{0.7,0.7,0.7}
\resizebox{\textwidth}{!}{
\begin{tikzpicture}[line cap=round,line join=round,>=triangle 45,x=0.25cm,y=0.25cm]

\draw [step=1,gray,line width=0.3pt, opacity=0.5] (-5,0) grid (16,21);

\draw [line width=1.5pt] (0,0)--(0,21);
\draw [line width=1.5pt] (11,0)--(11,21);
\draw [line width=1.5pt] (-5,0)--(16,0);
\draw [line width=1.5pt] (-5,11)--(16,11);

\draw [fill=blue, color=blue, line width=1pt, fill opacity=0.3] (-5,1)--(-5,5)--(-1,5)--(-1,1)--cycle;
\draw [fill=yellow, color=yellow, line width=1pt, fill opacity=0.3] (-5,6)--(-5,10)--(-1,10)--(-1,6)--cycle;
\draw [fill=blue, color=blue, line width=1pt, fill opacity=0.3] (-5,12)--(-5,16)--(-1,16)--(-1,12)--cycle;
\draw [fill=yellow, color=yellow, line width=1pt, fill opacity=0.3] (-5,17)--(-5,21)--(-1,21)--(-1,17)--cycle;

\draw [fill=green, color=green, line width=1pt, fill opacity=0.3] (5,1)--(5,5)--(1,5)--(1,1)--cycle;
\draw [fill=red, color=red, line width=1pt, fill opacity=0.3] (5,17)--(5,21)--(1,21)--(1,17)--cycle;

\draw [fill=blue, color=blue, line width=1pt, fill opacity=0.3] (6,1)--(6,5)--(10,5)--(10,1)--cycle;
\draw [fill=yellow, color=yellow, line width=1pt, fill opacity=0.3] (6,17)--(6,21)--(10,21)--(10,17)--cycle;

\draw [fill=green, color=green, line width=1pt, fill opacity=0.3] (12,1)--(12,5)--(16,5)--(16,1)--cycle;
\draw [fill=red, color=red, line width=1pt, fill opacity=0.3] (12,6)--(12,10)--(16,10)--(16,6)--cycle;
\draw [fill=green, color=green, line width=1pt, fill opacity=0.3] (12,12)--(12,16)--(16,16)--(16,12)--cycle;
\draw [fill=red, color=red, line width=1pt, fill opacity=0.3] (12,17)--(12,21)--(16,21)--(16,17)--cycle;

\begin{scriptsize}

\foreach \x in {-5,-1,1,5,6,10,12,16}
{
\draw (\x,-0.78) node[anchor=mid] {$\x$};
}
\foreach \y in {0,1,5,6,10,11,12,16,17,21}
{
\draw (-6.2,\y) node[anchor=mid] {$\y$};
}
\end{scriptsize}





\end{tikzpicture}
}
        \caption*{$T_2$}
    \end{subfigure}
\caption{The subsets $T_4$, $T_3$ and $T_2$ of $\cG$ in the case $p=11$, illustrated as the union of the four coloured subsets in each figure. The colours red, yellow, green and blue indicate the points of $A$, $B$, $C$ and $D$, respectively.}
\label{fig:t4t3t2}
\end{figure}

\begin{proposition}\label{thm:pversion_34}
	$f_3(2p)\ge |S_3|\ge 3.5(p-1)$ and $f_4(2p)\ge |S_4|= 4(p-1)$ for every prime $p\ge 3$.
\end{proposition}
\begin{proof}
	We prove that $S_k$ contains at most $k$ points on every line $\ell$ for $k\in \{3,4\}$. Clearly, $S_2\subseteq S_3\subseteq S_4$.
 
	Take any line $\ell$; the points of $\ell\cap S_k$ lie in at most two classes. By  \autoref{obs:region_property}, if any two points in $\cG$ are congruent then they determine a line of slope $0$, $\infty$ or $\pm 1$, and on this line there cannot be any more points of $\cG$ from the same class. So $|\ell\cap S_4|\le 4$, and if $\ell$ is of modulus $\ge 2$, then $|\ell\cap S_4|\le 2$.
 
	Horizontal and vertical lines only meet one class of $H(c,p)$, so they meet $S_4$ in at most $2$ points. So it remains to show that $|\ell\cap S_3|\le 3$ for $\ell$ of slope $\pm 1$.
	
	We already know that $|\ell\cap S_2|\le 2$, and now we show that $|\ell \cap (S_3\setminus S_2)|=|\ell\cap (M\cap E)\cap H(c,p)|\le 1$. Supposing the contrary, $\ell$ must meet $(M\cap E)\cap H(c,p)$ in two noncongruent points, as no two points of $M$ are congruent.
	
	If $\ell$ has slope $+1$, then considering the possibilities in  \autoref{obs:meet_classes}(b) for the two classes of $\ell\cap H(c,p)$, we cannot have one class in $B$ and the other in $C$, as such points are not on the same slope $+1$ line in $M$. However we also cannot have two classes in $A$ (or two in $D$), since the corresponding points would be reflections of each other in $x+y=p$ (or $x+y=2p$), but one point from every such pair in $A_{0,0}$ and $D_{0,1}$ is missing from $E$.
	
	If $\ell$ has slope $-1$, we get an analogous contradiction, giving that $S_3$ is indeed no-$4$-in-line.
	
	Now as $H(c,p)$ consists of $p-1$ full normal classes, and $T_4$ contains $4$ points of every normal class, $|S_4|=4(p-1)$. The set $T_3$ consists of $T_2$ (having $3$ points of each normal class), and $M\cap E$ (where $M$ has $1$ point of each normal class). By  \autoref{obs:efproperties}, normal classes $(x,y)$ that do not lie on the diagonals $y\pm x\equiv 0 \pmod p$ can be subdivided into groups of two classes in $E$ and two in $F$, with the value of $xy$ being invariant modulo $p$ within each group.
	
	Altogether $p-1$ classes satisfy $xy\equiv c \pmod p$, and on the diagonals $y-x\equiv 0 \pmod p$ and $y+x\equiv 0 \pmod p$, there are at most two such classes each. Half of the remaining classes are in $E$, thus $|S_3\setminus S_2|\ge \frac{p-1}{2}$  and hence $|S_3|\ge 3.5(p-1)$ holds.
\end{proof}
\section{Randomised algebraic construction for every \texorpdfstring{$k$}{k}}\label{sec:constructionGeneral}

In this section, we prove \autoref{k_asymp} by giving a random point set of the grid, for which we select a set of random modular hyperbolae, and then delete some points from the union of their point sets lying in an $n\times n$ grid in order to obtain our no-$(k+1)$-in-line set.

\begin{definition}\label{def:G-Omega-L}
    Let $p$ be an odd prime. 
    In the $2p\times 2p$ grid $\cG=\cG(p)\leteq \left[-\frac{p-1}{2}, \frac{3p-1}{2}\right] \times [0,2p-1]$, 
    consider the $2p\times (2p-1)$ subgrid $\cG'\leteq \left[-\frac{p-1}{2}, \frac{3p-1}{2}\right] \times [1,2p-1]$ which we get from $\cG$ by removing the bottom row. 
    Let $\cL'$ be the set of lines of slope $\pm 1$ intersecting $\cG'$.
    Define $\Omega_p$ to be the set of subsets of $\{1,\dots,p-1\}$.
   
\end{definition}

In the following, we build on \autoref{const:S3S4}. In particular, recall that $S(c,p)$ is the intersection of the modular hyperbola $H(c,p)$ with the grid $\cG(p)$.

\begin{construction}[Randomised construction using multiple modular hyperbolae]\label{con:randomConstruction} 
    Let $W\in \Omega_p$ and $k\in\N$ with $k\geq 2\Cardinality{W}$. 
    Let $\cS(W)\leteq \bigcup_{c\in W} S(c,p) \subseteq \cG'$, 
    and pick a subset $\cR\subseteq \cS(W)$ of minimum size such that  
    $|\ell\cap \cR|\geq |\ell\cap \cS(W)|-k$ 
    for every line $\ell\in\cL'$. 
    Our construction is   $\cP_k(W)\leteq \cS(W)\setminus\cR$.
\end{construction}

The remaining part of the section aims to determine the size of the construction above.

\begin{lemma}[Construction size]
\label{lem:construction} The trimmed point set
    $\cP_k(W)\subseteq \cG$ from \autoref{con:randomConstruction} is a no-$(k+1)$-in-line set of size 
    \begin{equation}
        \Cardinality{\cP_k(W)}\geq 4(p-1)\Cardinality{W}-X_{k,p}(W),
        \label{eq:constructionSize}
    \end{equation}
    where 
    $X_{k,p}(W)\leteq \sum_{\ell\in\cL'} \max\{0,\Cardinality{\ell\cap \cS(W)}-k\}$.
\end{lemma}

\begin{proof}
    First we prove that $\Cardinality{\ell\cap \cP_k(W)}\leq k$ for every line $\ell$ intersecting $\cG'$. Remember that modulo $p$ congruence classes of $\zz\times \zz$ are called \textit{classes} for short. We distinguish three cases depending on the slope of $\ell$.
    
    \begin{description}[style=nextline,leftmargin=1em]
    	\item[Case 1: $\ell$ has slope $0$ or $\infty$.] By \autoref{obs:meet_classes}(a), the line $\ell$ meets $H(c,p)$ in at most one class for each $c\in W$. By \autoref{obs:region_property}, $\ell\cap \cG$ contains at most two points from this class, so $\ell\cap S(c,p)=\ell\cap H(c,p)\cap \cG$ has size at most $2$. Altogether,
    	$$|\ell\cap \cP_k(W)|\le \sum_{c\in W} |\ell\cap S(c,p)|\le 2|W|\le k.$$
    	
    	\item[Case 2: $\ell$ has slope $\pm 1$.] In this case, $\ell\in \cL'$, so $|\ell\cap \cP_k(W)|=|\ell\cap \cS(W)|-|\ell\cap \cR|\le k$ by construction.
    	
    	\item[Case 3: $\ell$ has slope different from $0$, $\infty$ and $\pm 1$.] By \autoref{obs:atmosttwoclasses}, the line $\ell$ meets $H(c,p)$ in at most two classes for each $c\in W$. By the application of \autoref{obs:region_property} to the current case, $\ell\cap\cG$ contains at most one point from each class, so $\ell\cap S(c,p)=\ell\cap H(c,p)\cap \cG$ has size at most $2$. The conclusion is the same as in Case 1.
    \end{description}

    For the bound on $\Cardinality{\cP_k(W)}$, note that $\Cardinality{\cS(W)}=4(p-1)\Cardinality{W}$ by \autoref{thm:pversion_34}. 
    On the other hand, for every $\ell\in\cL'$, pick a subset $R_\ell\subseteq \ell\cap \cS(W)$ of size $\max\{0, \Cardinality{\ell\cap \cS(W)}-k\}$. 
    Then 
    $$\Cardinality{\cR}
    \leq \sum_{\ell\in \cL'} |R_{\ell}| 
    = X_{k,p}(W)$$ by definition.
\end{proof}

\begin{remark}\label{improve}
    Note that $\frac{1}{2} X_{k,p}(W)\leq \Cardinality{\cR}\leq X_{k,p}(W)$ holds for the erased point set.
    Indeed,  in the proof of \autoref{lem:construction}, we must remove the necessary excess points $R_\ell$ from $\cS(W)$ along the disjoint lines $\ell\in \cL'$ of slope $+1$. Considering the symmetries of slope $+1$ and slope $-1$ lines (cf. \autoref{lem:M4j}), this gives the lower bound on $|\cR|$. The upper bound follows from \autoref{con:randomConstruction}.
   In order to improve the upper bound  $\sum_{\ell\in \cL'} |R_{\ell}| 
    = X_{k,p}(W)$, one might try to maximize (or bound from below) the number of excess points that appear as excess points on lines of different slopes. This approach could also add a marginal numerical improvement, but we decide not to elaborate on it for general values of $k$.
       
    Note however, that when  $k=2$, the underlying number theory enables finding this optimum, see \autoref{sec:HJSW-2} and \autoref{sec:HJSW-34}. We will use similar ideas in \autoref{sec:smallK} to gain some improvement for small constant values of $k$. 
    In the general treatment of the current section, we will simply work with the above upper bound on $|\cR|$.
\end{remark}

\begin{definition}\label{def:ProbabilitySpace}
    For $\Omega_p$ from \autoref{def:G-Omega-L}, define $\cF$ to be the set of subsets of $\Omega_p$, and a probability measure $\P\from \cF\to \R, A\mapsto \frac{|A|}{|\Omega_p|}$ for all $A\in\cF$. 
    In particular, $\P(\{W\})=\frac{1}{|\Omega_p|}$ for every elementary event given by $W\in\Omega_p$.
    
    Now $X_{k,p}$ from \autoref{lem:construction} can be considered a random variable on the uniform random discrete probability space $(\Omega_p, \cF, \P)$.
    For $t\in\N$, define the event $A_t\leteq \{W\in \Omega_p: \Cardinality{W}=t\}\in \cF$. 
\end{definition}

One wishes to choose $W\in\Omega_p$ with $|W|\le k/2$ so that \eqref{eq:constructionSize} gives the largest lower bound for the construction size. 
Unfortunately, finding this optimal $W$ exactly is hard. 
So instead, we use the standard averaging argument 
\begin{equation}
    \min\{X_{k,p}(W):W\in A_t\}\leq \EE(X_{k,p}\mid A_t)=\frac{1}{|A_t|} \sum_{W\in A_t} X_{k,p}(W),
    \label{eq:min<E}
\end{equation}
which together with \eqref{eq:constructionSize} gives 
\begin{equation}
    4(p-1) t -\EE(X_{k,p}\mid A_t)\leq f_k(2p).
    \label{eq:boundWithE}
\end{equation}
In \autoref{rem:bestTS}, we will see that it is the best to take $W$ to be as large as possible, i.e., $t=\Cardinality{W}=\lfloor k/2\rfloor$.

We need the following technical statements to evaluate the conditional expected value of \eqref{eq:boundWithE}. 
Their proof is computationally heavy, so it is separated to \autoref{sec:E} enabling a better overview of the construction in the current section.

\begin{definition}\label{def:F}
    For $t,s\in \N$, define the positive number 
    \begin{equation}\label{eq:Fdef}
        F_{t,s}\leteq  
        \frac{1}{2^{t-1}(t+1)}
    \left(
        \sum_{r=0}^{b} 
        \binom{2t-s+1-4r}{2} \binom{t+1}{r} 
        + \sum_{r=0}^{a} (t-s-3r) \binom{t+1}{r}
    \right)
    \end{equation}
    where  $a\leteq \left\lfloor\frac{t-s}{3}\right\rfloor\leq \left\lfloor\frac{2t-s}{4}\right\rfloor\eqlet b$. 
\end{definition}

The expected value $\EE(X_{k,p} \mid A_t)$ of the number of deleted points of \autoref{con:randomConstruction} is given by the next statement. 
\begin{proposition}[Asymptotic value]\label{prop:key}
    For every $k,t,s\in \N$ with $k=2t+s$, 
     we have 
    \begin{equation}
        \EE(X_{k,p} \mid A_t)=F_{t,s}\cdot p+\bigO(1)
        \label{eq:EAsymptotics}
    \end{equation} as the odd prime $p$ tends to $\infty$, where $F_{t,s}$ is given by \autoref{def:F}.
\end{proposition}

\begin{proposition}[Upper bound]\label{prop:EUpperbound}
    For every $t,s\in \N$ with $s\geq 1$, the following inequalities hold for $F_{t,s}$, defined in \eqref{eq:Fdef}.
    \begin{equation}
    \begin{tikzcd}[row sep=tiny, column sep=scriptsize]
        & F_{t,0} \ar[r, phantom, "<"]
        & 4 - \frac{12}{\sqrt{2\pi t + 14}} + \frac{2}{t+1} \ar[r, phantom, "<"]
        & 4 
        \\
        F_{t,s} \ar[r, phantom, "\leq"]
        & F_{t,1} \ar[r, phantom, "<"] \ar[u, phantom, sloped, "\leq"]
        & 4 - \frac{20}{\sqrt{\pi (2t+1) + 7}} + \frac{6}{t+1}  \ar[u, phantom, sloped, "<"]
    \end{tikzcd}
    \label{eq:EUpperBound}
    \end{equation}
    In particular, $\sup\{F_{t,s}:t,s\in\N\}=4$. 
\end{proposition}

\begin{theorem}\label{thm:asymptotics}
    Let $k\in \N$. 
    Pick $t,s\in \N$ with $k=2t+s$. 
    Then as $n\to \infty$, we have  
    \begin{equation}
        kn - \big(s+ \tfrac{1}{2}F_{t,s}+\littleO(1)\big)n\leq f_k(n).
        \label{eq:bestBound}
    \end{equation} 
\end{theorem}

\begin{proof}
    Pick $\eps>0$. 
    \autoref{prop:key} shows that there exists $N_E\in \N$, 
    such that for every odd prime $p\geq N_E$, 
    we have $\EE(X_{k,p}\mid A_t)\leq (F_{t,s}+\eps)p$.
    Pick $0<\delta \neq 1$ such that $2t\delta<\frac{\eps}{4}$. 
    The Prime number theorem implies the existence of $N_P$ 
    such that whenever $n\geq N_P$, then there is an odd prime number $p$ satisfying $(1-\delta)n\leq 2p \leq n$. 
    Define the threshold $N_\eps\leteq \max\left\{N_P, \frac{2N_P}{1-\delta}, \frac{16t}{\eps}\right\}$.

    Pick an arbitrary integer $n\geq N_\eps$. 
    Since $n\geq N_P$ by the definition of $N_\eps$, there exists an odd prime $p$ such that 
    $(1-\delta)n\leq 2p \leq n$.
    Let $A_t\subseteq \Omega_p$ be the event from \autoref{def:ProbabilitySpace} for this prime $p$.
    Now, we have 
    \begin{align*}
        f_k(n) &\geq f_k(2p) && \text{using $2p\leq n$,}
        \\&\geq 4pt-4t-\EE(X_{k,p}\mid A_t) && \text{using \eqref{eq:boundWithE},}
        \\&\geq 4pt-4t-(F_{t,s}+\eps)p && \text{using \eqref{eq:EAsymptotics},}
        \\&\geq 2nt(1-\delta)-4t-n\left(\tfrac{1}{2}F_{t,s}+\tfrac{\eps}{2}\right) && \text{using $(1-\delta)n\leq 2p \leq n$,}
        \\&= n(2t-\tfrac{1}{2}F_{t,s}-\eps)+n\left(\tfrac{\eps}{2}-2t\delta\right)-4t 
        \\&\geq  n(2t-\tfrac{1}{2}F_{t,s}-\eps)+\tfrac{\eps}{4}n -4t && \text{using $2t\delta<\tfrac{\eps}{4}$,}
        \\&\geq  n(2t-\tfrac{1}{2}F_{t,s}-\eps) && \text{using $n\geq N_\eps\geq \tfrac{16t}{\eps}$,}
        \\&= kn- (s+\tfrac{1}{2}F_{t,s}+\eps)n  && \text{using $k=2t+s$}
    \end{align*}
    giving \eqref{eq:bestBound} as required.
\end{proof}

\begin{remark}\label{rem:bestTS}
    Note that $F_{t,s}<4$ from \autoref{prop:EUpperbound} implies that \autoref{thm:asymptotics} gives the following sharpest result when applied to $t\leteq \lfloor k/2\rfloor$ and $s\in\{0,1\}$.    
\end{remark}

\begin{corollary}\label{cor:asym}
    For every $k\geq 2$, as $n\to\infty$, we have 
    \begin{align*}
    \left(1-\tfrac{1}{2k}F_{k/2,0} -\littleO(1)\right)kn &\leq f_k(n)
        \when 2\mid k, \\
    \left(1-\tfrac{1}{k}-\tfrac{1}{2k}F_{(k-1)/2,1}-\littleO(1)\right)kn &\leq f_k(n)
        \when 2\nmid k.
    \end{align*}
\end{corollary}

As an immediate application, we obtain one of the main statements of this paper.
\begin{proof}[Proof of \autoref{k_asymp}]
    Apply \autoref{cor:asym} and use the upper bound of \autoref{prop:EUpperbound} to get 
    \begin{align*}
        kn-2n<\left(1-\frac{2}{k} + \frac{6}{k\sqrt{\pi k + 14}} - \frac{2}{k(k+2)}\right) kn &\leq f_k(n) \when 2\mid k, \\
        kn-3n<\left(1-\frac{3}{k} + \frac{10}{k\sqrt{\pi k + 7}} - \frac{6}{k(k+1)}\right)kn &\leq f_k(n) \when 2\nmid k.
    \end{align*}
    as stated. We have taken into account that we have strict inequality in \autoref{prop:EUpperbound} in order to get rid of the error term $o(1)$.
\end{proof}

\begin{remark}\label{rem:explicitVsStirling}
When $k$ is a concrete number, $F_{t,s}$ can be computed precisely using \eqref{eq:Fdef}, so instead of using an upper bound on the function $F_{t,s}$ using Stirling's approximation as in \autoref{k_asymp}, we can get a sharper bound of form $$(C_k-\littleO(1))kn\leq f_k(n)$$ straight from \autoref{cor:asym}.
We summarize the results of this section for small values of $k$ in \autoref{tab:lowerBoundsGeneral}, which also demonstrates the 
validity of the estimate in \autoref{k_asymp}.
These include the improved bounds of \autoref{k_small}, except for the case when $k\in\{3,6\}$, where even better results are obtained in the following sections.

\begin{table}[!htb]
    \centering
    \begin{subtable}{.5\textwidth}
        \centering
        \begin{tabular}{C||R@{${}\approx{}$}LCC}
        k & \multicolumn{2}{C}{C_k} & S_k & C_k-S_k \\ \hline\hline
 2 & \frac{1}{2} & 0.50000 & 0.41612 & 0.08390 \\
 4 & \frac{3}{4} & 0.75000 & 0.70769 & 0.04230 \\
 6 & \frac{13}{16} & 0.81250 & 0.79948 & 0.01300 \\
 8 & \frac{109}{128} & 0.85156 & 0.84489 & 0.00667 \\
 10 & \frac{559}{640} & 0.87344 & 0.87237 & 0.00107 \\
 12 & \frac{229}{256} & 0.89453 & 0.89097 & 0.00356 \\
 14 & \frac{1621}{1792} & 0.90458 & 0.90450 & 0.00008 \\
 16 & \frac{11255}{12288} & 0.91593 & 0.91483 & 0.00110 \\
 18 & \frac{2837}{3072} & 0.92350 & 0.92302 & 0.00048 \\
 20 & \frac{4763}{5120} & 0.93027 & 0.92968 & 0.00059 \\
        \end{tabular}
        \caption{For even $k$, we have $S_k=1-\frac{2}{k} + \frac{6}{k\sqrt{\pi k + 14}} - \frac{2}{k(k+2)}$.}
        \label{tab:evenK}
    \end{subtable}%
    \begin{subtable}{.5\textwidth}
        \centering
        \begin{tabular}{C||R@{${}\approx{}$}LCC}
        k & \multicolumn{2}{C}{C_k} & S_k & C_k-S_k \\ \hline\hline
3 & \frac{7}{12} & 0.58333 & 0.32249 & 0.26085 \\
 5 & \frac{41}{60} & 0.68333 & 0.61970 & 0.06363 \\
 7 & \frac{171}{224} & 0.76339 & 0.72961 & 0.03379 \\
 9 & \frac{193}{240} & 0.80417 & 0.78708 & 0.01709 \\
 11 & \frac{5}{6} & 0.83333 & 0.82284 & 0.01049 \\
 13 & \frac{4969}{5824} & 0.85319 & 0.84748 & 0.00572 \\
 15 & \frac{13379}{15360} & 0.87103 & 0.86562 & 0.00541 \\
 17 & \frac{34559}{39168} & 0.88233 & 0.87961 & 0.00272 \\
 19 & \frac{43457}{48640} & 0.89344 & 0.89076 & 0.00268 \\
 21 & \frac{17775}{19712} & 0.90173 & 0.89990 & 0.00184 \\
        \end{tabular}
        \caption{For odd $k$, we have $S_k=1-\frac{3}{k} + \frac{10}{k\sqrt{\pi k + 7}} - \frac{6}{k(k+1)}$.}
        \label{tab:oddK}
    \end{subtable}

    \caption{Let $k\in\N$ be fixed. The best lower bound of this section is $(C_k-\littleO(1))kn\leq f_k(n)$ from \autoref{cor:asym}. 
    Its approximated value is $S_k\cdot kn\leq f_k(n)$ from \autoref{k_asymp}.}
    \label{tab:lowerBoundsGeneral}
\end{table}
\end{remark}

\section{Expected value of the number of deleted points: proofs of 
\autoref{prop:key} and \autoref{prop:EUpperbound}}
\label{sec:E}

This section is devoted to the proofs of \autoref{prop:key} and \autoref{prop:EUpperbound} to complete the proof of \autoref{thm:asymptotics} from \autoref{sec:constructionGeneral}.

Since the proofs are rather long and technical, we start with an overview of the flow of the proofs. 
Let us recall that \autoref{prop:key} states the following: 
For every $k,t,s\in \N$ with $k=2t+s$, 
 we have 
\begin{equation*}
    \EE(X_{k,p} \mid A_t)=F_{t,s}\cdot p+\bigO(1)
\end{equation*}
for the expected number of deleted points of \autoref{con:randomConstruction}
as $p\to\infty$ where $F_{t,s}$ is given by \autoref{def:F}, 
and $A_t$ is the set of $t$-element subsets of $\F_p^*$, cf. \autoref{def:ProbabilitySpace}.
By \autoref{def:ProbabilitySpace} and \autoref{lem:construction},  we have 
\begin{equation}
    \EE(X_{k,p}\mid A_t)=
\sum_{W\in A_t} X_{k,p}(W)\cdot \frac{1}{\Cardinality{A_t}}
= \frac{1}{\Cardinality{A_t}}\sum_{W\in A_t}\sum_{\ell\in\cL'} \max\{0,\Cardinality{\ell\cap \cS(W)}-k\},
\label{eq:EfirstStep}
\end{equation}
where $\cS(W)=\bigcup_{c\in W} S(c,p)$ by \autoref{con:randomConstruction}. 
Note that $\Cardinality{\ell\cap \cS(W)}-k$ depends only on $s$ and the $5$-tuple
$$g_\ell(W)\leteq (\Cardinality{\{c\in W:\Cardinality{S(c,p)\cap \ell}=j\}}:j=0,\dots,4) \in\N^5,$$
cf. \eqref{eq:g_ell}. 
More explicitly, $\Cardinality{\ell\cap \cS(W)}-k = \Delta_s(g_\ell(W))$, for the suitable function $\Delta_s\from \N^5 \to \Z$  from \autoref{def:DeltaT}.  
Denote by $$T=T(t,s)=\{g_\ell(W):W\in A_t, \Delta_s(g_\ell(W))>0\}$$ the set of $5$-tuples induced by $W\in A_t$ giving a positive contribution to \eqref{eq:EfirstStep}, cf. \autoref{def:DeltaT}. 
Then rearranging the order of the summation (and leaving the $0$ terms) in \eqref{eq:EfirstStep} gives 
\begin{equation}
    \EE(X_{k,p}\mid A_t)=
\frac{1}{\Cardinality{A_t}}\sum_{\tau\in T}\Delta_s(\tau)\sum_{\ell\in\cL'} g_\ell^{-1}(\tau),
\label{eq:EsecondStep}
\end{equation}
as we will see in depth in \autoref{sec:step1}. 
Then in \autoref{sec:step2}, we evaluate the asymptotic behaviour of the inner sum of \eqref{eq:EsecondStep} (as $p\to\infty$) by carefully analysing the symmetries of the modular hyperbola $S(c,p)$. 
Afterwards, in \autoref{sec:step3}, we evaluate the asymptotic value of the outer sum of \eqref{eq:EsecondStep}. 
\autoref{sec:step4} concludes the proof of \autoref{prop:key}, and using various binomial identities, we compute a more explicit form of $F_{t,s}$, notable for the case $s\in\{0,1\}$.  
Finally, in \autoref{sec:step5}, using Stirling's formula, we determine an upper bound for $F_{t,s}$ in an explicit and closed form, 
thereby proving \autoref{prop:EUpperbound}.

\begin{remark}
Note that the sharpest upper bound for $\EE(X_{k,p}\mid A_t)$ will eventually be given by the cases $F_{t,0}$ and $F_{t,1}$ as noted in \autoref{rem:bestTS}. 
To prove this, throughout this section, we keep $s,t\in\N$ arbitrary with $k=2t+s$, as this essentially does not increase the length of the argument compared to assuming $s\in\{0,1\}$. 
Nevertheless, for the first reading, one may wish to fix the values $t=\left\lfloor k/2\right\rfloor$ (lower integer part), and $s=k-2t\in\{0,1\}$ (depending on the parity of $k$) for the rest of this section 
if the goal is to prove \autoref{k_asymp} without worrying if a sharper result is possible to attain for a different choice of $(t,s)$. 
\end{remark}

\subsection{Step 1: A general formula for the expected value}
\label{sec:step1}

To start, we need the following technical definitions.
\begin{definition}\label{def:M4j}
    For $j\in \{0,\dots, 4\}$ 
    and line $\ell\in\cL'$, 
    define 
    $M_{j}(\ell)\leteq \{c\in \{1,\dots,p-1\} : \Cardinality{\ell \cap S(c,p)}=j\}$.
\end{definition}

\begin{remark}\label{note:Mij}
    For any $W\in \Omega_p$, we have 
    $\bigcup_{j=0}^4 (W\cap M_{j}(\ell)) = W$ 
    and  
    $\Cardinality{\ell\cap \cS(W)} = \sum_{j=0}^4 j\cdot \Cardinality{W\cap M_{j}(\ell)}
    $.
\end{remark}

\begin{definition}\label{def:DeltaT}
    Let $t,s\in \N$. 
    Define the function $\Delta_s\from \N^5 \to \Z$ by $\tau=(\tau_0,\dots,\tau_4) \mapsto 2\tau_4+\tau_3 - \tau_1 - 2\tau_0-s$, and the index set 
    $T(t,s)\leteq \{\tau=(\tau_0,\dots,\tau_4)\in \N^5 : 
    \sum_{j=0}^4 \tau_j=t,\; 
    \Delta_s(\tau) > 0\}$.
\end{definition}
We will use the fact that $|T(t,s)|$ is independent of $p$.
While $X_{k,p}(W)$ is hard to compute for any fixed $W\in \Omega_p$, 
it turns out that the conditional expected value $\EE(X_{k,p}\mid A_t)$ (i.e., the average value of $X_{k,p}$ restricted to $A_t\subseteq \Omega_p$) can be expressed more explicitly.

\begin{lemma}\label{lem:EwithM4j}
    For every $k,t,s\in \N$ with $k=2t+s$,  
    we have  
    \begin{equation}\label{eq:E}
\EE(X_{k,p}\mid A_t)=
    \frac{1}{\binom{p-1}{t}} 
    \sum_{\tau\in T(t,s)}
    \Delta_s(\tau) 
    \sum_{\ell\in\cL'}
    \prod_{j=0}^4\binom{\Cardinality{M_{j}(\ell)}}{\tau_{j}}.
\end{equation}
\end{lemma}

\begin{proof}
    First, for $\ell\in\cL'$, we claim that the map  
    \begin{equation}    
    \begin{aligned}
        g_\ell\from \{W\in A_t:\Cardinality{\ell\cap \cS(W)}-k>0\} &\to T(t,s) \\
        W&\mapsto (\Cardinality{W\cap M_{0}(\ell)}, \dots,\Cardinality{W\cap M_{4}(\ell)})
    \end{aligned}
    \label{eq:g_ell}
    \end{equation}
    is well-defined.
    Indeed, \autoref{note:Mij} shows that 
    $\sum_{j=0}^4 g_\ell(W)_j = \Cardinality{W}=t$. 
    Hence, using the assumption, we see that $k=s+2\sum_{j=0}^4 g_\ell(W)_j$. 
    This together with \autoref{note:Mij} implies that 
    \[\Cardinality{\ell\cap \cS(W)}-k = \Big(\sum_{j=0}^4 j\cdot g_\ell(W)_j\Big) -\Big (s +2\sum_{j=0}^4 g_\ell(W)_j\Big) = 
    \Delta_s(g_\ell(W))\]
    by \autoref{def:DeltaT}. 
    So $g_\ell(W)\in T(t,s)$ by definition as claimed. 
    
    Second, using the claim and \eqref{eq:EfirstStep}, we can express the expected value as 
    \[\EE(X_{k,p}\mid A_t)
    = \frac{1}{\Cardinality{A_t}}\sum_{\ell\in\cL'} 
    \sum_{\tau\in T(t,s)}
    \sum_{W\in g_\ell^{-1}(\tau)}
    \Delta_s(\tau).
    \]
    Finally note that the inverse image has cardinality $\Cardinality{g_\ell^{-1}(\tau)} = \prod_{j=0}^4 \binom{\Cardinality{M_{j}(\ell)}}{\tau_j}$ 
    and that $\Cardinality{A_t}=\binom{p-1}{t}$.
\end{proof}

To express $\EE(X_{k,p}\mid A_t)$ more explicitly, 
using some technical definitions and elementary computational results, 
we compute the right-hand side of \eqref{eq:E} explicitly in two steps (see \autoref{sec:step2} and \autoref{sec:step3}).
Then we determine its exact asymptotic behaviour as $p\to\infty$ to prove \autoref{prop:key} (see \autoref{sec:step4}).  
Finally, we give a simple upper bound to this asymptotic to conclude the proof of \autoref{prop:EUpperbound} (see \autoref{sec:step5}).

\subsection{Step 2: Determining the inner sum of \texorpdfstring{\eqref{eq:E}}{(\ref{eq:E})}}
\label{sec:step2}
We start by computing $\Cardinality{M_{j}(\ell)}$ explicitly for every $\ell\in\cL'$. 

\begin{definition}\label{def:D}
    For $d\in \Z$, denote the lines of slope $\pm 1$ by  
    $\ell_d^+\leteq \ell[y=x+d]$ and 
    $\ell_d^-\leteq \ell[y=-x+d]$. 
    Let $D^+\leteq [-\frac{3p-3}{2}, \frac{5p-3}{2}]$ and 
    $D^-\leteq [-\frac{p-3}{2}, \frac{7p-3}{2} ]$, 
    two integer intervals of size $4p-2$. 
    Let $D^\top\leteq [\frac{p+1}{2},\frac{5p-3}{2}]$, the top half of the interval $D^+$. 
    Partition $D^\top$ into the following intervals of length $\frac{p-1}{2}$, $1$, $\frac{p-1}{2}$, $\frac{p-1}{2}$ and $\frac{p-1}{2}$ respectively: 
        $D_2\leteq [\frac{p+1}{2},p-1 ]$, 
        $D_3\leteq \{p\}$,  
        $D_4\leteq [p+1, \frac{3p-1}{2} ]$,  
        $D_5\leteq [\frac{3p+1}{2}, 2p-1 ]$, 
        $D_6\leteq [2p, \frac{5p-3}{2} ]$.
\end{definition}

\begin{remark}\label{rem:L'}
    Note that $\cL'$ from \autoref{def:G-Omega-L} can be expressed as   
    \[\cL' = \left\{\ell_d^+:d\in D^+\right\} \cup \left\{\ell_d^-:d\in D^- \right\}.\] 
\end{remark}

\begin{lemma}\label{lem:M4j}
    The exact value of $|M_{j}(\ell)|$ for every $\ell\in\cL'$ is given by the following statements. 
    \begin{enumerate}
        \item\label{item:D-} We have $\cardinality{M_{j}(\ell_d^-)} = \cardinality{M_{j}(\ell_{d-p}^+)}$ for every $d\in D^-$. 
        (So it is enough to find $\cardinality{M_{j}(\ell_{d}^+)}$ for $d\in D^+$.)
        
        \item\label{item:D+} We have $\cardinality{M_{j}(\ell_{p-d}^+)} = \cardinality{M_{j}(\ell_d^+)}$ for every $d\in D^+$. 
        (So it is enough to find $\cardinality{M_{j}(\ell_{d}^+)}$ for $d\in D^\top$.)

        \item\label{item:DT} 
        For every $d\in D^\top$, the value of $\cardinality{M_{j}(\ell_d^+)}$ is given in \autoref{tab:M4j}. 
 \begin{table}[!htb]
        \begin{tabular}{C||CCCCC}
        \Cardinality{M_{j}(\ell_d^+)}  & 
        d\in D_2  & 
        d\in D_3 & 
        d\in D_4 & 
        d\in D_5 & 
        d\in D_6
        \\ \hline\hline
        j=0 & 
        \frac{p-1}{2} & 
        \frac{p-1}{2} & 
        \frac{p-1}{2} & 
        \frac{p-1}{2} &
        d-\frac{3p+1}{2}
        \\ 
        j=1 &
        0 &
        0 & 
        1 & 
        d-\frac{3p-1}{2} &
        \frac{5p-1}{2}-d
        \\ 
        j=2 &
        1 & 
        0 &  
        d-p-1 & 
        2p-1-d  &
        0
        \\ 
        j=3 & 
        d-\frac{p+1}{2} & 
        \frac{p-1}{2} & 
        \frac{3p-1}{2}-d & 
        0 &
        0
        \\ 
        j=4 &
        p-1-d & 
        0 & 
        0 & 
        0 &
        0
        \end{tabular}
        \centering
        \caption{The value of $|M_{j}(\ell_d^+)|$ for $d\in D^\top$ and $j\in\{0,\dots,4\}$. The intervals $D_2,\dots,D_6$ are defined in \autoref{def:D}.}
        \label{tab:M4j}
    \end{table}

\end{enumerate} 
\end{lemma}

\begin{figure}[htb!]
	\centering
	
	\definecolor{ffqqqq}{rgb}{1.,0.,0.}
\definecolor{qqqqff}{rgb}{0.,0.,1.}
\definecolor{ffqqff}{rgb}{1.,0.,1.}
\definecolor{qqffqq}{rgb}{0.,1.,0.}
\definecolor{lightred}{RGB}{255, 100, 100}
\definecolor{lightgreen}{RGB}{100, 255, 100}
\definecolor{lightblue}{RGB}{100, 100, 255}
\definecolor{lightgggrey}{rgb}{0.7,0.7,0.7}

\begin{tikzpicture}[line cap=round,line join=round,>=triangle 45,x=0.7cm,y=0.7cm, scale=0.9]
\draw [step=1,gray,line width=0.2pt, opacity=0.5] (-5,0) grid (16,21);

\draw [line width=1.5pt] (0,0)--(0,21);
\draw [line width=1.5pt] (11,0)--(11,21);
\draw [line width=1.5pt] (-5,0)--(16,0);
\draw [line width=1.5pt] (-5,11)--(16,11);

\draw [fill=gray, color=gray, line width=1pt, fill opacity=0.2] (-5,1)--(-5,5)--(-1,5)--(-1,1)--cycle;
\draw [fill=gray, color=gray, line width=1pt, fill opacity=0.2] (-5,6)--(-5,10)--(-1,10)--(-1,6)--cycle;
\draw [fill=gray, color=gray, line width=1pt, fill opacity=0.2] (-5,12)--(-5,16)--(-1,16)--(-1,12)--cycle;
\draw [fill=gray, color=gray, line width=1pt, fill opacity=0.2] (-5,17)--(-5,21)--(-1,21)--(-1,17)--cycle;

\draw [fill=gray, color=gray, line width=1pt, fill opacity=0.2] (5,1)--(5,5)--(1,5)--(1,1)--cycle;
\draw [fill=gray, color=gray, line width=1pt, fill opacity=0.2] (5,6)--(5,10)--(1,10)--(1,6)--cycle;
\draw [fill=gray, color=gray, line width=1pt, fill opacity=0.2] (5,12)--(5,16)--(1,16)--(1,12)--cycle;
\draw [fill=gray, color=gray, line width=1pt, fill opacity=0.2] (5,17)--(5,21)--(1,21)--(1,17)--cycle;

\draw [fill=gray, color=gray, line width=1pt, fill opacity=0.2] (6,1)--(6,5)--(10,5)--(10,1)--cycle;
\draw [fill=gray, color=gray, line width=1pt, fill opacity=0.2] (6,6)--(6,10)--(10,10)--(10,6)--cycle;
\draw [fill=gray, color=gray, line width=1pt, fill opacity=0.2] (6,12)--(6,16)--(10,16)--(10,12)--cycle;
\draw [fill=gray, color=gray, line width=1pt, fill opacity=0.2] (6,17)--(6,21)--(10,21)--(10,17)--cycle;

\draw [fill=gray, color=gray, line width=1pt, fill opacity=0.2] (12,1)--(12,5)--(16,5)--(16,1)--cycle;
\draw [fill=gray, color=gray, line width=1pt, fill opacity=0.2] (12,6)--(12,10)--(16,10)--(16,6)--cycle;
\draw [fill=gray, color=gray, line width=1pt, fill opacity=0.2] (12,12)--(12,16)--(16,16)--(16,12)--cycle;
\draw [fill=gray, color=gray, line width=1pt, fill opacity=0.2] (12,17)--(12,21)--(16,21)--(16,17)--cycle;

\draw (-5,-0.78) node[anchor=mid] {$-\frac{p-1}{2}$};
\draw (-1,-0.78) node[anchor=mid] {$-1$};
\draw (0,-0.78) node[anchor=mid] {$0$};
\draw ( 1,-0.78) node[anchor=mid] {$1$};
\draw ( 5,-0.78) node[anchor=mid] {$\frac{p-1}{2}$};
\draw ( 6,-0.78) node[anchor=mid] {$\frac{p+1}{2}$};
\draw (10,-0.78) node[anchor=mid] {$p-1$};
\draw (12,-0.78) node[anchor=mid] {$p+1$};
\draw (16,-0.78) node[anchor=mid] {$\frac{3p-1}{2}$};

\draw (16.4,0) node[anchor=west] {$0$};
\draw (16.4,1) node[anchor=west] {$1$};
\draw (16.4,5) node[anchor=west] {$\frac{p-1}{2}$};
\draw (16.4,6) node[anchor=west] {$\frac{p+1}{2}$};
\draw (16.4,10) node[anchor=west] {$p-1$};
\draw (16.4,11) node[anchor=west] {$p$};
\draw (16.4,12) node[anchor=west] {$p+1$};
\draw (16.4,16) node[anchor=west] {$\frac{3p-1}{2}$};
\draw (16.4,17) node[anchor=west] {$\frac{3p+1}{2}$};
\draw (16.4,21) node[anchor=west] {$2p-1$};

\draw [dashed, line width=1pt] (5.5,0)--(5.5,21);
\draw [dashed, line width=1pt] (-5,0.5)--(15.5,21);

\draw [fill=lightred, draw=none, fill opacity=0.3] (-5,1)--(-5,5)--(11,21)--(15,21)--cycle;
\draw [color=lightred, line width=1pt] (-5,1)--(15,21);
\draw [color=lightred, line width=1pt] (-5,5)--(11,21);
\draw [color=red, line width=2pt] (-5,3)--(13,21);

\draw [color=BurntOrange, line width=2pt] (-5,6)--(10,21);

\draw [fill=lightgreen, draw=none, fill opacity=0.3] (-5,7)--(-5,11)--(5,21)--(9,21)--cycle;
\draw [color=lightgreen, line width=1pt] (-5,11)--(5,21);
\draw [color=lightgreen, line width=1pt] (-5,7)--(9,21);
\draw [color=ForestGreen, line width=2pt] (-5,9)--(7,21);

\draw [fill=lightblue, draw=none, fill opacity=0.3] (-5,12)--(-5,16)--(0,21)--(4,21)--cycle;
\draw [color=lightblue, line width=1pt] (-5,12)--(4,21);
\draw [color=lightblue, line width=1pt] (-5,16)--(0,21);
\draw [color=blue, line width=2pt] (-5,15)--(1,21);

\draw [fill=CarnationPink, draw=none, fill opacity=0.3] (-5,17)--(-5,21)--(-1,21)--cycle;
\draw [color=CarnationPink, line width=1pt] (-5,17)--(-1,21);
\draw [color=purple, line width=2pt] (-5,18)--(-2,21);

\draw[decorate, decoration={brace, amplitude=10pt}, color=lightred, line width=1pt] (-5.4,1)--(-5.4,5);
\draw [color=orange, line width=1pt] (-5.8,6)--(-5.45,6);
\draw[decorate, decoration={brace, amplitude=10pt}, color=lightgreen, line width=1pt] (-5.4,7)--(-5.4,11);
\draw[decorate, decoration={brace, amplitude=10pt}, color=lightblue, line width=1pt] (-5.4,12)--(-5.4,16);
\draw[decorate, decoration={brace, amplitude=10pt}, color=CarnationPink, line width=1pt] (-5.4,17)--(-5.4,21);

\draw (-5.9,3) node[anchor=east, color=lightred] {$D_2$};
\draw (-5.9,6) node[anchor=east, color=orange] {$D_3$};
\draw (-5.9,9) node[anchor=east, color=lightgreen] {$D_4$};
\draw (-5.9,14) node[anchor=east, color=lightblue] {$D_5$};
\draw (-5.9,19) node[anchor=east, color=CarnationPink] {$D_6$};

\draw (13,21.2) node[anchor=south, color=red] {$\ell^{(2)}$};
\draw (10,21.2) node[anchor=south, color=BurntOrange] {$\ell^{(3)}$};
\draw (7,21.2) node[anchor=south, color=ForestGreen] {$\ell^{(4)}$};
\draw (1,21.2) node[anchor=south, color=blue] {$\ell^{(5)}$};
\draw (-2,21.2) node[anchor=south, color=purple] {$\ell^{(6)}$};

\draw [sparse dotted, line width=1pt] (11,0)--(16.5,5.5) -- (0,22) -- (-5.5,16.5) -- cycle;

\draw [sparse dotted, line width=1pt] (0,0)--(16.5,16.5) -- (11,22) -- (-5.5,5.5) -- cycle;

\draw (4,7) node[labelbox] {$A_{0,0}$};
\draw (15,7) node[labelbox] {$A_{1,0}$};
\draw (2,20) node[labelbox] {$A_{0,1}$};
\draw (15,18) node[labelbox] {$A_{1,1}$};

\draw (-4,9) node[labelbox] {$B_{-1,0}$};
\draw (9,7) node[labelbox] {$B_{0,0}$};
\draw (-4,20) node[labelbox] {$B_{-1,1}$};
\draw (7,20) node[labelbox] {$B_{0,1}$};

\draw (4,2) node[labelbox] {$C_{0,0}$};
\draw (15,2) node[labelbox] {$C_{1,0}$};
\draw (2,15) node[labelbox] {$C_{0,1}$};
\draw (15,13) node[labelbox] {$C_{1,1}$};

\draw (-2,2) node[labelbox] {$D_{-1,0}$};
\draw (9,2) node[labelbox] {$D_{0,0}$};
\draw (-4,15) node[labelbox] {$D_{-1,1}$};
\draw (9,13) node[labelbox] {$D_{0,1}$};


\node[rectangle, draw, fill=black, minimum size=8pt, inner sep=0pt] at (-5,3) {};
\node[rectangle, draw, fill=black, minimum size=8pt, inner sep=0pt] at (-3,5) {};
\node[rectangle, draw, fill=black, minimum size=8pt, inner sep=0pt] at (6,14) {};
\node[rectangle, draw, fill=black, minimum size=8pt, inner sep=0pt] at (8,16) {};

\node[diamond, minimum width=14pt, minimum height=10pt, draw, fill=black!55!white, inner sep=0pt] at ( 1,9) {};
\node[diamond, minimum width=14pt, minimum height=10pt, draw, fill=black!55!white, inner sep=0pt] at (2,10) {};
\node[diamond, minimum width=14pt, minimum height=10pt, draw, fill=black!55!white, inner sep=0pt] at (12,20) {};
\node[diamond, minimum width=14pt, minimum height=10pt, draw, fill=black!55!white, inner sep=0pt] at (13,21) {};

\node[regular polygon, regular polygon sides=3, draw, fill=black!20!white, minimum size=12pt, inner sep=0pt] at (-1,7) {};
\node[regular polygon, regular polygon sides=3, draw, fill=black!20!white, minimum size=12pt, inner sep=0pt] at (4,12) {};
\node[regular polygon, regular polygon sides=3, draw, fill=black!20!white, minimum size=12pt, inner sep=0pt] at (10,18) {};

\node[circle, draw, fill=black!5!white, minimum size=8pt, inner sep=0pt] at (-4,4) {};
\node[circle, draw, fill=black!5!white, minimum size=8pt, inner sep=0pt] at (7,15) {};

\end{tikzpicture}
	
	\caption{Illustration of the grid $\cG$ and the sets $\{\ell_d^{+}: d\in D_m\}$ for each $m\in \{2, \dots, 6\}$, with each set shown as a differently coloured strip. In each set, one line $\ell^{(m)}=\ell_{d_m}^+$ (for some $d_m\in D_m$) is shown thickened as an example. As per \autoref{lem:M4j}, reflecting a line $\ell\in \cL'$ across any of the two dashed axes does not change the values of $|M_{j}(\ell)|$, allowing us to deduce these values for all $\ell$ from just those for $\ell_d^{+}$ with $d\in D^\top$. 
    The set $S(c,p)$ is symmetric to the dotted lines. For various values of $c$, this symmetry is indicated on $S(c,p)\cap \ell^{(2)}$ by the differently shaped grid points.}
	\label{fig:slope1_regimes}
\end{figure}

\begin{proof}
    Recall  \autoref{def:M4j} and \autoref{const:S3S4}.
    Denote by $P'$ the reflection of $P\in\cG$ across any of the lines $\ell[y=\pm x+ap]$ for $a\in\Z$ (indicated by dotted lines in \autoref{fig:slope1_regimes}). 
    Note that \autoref{obs:meet_classes} imply a \emph{partial reflection symmetry}: if $P\in S(c,p)$ (for some $c\in\F_p^*=\{1,\dots,p-1\}$) and $P'\in \cG$, then $P'\in S(c,p)$. 

    For \autoref{item:D-} and \autoref{item:D+}, note that the partial reflection symmetry above implies that $|M_{j}(\ell)|$ is invariant under reflecting the line $\ell\in \cL'$ across the lines $y=x+\frac{p}{2}$ and $x=\frac{p}{2}$ (indicated by dashed lines in \autoref{fig:slope1_regimes}).
    
    For \autoref{item:DT}, we first prove the case $d\in D_2$. 
    Consider $\ell_d^+$ (the red line in \autoref{fig:slope1_regimes}) 
    and its partial reflection symmetries to the lines $\ell[y=-x+ap]$ for $a\in\{0,1,2,3\}$, i.e., the dotted lines orthogonal to $\ell_d^+$.
    Let $P\leteq (x,y)\in\ell_d^+\cap \cG$ such that $p\nmid x$ and $p\nmid y$. Then there is a unique $c\in\F_p^*$ for which $P\in S(c,p)$, i.e., $c\equiv xy\pmod{p}$. 
    To describe $S(c,p)$, we separate the following disjoint cases based on which block (from \autoref{def:blocks}) $P$ belongs to. 
    Let $L\leteq \bigcup_{a=0}^3 \ell[y=-x+ap]$, the union of the points of our dotted lines.

    \begin{itemize}
        \item If $P\in  \ell_d^+\cap ((D_{-1,0}\cup D_{0,1})\setminus L)$, then the partial reflection symmetries in the dotted lines give four points of $\ell_d^+\cap S(c,p)$ (including $P$): two from $D_{-1,0}$ and two from $D_{0,1}$ (indicated by 
        \tikz[baseline=-1mm] \node[rectangle, draw, fill=black, minimum size=8pt, inner sep=0pt] (X) {};
         in \autoref{fig:slope1_regimes}). 

         \item Similarly, if $P\in  \ell_d^+\cap ((A_{0,0}\cup A_{1,1})\setminus L)$, then the partial reflection symmetries yield again four points of $\ell_d^+\cap S(c,p)$: two from $A_{0,0}$ and two from $A_{1,1}$ (indicated by 
        \tikz[baseline=-1mm] \node[diamond, minimum width=14pt, minimum height=10pt, draw, fill=black!55!white, inner sep=0pt] {};
         in \autoref{fig:slope1_regimes}).

         \item If $P\in \ell_d^+\cap (B_{-1,0}\cup C_{0,1}\cup B_{0,1})$, then we obtain three points altogether by the partial reflection symmetries: one from each block $B_{-1,0}$, $C_{0,1}$ and $B_{0,1}$
         (indicated by 
        \tikz[baseline=-1mm] \node[regular polygon, regular polygon sides=3, draw, fill=black!20!white, minimum size=12pt, inner sep=0pt] {};
         in \autoref{fig:slope1_regimes}).

         \item Otherwise, $P\in \ell_d^+\cap L$, i.e., $P$ is on a dotted line. 
         Note that $|\ell_d^+\cap L|=2$ and $\ell_d^+\cap L\subseteq S(c,p)$. 
         More explicitly, depending on the parity of $d$, there are two cases. 
         Either one point is from $D_{-1,0}$ and the other is its reflection in $\ell[x=-y+p]$ and belongs to $D_{0,1}$ (as indicated by 
        \tikz[baseline=-1mm] \node[circle, draw, fill=black!5!white, minimum size=8pt, inner sep=0pt] {};
         in \autoref{fig:slope1_regimes}), 
         or there is one in $A_{0,0}$ and one in $A_{1,1}$, its reflection in $\ell[x=-y+2p]$.      
    \end{itemize}
    Since $\Cardinality{\ell_d^+\cap H(c,p)\cap ([a,a+p-1]\times [b,b+p-1])}\leq 2$ for any $(a,b)\in\Z^2$, the (iterated) reflections of $P$ give the full set $\ell_d^+\cap S(c,p)$ in all cases above. 
    Then a simple computation gives 
    \[|M_{j}(\ell_d^+)|=\begin{cases}
        0 \when j=1, \\
        \frac{1}{2} \Cardinality{\ell_d^+\cap L} = 1 \when j=2, \\
        \frac{1}{3}\Cardinality{\ell_d^+\cap (B_{-1,0}\cup C_{0,1}\cup B_{0,1})} = d-\frac{p+1}{2} \when j=3, \\
        \frac{1}{4}\left(\Cardinality{\ell_d^+\cap (D_{-1,0}\cup D_{0,1}\cup A_{0,0}\cup A_{1,1})}-\Cardinality{\ell_d^+\cap L}\right) = p-1-d \when j=4.
    \end{cases}\]
    Finally, for $j=0$, we have $\Cardinality{M_{0}(\ell_d^+)}=\Cardinality{\F_p^*}-\sum_{j=1}^4 \Cardinality{M_{j}(\ell_d^+)} = \frac{p-1}{2}$ as stated in the $D_2$ column of \autoref{tab:M4j}. 

    In the other cases, 
    $\Cardinality{M_{j}(\ell_d^+)}$ can be expressed as  
    \begin{center}
    \begin{tabular}{C||M{2.5cm}|M{2.5cm}|M{2.5cm}|M{2.5cm}}
        \Cardinality{M_{j}(\ell_d^+)}  & 
        d\in D_3 & 
        d\in D_4 & 
        d\in D_5 & 
        d\in D_6
        \\ \hline\hline
        j=1 &
        0 & 
        1 & 
        \Cardinality{\ell_d^+\cap B_{-1,1}}+1 &
        \Cardinality{\ell_d^+\cap B_{-1,1}}
        \\ \hline
        j=2 &
        0 &  
        \multicolumn{2}{C|}{\frac{1}{2}\left(\Cardinality{\ell_d^+\cap (D_{-1,1}\cup A_{0,1})}-1\right)}  &
        0
        \\ \hline
        j=3 & 
        \multicolumn{2}{C|}{\frac{1}{3}\Cardinality{\ell_d^+\cap (B_{-1,0}\cup C_{0,1}\cup B_{0,1})}} & 
        \multicolumn{2}{C}{0}
        \\ \hline
        j=4 &
        \multicolumn{4}{C}{0}
        \end{tabular}
        \end{center}
    in a similar fashion. 
    Computing these cardinalities carefully yields the stated values of \autoref{tab:M4j}.    
\end{proof}

Now, the asymptotics of the inner sum of \eqref{eq:E} can be determined.
\begin{lemma}[Inner sum]\label{lem:T0}
    Let $t,s\in \N$, 
    and pick $\tau\in T(t,s)$ from \autoref{def:DeltaT}. 
    Then as the odd prime $p\to\infty$, we have 
    \[
    \frac{1}{ \binom{p-1}{t}} \sum_{\ell\in\cL'} \prod_{j=0}^4\binom{|M_{j}(\ell)|}{\tau_j}
    = 
    4\Phi(\tau)\frac{\binom{t+1}{\tau_0}}{2^{t+1}(t+1)}\cdot p + \bigO(1),\]
    where 
    \[\Phi(\tau)\leteq \begin{cases}
       
        2 \when  \tau_1=\tau_2=\tau_4=0,\\
       1  \when \tau_1=\tau_2=0<\tau_4 \text {\  \ or \ \ }  \tau_1=\tau_4=0<\tau_2,\\
         0  &\text{else}.
    \end{cases}\]
\end{lemma}
\begin{proof}
    Recall the intervals $D_m$ from \autoref{def:D}. 
	First, we claim that
	\begin{equation}
		K_m(\tau)\leteq \sum_{d\in D_m} \prod_{j=0}^4\binom{|M_{j}(\ell_d^+)|}{\tau_j}
		=  
		\begin{cases}
			\binom{(p-1)/2}{\tau_0}\binom{0}{\tau_1}\binom{1}{\tau_2}\binom{(p-1)/2}{\tau_3+\tau_4+1}
			\when m=2 \\
			\binom{(p-1)/2}{\tau_0}\binom{0}{\tau_1}\binom{0}{\tau_2}\binom{(p-1)/2}{\tau_3}\binom{0}{\tau_4}
			\when m=3 \\
			\binom{(p-1)/2}{\tau_0}\binom{1}{\tau_1}\binom{(p-1)/2}{\tau_2+\tau_3+1}\binom{0}{\tau_4}
			\when m=4 \\ 
			0
			\when m\in\{5,6\}.
		\end{cases}
		\label{eq:binom}
	\end{equation}
	Indeed, for $m\in\{2,4\}$, note that 
	for any $a,b,q\in \N$, we have $\sum_{c=0}^q \binom{c}{a}\binom{q-c}{b} = \binom{q+1}{a+b+1}$ as both sides express the number of subsets $Y\subseteq \{1,\dots,q+1\}$ of size $a+b+1$ 
	(and the left-hand side does so by grouping the cases when 
	$c+1\in Y$ with $\Cardinality{Y\cap \{1,\dots,c\}}=a$).
	This implies the claim for $m\in\{2,4\}$ using the values given by \autoref{tab:M4j}.
	The case $m=3$ is clear as $D_3$ is a singleton. 
	Finally, for $m\in\{5,6\}$, note that $2\tau_4+\tau_3\geq \Delta_s(\tau)>0$, so 
	$\tau_j>0$ for some $j\in\{3,4\}$, hence 
	$\binom{\cardinality{M_{j}(\ell_d^+)}}{\tau_j}=\binom{0}{\tau_j}=0$ for $d\in D_m$ by \autoref{tab:M4j}.

    Next, we claim that 
    \begin{equation}
        K_m(\tau) = \begin{cases}
        \displaystyle\frac{\binom{t+1}{\tau_0}}{2^{t+1} (t+1)!} p^{t+1} + \bigO(p^t)
            \when m\in\{2,4\} \text{ and } \tau_1=\tau_m=0\\
        \bigO(p^t) 
            \otherwise.
    \end{cases}
        \label{eq:Sm}
    \end{equation}
    Indeed, to start, observe that 
     \begin{equation}\binom{ap+b}{c}=\frac{1}{c!}\prod_{r=0}^{c-1} (ap+b-r)=\frac{a^c}{c!}p^c + \bigO(p^{c-1})\label{eq:binom-order}
     \end{equation} for fixed numbers $a,b,c\in \N$. 
    First consider the case $m=2$. 
    Now \eqref{eq:binom-order} with \eqref{eq:binom} shows that $K_2(\tau)$ is a polynomial in $p$ of degree $\tau_0+\tau_3+\tau_4+1\leq t+1$ as $t=\tau_0+\dots+\tau_4$. 
    Hence, this degree is $t+1$ if and only if $\tau_1=\tau_2=0$. 
    In this case, \eqref{eq:binom-order} shows that the leading term is 
    $$\frac{(p/2)^{\tau_0}}{\tau_0!} \frac{(p/2)^{\tau_3+\tau_4+1}}{(\tau_3+\tau_4+1)!}
    = \frac{\binom{t+1}{\tau_0}}{2^{t+1}(t+1)!}p^{t+1},$$ as claimed.  
    The case $m=4$ is similar. 
    If $m\in\{3,5,6\}$, then \eqref{eq:binom-order} shows that the degree of $K_m(\tau)$ in $p$ is at most $\tau_0+\tau_3\leq t$ as claimed.

    We are now ready to prove the statement. 
    Using \autoref{rem:L'} and \autoref{lem:M4j}, we have 
    \[
    \frac{1}{ \binom{p-1}{t}} \sum_{\ell\in\cL'} \prod_{j=0}^4\binom{|M_{j}(\ell)|}{\tau_j}
    = 
    \frac{4}{ \binom{p-1}{t}} 
    \sum_{m=2}^6
    K_m(\tau).
    \]
    Finally, 
    note that $\binom{p-1}{t}=\frac{1}{t!}p^t + \bigO(p^{t-1})$ by  \eqref{eq:binom-order}, 
    and $\sum_{m=2}^6 K_m(\tau) = 
    \Phi(\tau)\frac{\binom{t+1}{\tau_0}}{2^{t+1} (t+1)!} p^{t+1} + \bigO(p^t)$ by 
    \eqref{eq:Sm}.
    The statement follows.
\end{proof}

\subsection{Step 3: Determining the outer sum of \texorpdfstring{\eqref{eq:E}}{(\ref{eq:E})}}
\label{sec:step3}
Since the right-hand side of the identity of \autoref{lem:T0} depends heavily on $\tau_0$, the next definition and statement will be useful to determine the outer sum of \eqref{eq:E}.
\begin{definition}\label{def:Tr}
    For $t,s, r\in \N$ with $r\leq t$, define 
    $T_r(t,s)\leteq \{\tau\in T(t,s):\tau_0=r\}$ 
    where $T(t,s)$ is from \autoref{def:DeltaT}.
\end{definition}

\begin{lemma}[Outer sum]\label{lem:DeltaPhi}
	For any $t,s, r\in \N$ with $r\leq t$, 
    we have 
    \[\sum_{\tau\in T_r(t,s)} \Delta_s(\tau)\Phi(\tau)  = \begin{cases}
            \binom{2t-s+1-4r}{2}+(t-s-3r)
                \when 0\leq r\leq \frac{t-s}{3}, \\
            \binom{2t-s+1-4r}{2} 
                \when \frac{t-s}{3} < r \leq \frac{2t-s}{4}, \\
            0
                \when \frac{2t-s}{4} < r\leq t,
    \end{cases}\]
   where 
    $\Delta_s(\tau)$ is from \autoref{def:DeltaT}, 
    and $\Phi(\tau)$ is from \autoref{lem:T0}. 
\end{lemma}
\begin{proof}
For computational simplicity, we use the variables $A\leteq t-s-3r$, $B\leteq 2t-s-4r$ to replace $s$ and $t$ in the statement. 
Note that $0\leq s+2r=B - 2 A$ and $0\leq t-r=B-A$ by assumption. 
We write the index set $T_r(t,s)$ of the summation as a union of two index sets $I_2$ and $I_4$, which we define and compute below.
First, let 
\begin{align*}
    \!I_2&\!\leteq \{\tau\in T_r(t,s):\tau_1=\tau_2=0\} 
    \\&= \{(r,0,0,\tau_3,\tau_4)\in \N^5:r+\tau_3+\tau_4=t, ~2\tau_4+\tau_3-2r-s>0\}
        && \text{using \autoref{def:Tr},}
    \\&= \{(r,0,0,\tau_3,\tau_4)\in \N^5:\tau_3+\tau_4=B-A, ~2\tau_4+\tau_3+2A-B>0\}
        && \text{using the definition of $A$, $B$,}
    \\&= \{(r,0,0,B-d,d-A)\in\Z^5: B-d\geq 0, ~d-A\geq 0, ~d>0\} 
    && \text{$d\leteq 2\tau_4+\tau_3+2A-B$,}
    \\&= \{(r,0,0,B-d,d-A)\in\N^5: \max\{1,A\}\leq d\leq  B\} 
\intertext{
after noting that $(\tau_3, \tau_4)=(B-d,~d-A) \iff (\tau_3+\tau_4, ~2\tau_4+\tau_3+2A-B) =(B-A, d)$. 
Similarly,  let 
}
    \!I_4&\!\leteq \{\tau\in T_r(t,s):\tau_1=\tau_4=0\} 
    \\&= \{(r,0,\tau_2,\tau_3,0)\in \N^5:r+\tau_2+\tau_3=t, ~\tau_3-2r-s>0\}
        && \text{using \autoref{def:Tr},}
    \\&= \{(r,0,\tau_2,\tau_3,0)\in \N^5:\tau_2+\tau_3=B-A, ~\tau_3+2A-B>0\}
        && \text{using the definitions of $A$, $B$,}
    \\&= \{(r,0,A-d,B-2A+d,0)\in\Z^5: A-d\geq 0, ~B-2A+d\geq 0, ~d>0\} 
        && \text{$d\leteq \tau_3+2A-B,$}
    \\&= \{(r,0,A-d, B-2A+d,0)\in\N^5:1\leq d\leq A\} 
        && \text{using } B-2A\geq 0,
\end{align*}
as  $(\tau_2,\tau_3) = (A-d, ~B-2A+d) \iff (\tau_2+\tau_3, ~\tau_3+2A-B )=(B-A,d)$.
Note that for every $\tau\in I_2\cup I_4$, we have $\Delta_s(\tau) = d$ by definition.
Note that the definition of $\Phi(\tau)$ from \autoref{lem:T0} 
translates to $\Phi(\tau)=2\iff \tau\in I_2\cap I_4$ and to 
$\Phi(\tau)=1\iff \tau \in (I_2\cup I_4)\setminus (I_2\cap I_4)$. 
Thus the sum from the statement can be expressed as  
\[
\sum_{\tau\in T_r(t,s)} \Delta_s(\tau)\Phi(\tau)=
\sum_{\tau\in I_2}\Delta_s(\tau) + \sum_{\tau\in I_4}\Delta_s(\tau) 
= \sum_{d=\max\{1,A\}}^{B} \hspace{-4mm}d\hspace{1mm} +  \sum_{d=1}^{A} d 
=\begin{cases}
\displaystyle A+\sum_{d=1}^B d 
    \when 0\leq A, \\
\displaystyle \sum_{d=1}^B d 
    \when A<0.
\end{cases}
\]
Note that $0\leq A$ if and only if $r\leq \frac{t-s}{3}$. 
Finally, $\sum_{d=1}^B d=\binom{B+1}{2}$ if $0\leq B$, i.e. when $r\leq \frac{2t-s}{4}$, otherwise this sum is $0$. 
Then the statement follows.
\end{proof}

\subsection{Step 4: Explicit formulae for the expected value \texorpdfstring{$\EE(X_{k,p}\mid A_t)$}{E(X{k,p} | At)}}
\label{sec:step4}

We can now prove the first key statement of \autoref{sec:constructionGeneral} about the asymptotic value of $\EE(X_{k,p}\mid A_t)$.
\begin{proof}[Proof of \autoref{prop:key}]
Using the results of the previous subsections, we have 
\begin{align*}
    \EE(X_{k,p}\mid A_t) 
    &= \frac{1}{\binom{p-1}{t}} 
    \sum_{\tau\in T(t,s)}
    \Delta_s(\tau) 
    \sum_{\ell\in\cL'}
    \prod_{j=0}^4\binom{\Cardinality{M_{j}(\ell)}}{\tau_{j}} && \text{using \eqref{eq:E},}
    \\&=  
    \sum_{\tau\in T(t,s)}
    \Delta_s(\tau) 
    \cdot \left(4\Phi(\tau)\frac{\binom{t+1}{\tau_0}}{2^{t+1}(t+1)}\cdot p + \bigO(1)\right) 
        && \text{using \autoref{lem:T0}, }
    \\&=  
    \left(\frac{2}{2^{t}(t+1)}\sum_{r=0}^t \binom{t+1}{r}
    \sum_{\tau\in T_r(t,s)}
    \Delta_s(\tau) \Phi(\tau)\right)\cdot p + \bigO(1) 
        && \text{using \autoref{def:Tr},}
    \\&= F_{t,s}\cdot p + \bigO(1) 
        && \text{using \autoref{lem:DeltaPhi} and \eqref{eq:Fdef}}
\end{align*}
as stated.
\end{proof}

To find an upper bound for $F_{t,s}$ from \eqref{eq:Fdef} and prove \autoref{prop:EUpperbound} in the next subsection, 
the following binomial identity will be useful.

\begin{lemma}[A more explicit form]\label{moreex}
    Let $t,s\in \N$.  Write $a\leteq \left\lfloor\frac{t-s}{3}\right\rfloor\leq \left\lfloor\frac{2t-s}{4}\right\rfloor\eqlet b$.  
    Then $F_{t,s}$ from \eqref{eq:Fdef} has the following alternative form.
    \begin{equation}
    \begin{aligned}
         F_{t,s}= 4+\frac{(s+1)(s+2)}{t+1} 
          + \left(2-\frac{2(s+1)}{t+1}\right)&\cdot \frac{1}{2^{t}}\binom{t}{a}
         -\left(2+\frac{4(s+1)}{t+1} \right) \sum_{r=0}^{a-1} \frac{1}{2^{t}}\binom{t}{r}  
         {+} \\
         {+} \Big(  4(t-s-2b-1) -2 \Big) &\cdot \frac{1}{2^{t}} \binom{t}{b}
         - \left(4+\frac{(s+1)(s+2)}{t+1}\right)\sum_{r=b+1}^{t-b} \frac{1}{2^{t}}\binom{t}{r} 
    \end{aligned}
    \label{eq:Fsimplified}
    \end{equation}
    In particular, 
    writing $\indicator{2\nmid t}\leteq 1$ if $t$ is odd, and $\indicator{2\nmid t}\leteq 0$ if $t$ is even, we have 
    \begin{align}
        F_{t,0} &= 
        4 
        -\frac{6+\frac{2}{t+1}\indicator{2\nmid t}}{2^t}\binom{t}{\left\lfloor\frac{t}{2}\right\rfloor}
        +\frac{2}{t+1}
        + \frac{2-\frac{2}{t+1}}{2^t}\binom{t}{\left\lfloor\frac{t}{3}\right\rfloor}
        -\frac{2+\frac{4}{t+1}}{2^t}\sum_{r=0}^{\left\lfloor\frac{t}{3}\right\rfloor-1} \binom{t}{r},
        \label{eq:F0}
        \\
        F_{t,1} &= 
        4
        -\frac{10+\frac{6}{t+1}\indicator{2\nmid t}}{2^t}\binom{t}{\left\lfloor\frac{t}{2}\right\rfloor}
        +\frac{6}{t+1}
        + \frac{2-\frac{4}{t+1}}{2^t}\binom{t}{\left\lfloor\frac{t-1}{3}\right\rfloor}
        -\frac{2+\frac{8}{t+1}}{2^t}\sum_{r=0}^{\left\lfloor\frac{t-1}{3}\right\rfloor-1} \binom{t}{r}.
        \label{eq:F1}
    \end{align}
\end{lemma}

\begin{remark}\label{rem:negBinom}
    By convention, we use the following extension of the function $\binom{x}{r}$ for \textit{integer} values of $x$ and $r$:
    \begin{enumerate}[label=(\alph*)]
        \item $\binom{x}{r}=\frac{x(x-1)\ldots(x-r+1)}{r!}$ if $r\ge 0$,
        \item $\binom{x}{r}=0$ if $r<0$.
    \end{enumerate} 
\end{remark}

\begin{proof}
    We claim that for any $t,s\in\Z$ and $\alpha\in\Z$, we have   
    \[\sum _{r=0}^\alpha (t-s-3r)\binom{t+1}{r}  = (t-s) \binom{t}{\alpha}-(t+2 s+3) \sum _{r=0}^{\alpha-1} \binom{t}{r}.\]
    This is true for $\alpha<0$ by \autoref{rem:negBinom}(b), and for $\alpha\in\N$, it can be verified by induction on $\alpha$ using the binomial identities 
    $\frac{t+1}{\alpha+1}\binom{t}{\alpha}=\binom{t+1}{\alpha+1}=\binom{t}{\alpha}+\binom{t}{\alpha+1}$. 
Similarly, for any $t\in\N$ and $s,\beta\in\Z$ with $\beta<\frac{t+1}{2}$, we claim that 
    \[\sum_{r=0}^{\beta} 
        \binom{2t-s+1-4r}{2} \binom{t+1}{r}  
        = \left(2 t+2 + \frac{(s+1)(s+2)}{2}\right) \left(2^t-\sum _{\mathclap{r=\beta+1}}^{t-\beta} \binom{t}{r}\right) 
        +(t+1)(2t-2s-4\beta-3) \binom{t}{\beta}\]
        holds by \autoref{rem:negBinom}(b) and by induction on $\beta$ using the identities 
$\binom{x}{2}=\frac{x(x-1)}{2}$ (for any $x\in\Z$), 
$\frac{t+1}{\beta+1}\binom{t}{\beta}=\binom{t+1}{\beta+1}=\binom{t}{\beta}+\binom{t}{\beta+1}$ for $\beta\ne -1$, and 
$\binom{t}{\beta}=\binom{t}{t-\beta}$ for $t\ge 0$. 
Apply these two claims with $\alpha\leteq a=\left\lfloor\frac{t-s}{3}\right\rfloor$, 
and $\beta\leteq b=\left\lfloor\frac{2t-s}{4}\right\rfloor$ to \eqref{eq:Fdef}
after noting that $\beta<\frac{t+1}{2}$ as $s\geq 0$. 
Substituting back to \eqref{eq:Fdef}, a short computation produces \eqref{eq:Fsimplified} as stated.

    We obtain \eqref{eq:F0} and \eqref{eq:F1} by 
    simplifying \eqref{eq:Fsimplified} in the case $s\in\{0,1\}$. 
    More concretely, we need to evaluate the given weighted sum of the binomial coefficients $\binom{t}{r}$ for $b\leq r\leq t-b$.
    Depending on the parity of $t$ and the value of $s$, we can express this set as follows.
    \[\{r\in\N:b\leq r\leq t-b\}=\begin{cases}
        \{b\}
        = \left\{\frac{t}{2}\right\}
            \when s=0,\quad 2\mid t
        \\ 
        \{b, b+1\}
        = \left\{\frac{t-1}{2}, \frac{t+1}{2}\right\}
            \when s=0,\quad 2\nmid t
        \\
        \{b,b+1,b+2\}
        = \left\{\frac{t}{2}-1, \frac{t}{2}, \frac{t}{2}+1 \right\}
            \when s=1,\quad 2\mid t
        \\
        \{b,b+1\}
        = \left\{\frac{t-1}{2}, \frac{t+1}{2}\right\}
            \when s=1,\quad 2\nmid t\\
    \end{cases}\]
    Note that all of the resulting binomial coefficients can be expressed as a multiple of $\binom{t}{\lfloor t/2\rfloor}$ using standard binomial identities: 
    for $2\mid t$, we have 
    $\binom{t}{t/2}=\binom{t}{\lfloor t/2\rfloor}$, 
    $\binom{t}{(t/2)-1}= \binom{t}{(t/2)+1} = \frac{t}{t+2}\binom{t}{\lfloor t/2\rfloor}$, 
    whereas for $2\nmid t$, we have 
    $\binom{t}{(t-1)/2}=\binom{t}{(t+1)/2}=\binom{t}{\lfloor t/2\rfloor}$.
    A straightforward simplification of \eqref{eq:Fsimplified} using the observations above yield \eqref{eq:F0} and \eqref{eq:F1} as stated. 
\end{proof}

\subsection{Step 5: Determining an upper bound on \texorpdfstring{$F_{t,s}$}{F{t,s}} -- Proof of \autoref{prop:EUpperbound}}
\label{sec:step5}

The final ingredients are some estimates for the following special binomial coefficients.
\begin{lemma}\label{lem:binomBounds}
    For every $t\in\N$, we have 
    \[\frac{1}{\sqrt{\frac{\pi}{2}t + 4- \frac{\pi}{2}}}\leq \frac{1}{2^t}\binom{t}{\lfloor t/2\rfloor} \qquad \text{and}\qquad 
    \frac{1}{2^t}\binom{t}{\lfloor t/3\rfloor} \leq 3\cdot 0.95^t.\]
\end{lemma}
\begin{proof}
    For both statements, we use the following variant of Stirling's formula \cite{Stirling}: 
    for any $m\in\zz^{+}$, we have $m!=\sqrt{2\pi m}\left(\frac{m}{e}\right)^m \exp(r_m)$ where $r_m$ satisfies $\frac{1}{12m+1}<r_m<\frac{1}{12m}$.

    Consider the first statement. One can check that it holds for $t\in\{0,1,\dots,7\}$. 
    Thus we let $m\leteq \lceil t/2\rceil\geq 4$. Now we have 
   \begin{align*}
   \frac{1}{2^t}\binom{t}{\lfloor t/2\rfloor}
    &=\frac{1}{4^m}\binom{2m}{m}
        && \text{as for odd $t$, $\binom{t}{\lfloor t/2\rfloor}=\binom{2m-1}{m-1}=\frac{1}{2}\binom{2m}{m}$, }
    \\&= \sqrt{\frac{\exp(2r_{2m}-4r_m)}{\pi m}}
        && \text{by Stirling's formula,}
    \\&\geq \sqrt{\frac{1+2r_{2m}-4r_m}{\pi m}}
        &&\text{as $\exp(x)\geq 1+x$ for every $x\in \R$,}
    \\&> \sqrt{\frac{1+\frac{2}{25m}-\frac{4}{12m}}{\pi m}} 
        && \text{as $\tfrac{1}{25m}<\tfrac{1}{12\cdot 2m+1}<r_{2m}$ and $r_m<\tfrac{1}{12m}$, }
    \\&\geq  \frac{1}{\sqrt{\pi (m-1) + 4}}
        && \text{equivalently, $m\geq \tfrac{76-19 \pi}{300-94 \pi}\approx 3.5$,}
    \\&\geq \frac{1}{\sqrt{\pi \frac{t-1}{2} + 4}}
        && \text{as $m=\lceil t/2\rceil\leq \tfrac{t+1}{2}$.}
    \end{align*}
Note that the penultimate inequality is equivalent to $(300-94\pi)m -(76-19 \pi)\geq 0$ after squaring both sides and multiplying by the common  denominator.

Consider the second part. It is true for $t\in\{0,1,2\}$. 
Otherwise, for $m\leteq \lfloor t/3\rfloor \geq 1$, and in a similar fashion, using $0<\frac{1}{12m+1}<r_m<\frac{1}{12m}$, 
we have 
\begin{align*}
    \frac{1}{2^t}\binom{t}{\lfloor t/3 \rfloor} 
    &\leq \frac{1}{2^{3m}}\binom{3m+2}{m} 
    = \frac{1}{8^m} \frac{(3m+2)(3m+1)}{(2m+2)(2m+1)}\binom{3m}{m}
    \leq \frac{1}{8^m}\cdot  \frac{9}{4} \binom{3m}{m}
    \\&=   \frac{9\sqrt{3}}{8\sqrt{\pi m} } \exp(r_{3m}-r_{2m}-r_m) \left(\frac{27}{32}\right)^m 
    \\&< \underbrace{\frac{9\sqrt{3}}{8\sqrt{\pi m} }\exp\left(\frac{1}{12\cdot 3m}\right)}_{g(m)} \cdot \left(\frac{27}{32}\right)^m 
    \le g(1) \cdot \left(\frac{27}{32}\right)^m 
    < 3\cdot \left(\frac{27}{32}\right)^{t/3} 
    <  3\cdot 0.95^t,
\end{align*}
since the function $g(m)$ is strictly decreasing for $m>0$.
\end{proof}

Finally, 
we can now prove the second key statement of \autoref{sec:constructionGeneral} about bounding the asymptotic value of $\EE(X_{k,p}\mid A_t)$.

\begin{proof}[Proof of \autoref{prop:EUpperbound}]
     The formula of \eqref{eq:Fdef} shows that $s\mapsto F_{t,s}$ is a decreasing function. 

     Now consider the upper bound for $F_{t,0}$ from \eqref{eq:EUpperBound}. Using \eqref{eq:F0}, a computer verifies the claim for  $t< 300$, 
     so assume $t\geq 300$. 
     Now the function $g\from \R\to \R,x\mapsto \left(\frac{\pi}{2}t+x\right)^{-1/2}$ is strictly decreasing and convex, so $g(c)\geq g(d)-g'(d)\cdot (d-c)$ 
     for $c\leteq 4-\frac{\pi}{2}<\tfrac{7}{2} \eqlet d$. 
     Note that $-g'(d)\cdot (d-c)=\frac{d-c}{2}\left(\frac{\pi}{2}t+d\right)^{-3/2}\geq 3\cdot 0.95^t$ because $t\geq 300$. 
    Thus 
     \begin{align*}
     F_{t,0} 
     &< 4 + \frac{2}{t+1}-\frac{6}{2^t}\binom{t}{\lfloor t/2\rfloor} + \frac{6}{2^t}\binom{t}{\lfloor t/3\rfloor}
        && \text{using \eqref{eq:F0},}
     \\&\leq 4 + \frac{2}{t+1}- 6\cdot \left(g(c)-3\cdot 0.95^t\right)
        && \text{using \autoref{lem:binomBounds},}
     \\&\leq 4 + \frac{2}{t+1}- 6\cdot\left(g(d)-g'(d)\cdot (d-c)-3\cdot 0.95^t\right)
        && \text{using convexity,}
     \\&\leq 4 + \frac{2}{t+1}- 6\cdot g(d)
        && \text{as $t\geq 300$,}
        \\&= 4 - \frac{12}{\sqrt{2\pi t+14}} + \frac{2}{t+1},
     \end{align*}
     as stated.

     The upper bound for $F_{t,1}$ from \eqref{eq:EUpperBound} follows similarly using \eqref{eq:F1} 
     by separating the cases $t<300$ and $t\geq 300$ 
     and taking $d=\frac{\pi}{4}+\tfrac{7}{4}$. 
\end{proof}

\section{A better  construction for \texorpdfstring{$k=3$}{k=3}}
\label{sec:k=3}

In this section, we sketch the  proof of  \autoref{k=3}, which is a lower bound on $f_3(n)$ improving \autoref{thm:pversion_34}. The full proof is carried out in \autoref{sec:full_k=3}.

The idea is similar to the HJSW construction, but this time we use a grid of size $3p\times 3p$ (where $p\ge 3$ is a prime), hence consisting of $6\times 6$ blocks of normal points. Recall that a point $(x,y)\in \zz\times \zz$ is \textit{normal} if $p\nmid x$ and $p\nmid y$.  Again we will take the points of the modular hyperbola $H(c,p)$ that lie in certain areas of the grid.

The \textit{mod $p$ parity} of an integer $z$ refers to the parity of the unique integer $z_0\in [0, p-1]$ such that $z\equiv z_0\pmod{p}$. 

In order to define the point set, we define new pairing maps between mod $p$ congruence classes that relate to lines of modulus $2$.

Recall from \autoref{obs:meet_classes} that if a line of slope $+1$ met $H(c,p)$ in two congruence classes, then these classes were $(x,y)$ and $(-y,-x)$ for some $x,y\in \F_p^{*}$, and for lines of slope $-1$ we had $(x,y)$ and $(y,x)$. Now we make similar observations for modulus $2$ lines, that is, lines of slope $\pm 2$ or $\pm \frac12$.

\begin{observation}\label{obs:meet_classes_2}
	Let $\ell$ be a line of modulus $2$. If $\ell$ meets $H(c,p)$ in two congruence classes, with one of them being $(x,y)$ for some $x,y\in \F_p^{*}$, then depending on whether $\ell$ has slope $+2$, $-2$, $+\frac12$, or $-\frac12$, the other congruence class must be $\left(-\frac{y}{2}, -2x\right)$, $\left(\frac{y}{2}, 2x\right)$, $\left(-2y, -\frac{x}{2}\right)$, or $\left(2y, \frac{x}{2}\right)$ respectively.
\end{observation}
\begin{proof}
	This can be obtained by considering the quadratic equation that describes the intersections of the modular hyperbola $xy=c$ and the line $y=mx+d$ over $\F_p$ for $m\in \left\{\pm 2, \pm \frac12\right\}$.
\end{proof}

When $\ell$ meets $H(c,p)$ in just one congruence class, the two mentioned congruence classes must coincide (e.g., we have $y=-2x$ in the slope $+2$ case).

For the slope $+2$ case, \autoref{fig:class_pairing_slope2} gives a picture of the transformation $(x,y)\mapsto \left(-\frac{y}{2}, -2x\right)$ for $p=13$. In this illustration, we depict the images of four out of the eight subsets obtained by partitioning the normal points of $\F_p^2$ based on their congruence class location ($A$, $B$, $C$, $D$) and their $y$-coordinate's mod $p$ parity. Mappings for the other slopes are similar in nature, with full columns mapped to full rows. However, for slopes $\pm \frac12$, the structure is based on the mod $p$ parity of the $x$-coordinate instead of that of $y$.

\begin{figure}[htb!]
	\centering
	
	\definecolor{ffqqqq}{rgb}{1.,0.,0.}
\definecolor{qqqqff}{rgb}{0.,0.,1.}
\definecolor{ffqqff}{rgb}{1.,0.,1.}
\definecolor{qqffqq}{rgb}{0.,1.,0.}
\definecolor{lightgggrey}{rgb}{0.7,0.7,0.7}
\definecolor{turquoise}{rgb}{0.,1.,1.}
\begin{minipage}{0.47\textwidth}
\centering
\begin{tikzpicture}[line cap=round,line join=round,>=triangle 45,x=0.5cm,y=0.5cm]
\clip(-1,-1) rectangle (13.5,13.5);
\draw [line width=2.pt] (0.,13.)-- (13.,13.);
\draw [line width=2.pt] (13.,0.)-- (13.,13.);
\draw [line width=2.pt] (0.,13.)-- (0.,0.);
\draw [line width=2.pt] (0.,0.)-- (13.,0.);
\begin{scriptsize}
\foreach \i in {0,...,13}
{
\draw (\i,-0.5) node[anchor=center] {$\i$};
\draw (-0.5,\i) node[anchor=center] {$\i$};
}
\foreach \i in {1,...,6}
{
\draw [fill=ffqqff] (\i,1.) circle (3.45pt);
\draw [fill=qqffqq] (\i,2.) circle (3.45pt);
\draw [fill=ffqqff] (\i,3.) circle (3.45pt);
\draw [fill=qqffqq] (\i,4.) circle (3.45pt);
\draw [fill=ffqqff] (\i,5.) circle (3.45pt);
\draw [fill=qqffqq] (\i,6.) circle (3.45pt);
\draw [fill=turquoise] (\i,7.) circle (3.45pt);
\draw [fill=ffqqqq] (\i,8.) circle (3.45pt);
\draw [fill=turquoise] (\i,9.) circle (3.45pt);
\draw [fill=ffqqqq] (\i,10.) circle (3.45pt);
\draw [fill=turquoise] (\i,11.) circle (3.45pt);
\draw [fill=ffqqqq] (\i,12.) circle (3.45pt);
}
\draw (0.8,12) node[anchor=east] {$R$};
\draw (6.2,12) node[anchor=west] {$S$};
\draw (0.8,8) node[anchor=east] {$T$};
\draw (6.2,8) node[anchor=west] {$U$};
\end{scriptsize}
\end{tikzpicture}
\end{minipage}%
\begin{minipage}{0.06\textwidth}
\centering
\scalebox{2.5}{$\rightarrow$}
\end{minipage}%
\begin{minipage}{0.47\textwidth}
\centering
\begin{tikzpicture}[line cap=round,line join=round,>=triangle 45,x=0.5cm,y=0.5cm]
\clip(-1,-1) rectangle (13.5,13.5);
\draw [line width=2.pt] (0.,13.)-- (13.,13.);
\draw [line width=2.pt] (13.,0.)-- (13.,13.);
\draw [line width=2.pt] (0.,13.)-- (0.,0.);
\draw [line width=2.pt] (0.,0.)-- (13.,0.);
\begin{scriptsize}
\foreach \i in {0,...,13}
{
\draw (\i,-0.5) node[anchor=center] {$\i$};
\draw (-0.5,\i) node[anchor=center] {$\i$};
}
\foreach \i in {1,...,3}
{
\draw [fill=turquoise] (\i,1.) circle (3.45pt);
\draw [fill=turquoise] (\i,3.) circle (3.45pt);
\draw [fill=turquoise] (\i,5.) circle (3.45pt);
\draw [fill=turquoise] (\i,7.) circle (3.45pt);
\draw [fill=turquoise] (\i,9.) circle (3.45pt);
\draw [fill=turquoise] (\i,11.) circle (3.45pt);
}
\foreach \i in {4,...,6}
{
\draw [fill=ffqqff] (\i,1.) circle (3.45pt);
\draw [fill=ffqqff] (\i,3.) circle (3.45pt);
\draw [fill=ffqqff] (\i,5.) circle (3.45pt);
\draw [fill=ffqqff] (\i,7.) circle (3.45pt);
\draw [fill=ffqqff] (\i,9.) circle (3.45pt);
\draw [fill=ffqqff] (\i,11.) circle (3.45pt);
}
\foreach \i in {7,...,9}
{
\draw [fill=ffqqqq] (\i,1.) circle (3.45pt);
\draw [fill=ffqqqq] (\i,3.) circle (3.45pt);
\draw [fill=ffqqqq] (\i,5.) circle (3.45pt);
\draw [fill=ffqqqq] (\i,7.) circle (3.45pt);
\draw [fill=ffqqqq] (\i,9.) circle (3.45pt);
\draw [fill=ffqqqq] (\i,11.) circle (3.45pt);
}
\foreach \i in {10,...,12}
{
\draw [fill=qqffqq] (\i,1.) circle (3.45pt);
\draw [fill=qqffqq] (\i,3.) circle (3.45pt);
\draw [fill=qqffqq] (\i,5.) circle (3.45pt);
\draw [fill=qqffqq] (\i,7.) circle (3.45pt);
\draw [fill=qqffqq] (\i,9.) circle (3.45pt);
\draw [fill=qqffqq] (\i,11.) circle (3.45pt);
}
\draw (7,11.2) node[anchor=south] {$R$};
\draw (7,0.8) node[anchor=north] {$S$};
\draw (9,11.2) node[anchor=south] {$T$};
\draw (9,0.8) node[anchor=north] {$U$};
\end{scriptsize}
\end{tikzpicture}
\end{minipage}%
	
	\caption{Illustration of the congruence class pairing map for slope $+2$ lines for $p=13$, with the images of the four coloured subsets of $\F_p^2$ on the left marked with the same colour on the right. The change in orientation illustrated for the red set by the images of the points marked $R$, $S$, $T$, $U$ is valid for each coloured set.}
	\label{fig:class_pairing_slope2}
\end{figure}

Let $n=3p$ for an odd prime $p$, and take the $3p\times 3p$ grid $\cG^{+}=[0, 3p-1]\times \left[-\frac{p-1}{2}, \frac{5p-1}{2}\right]$. In order to present our no-$4$-in-line construction in $\cG^+$, we introduce the following notion.

\begin{definition}[$m$-equivalent points]
Let $m\in \rr\cup \{\infty\}$. We say that two points $P, Q$ of $\cG^{+}$ are \textit{$m$-equivalent}, if there is a line $\ell$ of slope $m$ that contains both points and there exists an element $c\in \F_p^*$ such that $\{P, Q\}\subseteq H(c,p)$.
\end{definition}

For each direction $m$, this defines an equivalence relation on the normal points of $\cG^+$.

Our aim is the following. We first produce a subset $T$ of $\cG^+$ that contains at most $3$ points in every $m$-equivalence class for every direction $m$, and then our no-$4$-in-line construction will be $T':=T\cap H(c,p)$ for one fixed value $c\in \F_p^*$.

In $\cG^{+}$, every normal congruence class is represented by 9 points arranged in a $3\times 3$ grid-like manner, and by \autoref{obs:atmosttwoclasses} and \autoref{obs:meet_classes}(a), every line can meet $H(c,p)$ in at most two congruence classes, or in at most one if the line is horizontal or vertical. So on horizontal and vertical lines, $T'$ can have at most $3$ points. There are no two congruent points on a line of modulus $\ge 3$, therefore such lines will meet $T'$ in at most $2$ points. So we only need to consider lines of modulus $1$ or $2$.

The class-pairing properties of $H(c,p)$ described in \autoref{obs:meet_classes} and \autoref{obs:meet_classes_2} will enable us to partition $\cG^{+}$ into regions for every direction $m\in \left\{\pm 1, \pm 2, \pm \frac12\right\}$, such that each region only contains pairwise $m$-inequivalent points, and any two regions are either pointwise $m$-equivalent copies of each other or contain no pair of $m$-equivalent points at all.

To obtain our set $T$, we keep at most $3$ regions in every such pointwise $m$-equivalent family of regions, and completely erase the other regions. To maximize the number of remaining points, our intention is to choose the erased regions for the different directions so that they overlap each other as much as possible.

\autoref{fig:finalset} depicts the set $T$, from which we retain each point belonging to the $H(c,p)$ construction for one fixed $c\in \mathbb{F}_p^*$. According to the colours of the regions and the legend, in some of the regions we only put those points into $T$ whose $x$-coordinate, $y$-coordinate, or both, have a certain mod $p$ parity. The full procedure of obtaining $T$ is presented in \autoref{sec:full_k=3}.

\begin{figure}[h!!]
		\centering
		\input{figures/finalset}
        \captionsetup{skip=3pt}
		\captionof{figure}{Our final set $T$, from which the no-$4$-in-line construction $T'$ is obtained via $T':=T\cap H(c,p)$. According to the legend, the colour of each area indicates restrictions on the mod $p$ parity of the coordinates for a point in that area to be included in $T$.}
		\label{fig:finalset}
\end{figure}
\section{Concluding remarks}\label{sec:conclusion}

The main results of our paper establish a lower bound on $f_k(n)$ by $|f_k(n)-kn|=O(n)$
for any  value of $k$. Yet, we believe that this holds in a stronger form, namely $|f_k(n)-kn|=o(n)$ also holds for  $k:=k(n)\gg 1$. 
Notice however that \autoref{k=3} gives evidence that we cannot expect probabilistic bounds using multiple Hall--Jackson--Sudbery--Wild type $2p\times 2p$ constructions to be near-optimal.
As a matter of fact, we conjecture that even the   \autoref{hallcon} of Hall et al. might not be asymptotically tight. On the other hand, as we mentioned at the end of the introduction, Grebennikov and Kwan proved that $|f_k(n)-kn|=o(n)$  holds for  $k:=k(n)\gg 1$ in a strong form: $f_k(n)-kn=0$ for  $k:=k(n)\ge  10^{37}$ provided that $n\ge k$.

Similarly to the case when $k$ was not small, discussed in \cite{kovacs2025settling}, the no-$(k+1)$-in-line property of the presented constructions essentially relied on the fact that they \textit{intersect the lines with small modulus}, i.e., lines of slope $\pm 1$ beside the horizontal and vertical lines, in at most $k$ points.
Several  well-known problems have similar requirements, see e.g. \cite{  klavvzar2025general, cooper2014martin}. The most prominent example might be the so-called $n$-queens problem. To solve it, one aims  to place $n$ chess queens on an $n \times n$ chessboard so that no two queens share the same row, column, or diagonal. 
The number of such placements was determined asymptotically just recently by Simkin \cite{simkin2023number}, see also the preceding work of Bowtell, Keevash \cite{bowtell2021n} and Luria, Simkin \cite{luria2021lower}.

\appendix 
\section{Improvement for certain small values of \texorpdfstring{$k$}{k}}\label{sec:smallK}

\subsection{The main idea}

Let us recall \autoref{con:randomConstruction}:  
we started from $\cS(W)\subseteq \cG$ and removed a  point set $\cR$ from it to obtain a no-$(k+1)$-in-line construction $\cP(W)$. 
In this section, along the idea of \autoref{improve}, we alter the process of finding $\cR$ from \autoref{sec:constructionGeneral} 
to obtain  stronger lower bounds for $f_k(n)$ when $k$ is small, at least for certain values of $k$. 

For example, in the case $W=\{c\}$ and $k=2$, 
\autoref{sec:HJSW-2} essentially details an optimal choice for the points $\cR$ to be removed to gain the largest no-$3$-in-line set $S_2(c,p)$ inside $\cS(W)$. 
This value of $\Cardinality{\cP_k(W)}$ is a sharpening of the estimate of \eqref{eq:constructionSize} that we used in \autoref{sec:constructionGeneral}. 
Similarly, for $W=\{c\}$ and $k=3$, removing the points described in \autoref{sec:HJSW-34} gives the largest no-$4$-in-line set inside $\cS(W)$, namely $S_3(c,p)$.

We mix these optimal removals with the probabilistic process of \autoref{sec:constructionGeneral} to obtain the following process.
For arbitrary $W\in\Omega_p$ and $k\in \N$, we remove points from the disjoint union $\cS_0\leteq \cS(W)=\bigcup_{c\in W}S_4(c,p)\subset \cG$ in two 
steps to obtain a no-$(k+1)$-in-line set.
\begin{enumerate}
    \item \textbf{Manual removal}. 
    For an arbitrary function $\kappa\from W\to \{2,3,4\}$,  
    we manually remove $S_4(c,p)\setminus S_{\kappa(c)}(c,p)$ from $\cS_0$ for every $c\in W$. 
    In this way, we obtain an intermediate set
    $\cS_1\leteq \bigcup_{c\in W} S_{\kappa(c)}(c,p)$. 
    \item \textbf{Probabilistic removal}. Finish by applying the (suboptimal) averaging argument as in \autoref{sec:constructionGeneral} to $\cS_1$ to remove the remaining excess points to obtain a no-$(k+1)$-in-line set $\cS_2$.
\end{enumerate}
Note that when $\kappa(c)=4$ for every $c$, this simplifies to the method described in \autoref{sec:constructionGeneral}. 
If $W=\{c\}$ and $\kappa(c)=k\in\{2,3\}$, then we obtain the best no-$(k+1)$-in-line set already after the manual step, i.e. the probabilistic one has no effect.

In general, instead of removing the excess points from $\cS_0$ in a single (and global) step, 
the process above artificially subdivides this problem into many (local) subproblems of removing some of the excess points from $S_4(c,p)$ each $c\in W$ (before applying a global removal step).
The (manual) solution above to these local problems is optimal but fails to consider the global problem in its entirety, whereas the (probabilistic) solution above to the global problem is suboptimal. 
We expect the effect of the first (manual) step to be relevant for small $k$, and 
suspect that the second (probabilistic) one will dominate for large $k$.

\subsection{Improved construction}

We now expand the ideas above formally. 
The development and structure are very similar to \autoref{sec:constructionGeneral}, 
so here we are only highlighting the key steps.

As in \autoref{def:G-Omega-L}, 
define $\bar \Omega_p$ to be the set consisting of triplets $\bar W=(W_2,W_3,W_4)$ of pairwise disjoint subsets of $\F_p^*=\{1,\dots,p-1\}$. 
Recall the hyperbola $S_i(c,p)\subset \cG$ from \autoref{sec:prelim} for $i\in\{2,3,4\}$ and \autoref{def:G-Omega-L}.
\begin{construction}[Extended randomised construction, cf. \autoref{con:randomConstruction}]
\label{con:randomConstructionModified}
    Let $\bar W\in \bar \Omega_p$ and $k\in\N$ with $k\geq 2( \Cardinality{W_2}+\Cardinality{W_3}+\Cardinality{W_4})$. 
    Let $\cS(\bar W)\leteq \bigcup_{i=2}^4 \bigcup_{c\in W_i} S_i(c,p) \subset \cG'$, 
    and pick a subset $\bar \cR\subseteq \cS(\bar W)$ of minimal size such that  
    $|\ell\cap \bar \cR|\geq |\ell\cap \cS(\bar W)|-k$ 
    for every line $\ell\in\cL'$. 
    Our construction is  $\cP_k(\bar W)\leteq \cS(\bar W)\setminus\bar{\cR}$.
\end{construction}

\begin{lemma}[Construction size]
\label{lem:constructionModified}
    $\cP_k(\bar W)\subset \cG$ from \autoref{con:randomConstructionModified} is a no-$(k+1)$-in-line set of size 
    \begin{equation}
        \Cardinality{\cP_k(\bar W)}
        \geq 
        (p-1)\left(3\Cardinality{W_2}+\tfrac{7}{2}\Cardinality{W_3} + 4\Cardinality{W_4}\right) -\bar X_{k,p}(\bar W)
        \label{eq:constructionSizeModified}
    \end{equation}
    where 
    $\bar X_{k,p}(\bar W)\leteq \sum_{\ell\in\cL'} \max\{0,\Cardinality{\ell\cap \cS(\bar W)}-k\}$.
\end{lemma}
\begin{proof}
    The proof is analogous to that of \autoref{lem:construction}, 
    but now we need to use 
    $\Cardinality{S_2}=3(p-1)$ from \autoref{hjsw_th}
    and 
    $\Cardinality{S_3}\geq \frac{7}{2}(p-1)$, 
    $\Cardinality{S_4} = 4(p-1)$
    from \autoref{thm:pversion_34}.
\end{proof}

As in \autoref{def:ProbabilitySpace}, consider the uniform random discrete probability space on $\bar \Omega_p$ on which $\bar X_{k,p}$ is a random variable.
For $\bar t=(t_2,t_3, t_4)\in\N^3$, define the event $A_{\bar t}\leteq \{\bar W\in \bar \Omega_p: \Cardinality{W_i}=t_i\text{ for } i\in\{2,3,4\}\}$.

In the following statement, we find the conditional expected value $\EE(\bar X_{k,p}\mid A_{\bar t})$ that expresses the number of deleted points from \autoref{con:randomConstructionModified}.
\begin{lemma}[Cf. \autoref{lem:EwithM4j}]\label{lem:EwithMij}
    For every $k\in \N$, $\bar t=(t_2,t_3,t_4)\in\N^3$ and $s\in\N$ with $k=2t_2+2t_3+2t_4+s$,  
    we have  
    \begin{equation}
    \label{eq:Emodified}
\EE(\bar X_{k,p}\mid A_{\bar t})=
    \frac{4}{\binom{p-1}{t_2+t_3+t_4}\cdot \binom{t_2+t_3+t_4}{t_2,t_3,t_4}}  
    \sum_{\bar \tau\in T(\bar t,s)}
    \Delta_s(\bar \tau) 
    \sum_{m=2}^6
    \sum_{d\in D_m}
    \prod_{i=2}^4\prod_{j=0}^4\binom{\Cardinality{M_{i,j}(\ell_d^+)}}{\tau_{i,j}}
\end{equation}
where 
\begin{itemize}
    \item $T(\bar t,s)$ is the index set consisting of those $\bar\tau\leteq (\tau_{i,j})\in\N^{3\times 5}$ 
    ($i\in\{2,3,4\}$ and $j\in\{0,\dots,4\}$) that satisfy 
    $\sum_{j=0}^4 \tau_{i,j}=t_i$ for all $i\in\{2,3,4\}$, 
    and $0<\Delta_s(\bar\tau)\leteq -s+ \sum_{i=2}^4 \sum_{j=0}^4 (j-2)\tau_{i,j}$, and

    \item the value of $M_{i,j}(\ell)\leteq \{c\in\F_p^* : \Cardinality{\ell \cap S_i(c,p)}=j\}$ (for any line $\ell\in\cL'$) is given in \autoref{tab:M4j} (where $M_j(\ell)=M_{4,j}(\ell)$) and \autoref{tab:Mij}.

   \begin{table}[!htb]
    \centering
\begin{tabular}{C||CC|CC|CC|C|C}
\multicolumn{1}{c||}{\multirow{2}{*}{$\Cardinality{M_{i,j}(\ell_d^+)}$}} & 
\multicolumn{2}{C|}{d\in D_2}  & 
\multicolumn{2}{C|}{d\in D_3} & 
\multicolumn{2}{C|}{d\in D_4} & 
d\in D_5 & 
d\in D_6
\\ 
  & i=2 & i=3  & 
i=2 & i=3 & 
i=2 & i=3 & 
i\in \{2,3\} & 
i\in \{2,3\}
 \\ \hline\hline
j=0 & \frac{p-1}{2} & \frac{p-1}{2}  & 
\frac{p-1}{2} & \frac{p-1}{2} & 
\frac{p-1}{2} & \frac{p-1}{2} & 
\frac{p-1}{2} & d-\frac{3p+1}{2}\\
j=1 & 1 & 0 & 
0 & 0 &  
1 & 1 & 
d-\frac{3p-1}{2} & \frac{5p-1}{2}-d
\\
j=2 & \frac{p-3}{2} & d-\frac{p-1}{2} & 
\frac{p-1}{2} & 0 &  
\frac{p-3}{2} & d-p-1 & 
2p-1-d  & 0
\\
j=3 & 0 & p-1-d & 
0 & \frac{p-1}{2} & 
0 & \frac{3p-1}{2}-d & 
0 & 0
\\
j=4 & 0 & 0 & 
0 & 0 & 
0 & 0 & 
0 & 0
\end{tabular}
\caption{The value of $|M_{i,j}(\ell_d^+)|$ for $d\in D^\top$, $i\in\{2,3\}$ and $j\in\{0,\dots,4\}$ which can be obtained as in the proof of \autoref{lem:M4j}. The intervals $D_2,\dots,D_6$ are defined in \autoref{def:D}.}
\label{tab:Mij}
\end{table}

\end{itemize}
\end{lemma}

\begin{proof}
    The proof is analogous to that of \autoref{lem:EwithM4j} with the following minor modifications.

    Note that similarly to \autoref{lem:M4j}, the symmetries 
    $|{M_{i,j}(\ell_d^-)} = \cardinality{M_{i,j}(\ell_{d-p}^+)}$ and $\cardinality{M_{i,j}(\ell_{p-d}^+)} = \cardinality{M_{i,j}(\ell_d^+)}$ hold for every $d\in \Z$, $i\in\{2,3,4\}$ and $j\in\{0,\dots,4\}$, 
    hence we may replace the summation on the lines $\sum_{\ell\in \cL'}$ by the summation  $4\sum_{m=2}^6
    \sum_{d\in D_m}$ as in the proof of \autoref{lem:T0}.

    The map is given by $ g_\ell(\bar W)_{i,j} \leteq \Cardinality{W_i\cap M_{i,j}(\ell)}$. 
    With this setup, we have $\Cardinality{\ell\cap \cS(\bar W)}-k = \Delta_s( g_\ell(\bar W))$, and 
    $\Cardinality{ g_\ell^{-1}(\bar \tau)} = \prod_{i=2}^4 \prod_{j=0}^4 \binom{\Cardinality{M_{i,j}(\ell)}}{\tau_{i,j}}$. 
    Finally $\Cardinality{A_{\bar t}} = \binom{p-1}{t_2+t_3+t_4}\cdot \binom{t_2+t_3+t_4}{t_2,t_3,t_4} = 
    \frac{(p-1)!}{(p-1-t_2-t_3-t_4)! t_2! t_3! t_4!}$. 
\end{proof}

We can determine the asymptotic behaviour of $\EE(\bar X_{k,p}\mid A_{\bar t})$.
\begin{proposition}[Cf. \autoref{prop:key}]
\label{prop:keyModified}
    For every $k\in \N$, $\bar t=(t_2,t_3,t_4)\in \N^3$ and $s\in \N$ satisfying $k=2t_2+2t_3+2t_4+s$, 
    there exists $F_{\bar t, s} \in \R$ independent of $p$ such that 
    as the odd prime $p\to\infty$, we have 
    \begin{equation}
        \EE(\bar X_{k,p}\mid A_{\bar t}) = F_{\bar t, s}\cdot p + \bigO(1).
        \label{eq:EAsymptoticsModified}
    \end{equation}
\end{proposition}

\begin{proof}[Proof of \autoref{prop:keyModified}]
	Define 
	\[P\leteq 4\sum_{\bar \tau\in T(\bar t,s)}
	\Delta_s(\bar \tau) 
	\sum_{m=2}^6
	\sum_{d\in D_m}
	\prod_{i=2}^4\prod_{j=0}^4\binom{\Cardinality{M_{i,j}(\ell_d^+)}}{\tau_{i,j}},
    \qquad 
	Q\leteq \binom{p-1}{t_2+t_3+t_4}\cdot \binom{t_2+t_3+t_4}{t_2,t_3,t_4},\]
	so that the expected value from the statement is $\EE(\bar X_{k,p}\mid A_{\bar t})=\frac{P}{Q}$ using \eqref{eq:Emodified}.  
	
	We claim that $P$ is a polynomial in $p$ of degree $\deg P\leq t+1$ where $t \leteq t_2+t_3+t_4$. 
	Indeed,  
	pick $\bar \tau\in T(\bar t, s)$,  $m\in\{2,\dots,6\}$ and $d\in D_m$. 
	Now \autoref{tab:Mij} and \autoref{tab:M4j} shows that $\Cardinality{M_{i,j}(\ell_d^+)}$ is a polynomial in $p,d$ of total degree at most $1$. 
	Thus $\binom{\Cardinality{M_{i,j}(\ell_d^+)}}{\tau_{i,j}}$ is a polynomial in $p,d$ of total degree at most $\tau_{i,j}$, as $\binom{x}{\delta}=\frac{1}{\delta!}\prod_{\iota=0}^{\delta-1}(x-\iota)$ for any $x\in\R$ and $\delta\in\N$. 
	Hence $P_{\bar \tau,m}(p,d) \leteq \prod_{i=2}^4\prod_{j=0}^4\binom{\Cardinality{M_{i,j}(\ell_d^+)}}{\tau_{i,j}}$ is a polynomial in $p,d$ of total degree at most $\sum_{i=2}^4\sum_{j=0}^4 \tau_{i,j} = \sum_{i=2}^4 t_i = t$ by the definition of $\bar\tau\in  T(\bar t, s)$ from \autoref{lem:EwithMij}. 
	Recall for $x,\delta\in \N$, the sum $\sum_{\iota=1}^x \iota^\delta$ is a polynomial in $x$ of degree $\delta+1$. 
	This implies that $\sum_{d\in D_m} P_{\bar \tau,m}(p,d)$ is a polynomial in $p$ of degree at most $t+1$ because the integer interval $D_m$ has endpoints given by linear polynomial expressions in $p$, see \autoref{def:D}. 
	This shows that $P$ is indeed as claimed.
	
	Next, note that $Q$ is a polynomial of degree $t_2+t_3+t_4=t$ in $p$. 
	This means that  $\EE(\bar X_{k,p}\mid A_{\bar t})=\frac{P}{Q}$ is a rational function in $p$ with $\deg(P)\leq t+1=\deg(Q)+1$, hence the statement follows.
\end{proof}

We obtain the following general lower bound for $f_k(n)$.
\begin{theorem}[Cf. \autoref{thm:asymptotics}]
\label{thm:asymptoticsModified}
    Let $k\in \N$. 
    Pick $\bar t=(t_2,t_3,t_4)\in \N^3$ and $s\in \N$ with $k=2t_2+2t_3+2t_4+s$. 
    Then as $n\to \infty$, we have  
    \begin{equation}
        kn - \big(\tfrac{1}{2}t_2+\tfrac{1}{4}t_3+s+\tfrac{1}{2} F_{\bar t,s}+\littleO(1)\big)n \leq f_k(n).
        \label{eq:bestBoundModified}
    \end{equation} 
\end{theorem}
\begin{proof}
    The proof is completely analogous to that of \autoref{thm:asymptotics}, but now we use 
    \eqref{eq:constructionSizeModified} instead of \eqref{eq:constructionSize}, 
    and \eqref{eq:EAsymptoticsModified} instead of \eqref{eq:EAsymptotics}. 
\end{proof}

\begin{remark}
	Note that \autoref{thm:asymptoticsModified} gives back \autoref{thm:asymptotics} in the case $\bar t = (0,0,t)$.
\end{remark}

While it is possible to compute $\EE(\bar X_{k,p}\mid A_{\bar t})$ using \eqref{eq:Emodified} for concrete values of $\bar t$ and $s$, it would be very tedious to develop a concrete general formula similar to \eqref{eq:Fdef} or to \autoref{moreex}. 
So instead, we use a computer to check every possible setup in \autoref{thm:asymptoticsModified} for $2\leq k\leq 22$. 
As expected, we find some improvements to \autoref{tab:lowerBoundsGeneral}, but these are limited to small $k$ as summarised in \autoref{tab:improvedLowerBounds}. 

\begin{table}[htb!]
    \centering
    \begin{tabular}{C||CC|R@{${}\approx{}$}LCl}
        k & \bar t & s & \multicolumn{2}{C}{\bar C_k} & \bar C_k - C_k & Comment about the type of $\bar t$ 
        \\ \hline\hline
        2 & (1, 0, 0) & 0 & \frac{3}{4} & 0.75000 & \frac{1}{4} 
        	& Using $S_2$ is crucial only here. \\ \hline
        3 & \makecell{(0, 1, 0)\\(0, 0, 1)} & 1 & \frac{7}{12} & 0.58333 & 0 & 
        	Types: $(0,1,t-1)$, $(0,0,t)$. \\ \hline
        4 & \makecell{(2,0,0)\\(1,1,0)\\(1,0,1)\\(0,2,0)\\(0,1,1)\\(0,0,2)} & 0 & \frac{3}{4} & 0.75000 & 0
        	& All possible setups are equally good. \\ \hline
        5 & \makecell{(0,1,1)\\(0,0,2)} & 1 & \frac{41}{60} & 0.68333 & 0 & 
        	Types: $(0,1,t-1)$, $(0,0,t)$. \\ \hline
        6 & \makecell{(0,2,1)\\(0,1,2)} & 0 & \frac{59}{72} & 0.81944 & \frac{1}{144} & 
        	Using $S_3$ is crucial only here. \\ \hline
        8 & \makecell{(0,1,3)\\(0,0,4)} & 0 & \frac{109}{128} & 0.85156 & 0 & 
        	Types: $(0,1,t-1)$, $(0,0,t)$. \\ \hline
        10 & \makecell{(0,1,4)\\(0,0,5)} & 0 & \frac{559}{640} & 0.87344 & 0 &
        	Types: $(0,1,t-1)$, $(0,0,t)$.  
    \end{tabular}
    \caption{For each $k\leq 22$, all optimal setup(s) of \autoref{thm:asymptoticsModified} w.r.t. $\bar t$ are displayed giving the  best bound $(\bar C_k-\littleO(1))\cdot kn\leq f_k(n)$. 
   	We do not display those values of  $k$  where this bound $\bar C_k$ is the same as $C_k$ from \autoref{tab:lowerBoundsGeneral} and no other setup achieves $\bar C_k$. 
   	Actual improvement ($\bar C_k-C_k$) happened only for $k\in\{2,6\}$.
}
    \label{tab:improvedLowerBounds}
\end{table}

The following lower bounds are actual improvements to \autoref{thm:asymptotics}.
\begin{theorem}
	As $n\to \infty$, we have 
    \begin{align*}
        \left(\tfrac{3}{4}-\littleO(1)\right)\cdot 2n& \leq f_2(n)\quad\text{(as done also by \cite{hall1975some}), and }\\
        \left(\tfrac{59}{72}-\littleO(1)\right)\cdot 6n &\leq f_6(n).
    \end{align*}
\end{theorem}

\section{Full proof of \autoref{k=3} -- the detailed case of \texorpdfstring{$k=3$}{k=3}}\label{sec:full_k=3}

In this section, we prove  \autoref{k=3}. The idea is similar to the HJSW construction, but this time we use a grid of size $3p\times 3p$ (where $p\ge 3$ is a prime), hence consisting of $6\times 6$ blocks of normal points. Again we will take the points of the modular hyperbola $H(c,p)$ that lie in certain areas of the grid.

Recall that a point $(x,y)\in \zz\times \zz$ is \textit{normal} if $p\nmid x$ and $p\nmid y$.

\subsection{Congruence class pairing for modulus \texorpdfstring{$2$}{2} lines}

Recall from \autoref{obs:meet_classes} that if a line of slope $+1$ met $H(c,p)$ in two congruence classes, then these classes were $(x,y)$ and $(-y,-x)$ for some $x,y\in \F_p^{*}$, and for lines of slope $-1$ we had $(x,y)$ and $(y,x)$. Now we reiterate some observations for modulus $2$ lines, that is, lines of slope $\pm 2$ or $\pm \frac12$ from \autoref{sec:k=3}.

\begin{observation}
	Let $\ell$ be a line of modulus $2$. If $\ell$ meets $H(c,p)$ in two congruence classes, with one of them being $(x,y)$ for some $x,y\in \F_p^{*}$, then depending on whether $\ell$ has slope $+2$, $-2$, $+\frac12$, or $-\frac12$, the other congruence class must be $\left(-\frac{y}{2}, -2x\right)$, $\left(\frac{y}{2}, 2x\right)$, $\left(-2y, -\frac{x}{2}\right)$, or $\left(2y, \frac{x}{2}\right)$ respectively.
\end{observation}

When $\ell$ meets $H(c,p)$ in just one congruence class, the two mentioned congruence classes must coincide (e.g., we have $y=-2x$ in the slope $+2$ case).

For the slope $+2$ case, \autoref{fig:class_pairing_slope2} gave us a picture of the transformation $(x,y)\mapsto \left(-\frac{y}{2}, -2x\right)$ for $p=13$. In this illustration, we depicted the images of four out of the eight subsets obtained by partitioning the normal points of $\F_p^2$ based on their congruence class location ($A$, $B$, $C$, $D$) and their $y$-coordinate's mod $p$ parity. Mappings for the other slopes are similar in nature, with full columns mapped to full rows. However, for slopes $\pm \frac12$, the structure is based on the mod $p$ parity of the $x$-coordinate instead of that of $y$.

\subsection{Proof of \autoref{k=3}: the \texorpdfstring{$6\times 6$}{6x6} block construction}

Let $n=3p$ for an odd prime $p$, and take the $3p\times 3p$ grid $\cG^{+}=[0, 3p-1]\times \left[-\frac{p-1}{2}, \frac{5p-1}{2}\right]$. In order to present our no-$4$-in-line construction in $\cG^+$, we introduce the following notion.

\begin{definition}[$m$-equivalent points]
Let $m\in \rr\cup \{\infty\}$. We say that two points $P, Q$ of $\cG^{+}$ are \textit{$m$-equivalent}, if there is a line $\ell$ of slope $m$ that contains both points and there exists an element $c\in \F_p^*$ such that $\{P, Q\}\subseteq H(c,p)$.
\end{definition}

For each direction $m$, this defines an equivalence relation on the normal points of $\cG^+$.

Our aim is the following. We first produce a subset $T$ of $\cG^+$ that contains at most $3$ points in every $m$-equivalence class for every direction $m$, and then our no-$4$-in-line construction will be $T':=T\cap H(c,p)$ for one fixed value $c\in \F_p^*$.

In $\cG^{+}$, every normal congruence class is represented by 9 points arranged in a $3\times 3$ grid-like manner, and by \autoref{obs:atmosttwoclasses} and \autoref{obs:meet_classes}(a), every line can meet $H(c,p)$ in at most two congruence classes, or in at most one if the line is horizontal or vertical. So on horizontal and vertical lines, $T'$ can have at most $3$ points. There are no two congruent points on a line of modulus $\ge 3$, therefore such lines will meet $T'$ in at most $2$ points. So we only need to consider lines of modulus $1$ or $2$.

The class-pairing properties of $H(c,p)$ described in  \autoref{obs:meet_classes} and \autoref{obs:meet_classes_2} will enable us to partition $\cG^{+}$ into regions for every direction $m\in \left\{\pm 1, \pm 2, \pm \frac12\right\}$, such that each region only contains pairwise $m$-inequivalent points, and any two regions are either pointwise $m$-equivalent copies of each other or contain no pair of equivalent points at all.

To obtain our set $T$, we keep at most $3$ regions in every such pointwise $m$-equivalent family of regions, and completely erase the other regions. To maximize the number of remaining points, our intention is to choose the erased regions for the different directions so that they overlap each other as much as possible.

In this subsection, the main grid lines in our figures are always of the form $x=k\cdot\frac{p}{2}$ or $y=\ell\cdot\frac{p}{2}$ for some $k,\ell\in\zz$. As a result of this, our figures do not always emphasize the fact that our constructed sets only contain grid points with integer coordinates, but this is always implied. Also, for the sake of simplicity, any black segments on the figures (including grid lines) are considered not to be included in our final construction $T$, even though some of the areas resemble the regions $E$ and $F$ as defined in \autoref{sec:prelim}. Since $H(c,p)$ meets every such segment in at most a constant number of points, this only results in a $O(1)$ decrease in the total size of the construction. The endpoints of such black segments are always of the form $\left(k\cdot \frac{p}{8}, \ell\cdot \frac{p}{8}\right)$ for some $k,\ell\in \zz$.

According to \autoref{obs:meet_classes}(b), for lines of slope $+1$ that meet $H(c,p)$ in two congruence classes, the two classes are of the form $(x,y)$ and $(-y,-x)$. Using this fact, it is simple to partition $\cG^{+}$ for $m=+1$ in the required way. \autoref{fig:slopepm1partition}(a) shows only those region classes that contain at least four $(+1)$-equivalent copies of each region.

It is easy to see that two points of $\cG^+$ are $(+1)$-equivalent if and only if their images under the reflection in the line $x=\frac32p$ are $(-1)$-equivalent. Therefore by using this symmetry, we also get a corresponding decomposition for $m=-1$ as shown in \autoref{fig:slopepm1partition}(b).

\begin{figure}[ht!]
	\centering
	\begin{subfigure}{0.47\textwidth}%
		\centering%
		\definecolor{qqffff}{rgb}{0.,1.,1.}
\definecolor{ududff}{rgb}{0.30196078431372547,0.30196078431372547,1.}
\begin{tikzpicture}[line cap=round,line join=round,>=triangle 45,x=0.19cm,y=0.19cm, scale=0.75]
\clip(-4.,-10.5) rectangle (41.,34.5);

\definecolor{slopeoneA}{RGB}{244,124,124}
\definecolor{slopeoneB}{RGB}{123,170,247}
\definecolor{slopeoneC}{RGB}{127,214,141}
\definecolor{slopeoneD}{RGB}{255,193,102}
\definecolor{slopeoneE}{RGB}{198,156,242}
\definecolor{slopeoneF}{RGB}{109,221,212}
\definecolor{slopeoneG}{RGB}{245,228,109}
\definecolor{slopeoneH}{RGB}{244,145,190}
\definecolor{slopeoneI}{RGB}{214,176,120}
\definecolor{slopeoneJ}{RGB}{155,196,230}

\def\NumStyle{\bfseries\fontsize{9}{9}\selectfont}

\def\FullSquare#1#2#3#4{%
  \fill[fill=#3,fill opacity=0.82] (#1,#2) rectangle ++(6.5,6.5);
  \node[font=\NumStyle] at ({#1+3.25},{#2+3.25}) {#4};
}
\def\HalfTL#1#2#3#4{%
  \fill[fill=#3,fill opacity=0.82] (#1,{#2+6.5}) -- (#1,#2) -- ({#1+6.5},{#2+6.5}) -- cycle;
  \node[font=\NumStyle] at ({#1+2.0},{#2+4.5}) {#4};
}
\def\HalfBR#1#2#3#4{%
  \fill[fill=#3,fill opacity=0.82] (#1,#2) -- ({#1+6.5},#2) -- ({#1+6.5},{#2+6.5}) -- cycle;
  \node[font=\NumStyle] at ({#1+4.5},{#2+2.0}) {#4};
}
\def\HalfBL#1#2#3#4{%
  \fill[fill=#3,fill opacity=0.82] (#1,{#2+6.5}) -- (#1,#2) -- ({#1+6.5},#2) -- cycle;
  \node[font=\NumStyle] at ({#1+2.0},{#2+2.0}) {#4};
}
\def\HalfTR#1#2#3#4{%
  \fill[fill=#3,fill opacity=0.82] (#1,{#2+6.5}) -- ({#1+6.5},{#2+6.5}) -- ({#1+6.5},#2) -- cycle;
  \node[font=\NumStyle] at ({#1+4.5},{#2+4.5}) {#4};
}

\HalfBL{0}{6.5}{slopeoneA}{1}          
\HalfTR{0}{6.5}{slopeoneA}{1}          
\HalfBL{13}{19.5}{slopeoneA}{1}        
\HalfTR{13}{19.5}{slopeoneA}{1}        

\HalfBL{6.5}{13}{slopeoneB}{2}         
\HalfTR{6.5}{13}{slopeoneB}{2}         
\HalfBL{19.5}{26}{slopeoneB}{2}        
\HalfTR{19.5}{26}{slopeoneB}{2}        

\HalfTL{0}{0}{slopeoneC}{3}            
\HalfTL{6.5}{6.5}{slopeoneC}{3}        
\HalfTL{13}{13}{slopeoneC}{3}          
\HalfTL{19.5}{19.5}{slopeoneC}{3}      
\HalfTL{26}{26}{slopeoneC}{3}          

\HalfBR{0}{0}{slopeoneD}{4}            
\HalfBR{6.5}{6.5}{slopeoneD}{4}        
\HalfBR{13}{13}{slopeoneD}{4}          
\HalfBR{19.5}{19.5}{slopeoneD}{4}      
\HalfBR{26}{26}{slopeoneD}{4}          

\HalfBL{0}{-6.5}{slopeoneE}{5}         
\HalfTR{0}{-6.5}{slopeoneE}{5}         
\HalfBL{13}{6.5}{slopeoneE}{5}         
\HalfTR{13}{6.5}{slopeoneE}{5}         
\HalfBL{26}{19.5}{slopeoneE}{5}        
\HalfTR{26}{19.5}{slopeoneE}{5}        

\HalfBL{6.5}{0}{slopeoneF}{6}          
\HalfTR{6.5}{0}{slopeoneF}{6}          
\HalfBL{19.5}{13}{slopeoneF}{6}        
\HalfTR{19.5}{13}{slopeoneF}{6}        
\HalfBL{32.5}{26}{slopeoneF}{6}        
\HalfTR{32.5}{26}{slopeoneF}{6}        

\HalfTL{6.5}{-6.5}{slopeoneG}{7}       
\HalfTL{13}{0}{slopeoneG}{7}           
\HalfTL{19.5}{6.5}{slopeoneG}{7}       
\HalfTL{26}{13}{slopeoneG}{7}          
\HalfTL{32.5}{19.5}{slopeoneG}{7}      

\HalfBR{6.5}{-6.5}{slopeoneH}{8}       
\HalfBR{13}{0}{slopeoneH}{8}           
\HalfBR{19.5}{6.5}{slopeoneH}{8}       
\HalfBR{26}{13}{slopeoneH}{8}          
\HalfBR{32.5}{19.5}{slopeoneH}{8}      

\HalfBL{13}{-6.5}{slopeoneI}{9}        
\HalfTR{13}{-6.5}{slopeoneI}{9}        
\HalfBL{26}{6.5}{slopeoneI}{9}         
\HalfTR{26}{6.5}{slopeoneI}{9}         

\HalfBL{19.5}{0}{slopeoneJ}{10}        
\HalfTR{19.5}{0}{slopeoneJ}{10}        
\HalfBL{32.5}{13}{slopeoneJ}{10}       
\HalfTR{32.5}{13}{slopeoneJ}{10}       

\draw [line width=2pt] (0.,32.5)-- (0.,-6.5);
\draw [line width=2pt] (13.,-6.5)-- (13.,32.5);
\draw [line width=2pt] (26.,32.5)-- (26.,-6.5);
\draw [line width=2pt] (0.,0.)-- (39.,0.);
\draw [line width=2pt] (0.,26.)-- (39.,26.);
\draw [line width=2pt] (39.,13.)-- (0.,13.);
\draw [line width=2pt] (39.,32.5)-- (39.,-6.5);
\draw [line width=1pt] (6.5,32.5)-- (6.5,-6.5);
\draw [line width=1pt] (19.5,-6.5)-- (19.5,32.5);
\draw [line width=1pt] (32.5,32.5)-- (32.5,-6.5);
\draw [line width=1pt] (0.,32.5)-- (39.,32.5);
\draw [line width=1pt] (39.,19.5)-- (0.,19.5);
\draw [line width=1pt] (0.,6.5)-- (39.,6.5);
\draw [line width=1pt] (39.,-6.5)-- (0.,-6.5);

\draw [line width=1pt] (0,13)--(6.5,6.5);        
\draw [line width=1pt] (13,26)--(19.5,19.5);     

\draw [line width=1pt] (6.5,19.5)--(13,13);      
\draw [line width=1pt] (19.5,32.5)--(26,26);     

\draw [line width=1pt] (0,0)--(6.5,-6.5);        
\draw [line width=1pt] (13,13)--(19.5,6.5);      
\draw [line width=1pt] (26,26)--(32.5,19.5);     

\draw [line width=1pt] (6.5,6.5)--(13,0);        
\draw [line width=1pt] (19.5,19.5)--(26,13);     
\draw [line width=1pt] (32.5,32.5)--(39,26);     

\draw [line width=1pt] (13,0)--(19.5,-6.5);      
\draw [line width=1pt] (26,13)--(32.5,6.5);      

\draw [line width=1pt] (19.5,6.5)--(26,0);       
\draw [line width=1pt] (32.5,19.5)--(39,13);     

\draw [line width=1pt] (0,0)--(6.5,6.5);         
\draw [line width=1pt] (6.5,6.5)--(13,13);       
\draw [line width=1pt] (13,13)--(19.5,19.5);     
\draw [line width=1pt] (19.5,19.5)--(26,26);     
\draw [line width=1pt] (26,26)--(32.5,32.5);     

\draw [line width=1pt] (6.5,-6.5)--(13,0);       
\draw [line width=1pt] (13,0)--(19.5,6.5);       
\draw [line width=1pt] (19.5,6.5)--(26,13);      
\draw [line width=1pt] (26,13)--(32.5,19.5);     
\draw [line width=1pt] (32.5,19.5)--(39,26);     

\draw (0,-8.5) node[anchor=center] {$0$};
\draw (6.5,-8.5) node[anchor=center] {$\frac12p$};
\draw (13,-8.5) node[anchor=center] {$p$};
\draw (19.5,-8.5) node[anchor=center] {$\frac32p$};
\draw (26,-8.5) node[anchor=center] {$2p$};
\draw (32.5,-8.5) node[anchor=center] {$\frac52p$};
\draw (39,-8.5) node[anchor=center] {$3p$};
\draw (-2.5,-6.5) node[anchor=center] {$-\frac12p$};
\draw (-2.5,0) node[anchor=center] {$0$};
\draw (-2.5,6.5) node[anchor=center] {$\frac12p$};
\draw (-2.5,13) node[anchor=center] {$p$};
\draw (-2.5,19.5) node[anchor=center] {$\frac32p$};
\draw (-2.5,26) node[anchor=center] {$2p$};
\draw (-2.5,32.5) node[anchor=center] {$\frac52p$};
\end{tikzpicture}%
  \vspace{0.3cm}
		\caption{$m=+1$}%
		\label{fig:slopepm1partition:a}%
	\end{subfigure}%
	\hfill%
	\begin{subfigure}{0.47\textwidth}
		\centering
		\definecolor{qqffff}{rgb}{0.,1.,1.}
\definecolor{ududff}{rgb}{0.30196078431372547,0.30196078431372547,1.}
\begin{tikzpicture}[line cap=round,line join=round,>=triangle 45,x=0.19cm,y=0.19cm, scale=0.75]
\clip(-4.,-10.5) rectangle (41.,34.5);

\definecolor{slopeoneA}{RGB}{244,124,124}
\definecolor{slopeoneB}{RGB}{123,170,247}
\definecolor{slopeoneC}{RGB}{127,214,141}
\definecolor{slopeoneD}{RGB}{255,193,102}
\definecolor{slopeoneE}{RGB}{198,156,242}
\definecolor{slopeoneF}{RGB}{109,221,212}
\definecolor{slopeoneG}{RGB}{245,228,109}
\definecolor{slopeoneH}{RGB}{244,145,190}
\definecolor{slopeoneI}{RGB}{214,176,120}
\definecolor{slopeoneJ}{RGB}{155,196,230}

\def\NumStyle{\bfseries\fontsize{9}{9}\selectfont}

\def\FullSquare#1#2#3#4{%
  \fill[fill=#3,fill opacity=0.82] (#1,#2) rectangle ++(6.5,6.5);
  \node[font=\NumStyle] at ({#1+3.25},{#2+3.25}) {#4};
}
\def\HalfTL#1#2#3#4{%
  \fill[fill=#3,fill opacity=0.82] (#1,{#2+6.5}) -- (#1,#2) -- ({#1+6.5},{#2+6.5}) -- cycle;
  \node[font=\NumStyle] at ({#1+2.0},{#2+4.5}) {#4};
}
\def\HalfBR#1#2#3#4{%
  \fill[fill=#3,fill opacity=0.82] (#1,#2) -- ({#1+6.5},#2) -- ({#1+6.5},{#2+6.5}) -- cycle;
  \node[font=\NumStyle] at ({#1+4.5},{#2+2.0}) {#4};
}
\def\HalfBL#1#2#3#4{%
  \fill[fill=#3,fill opacity=0.82] (#1,{#2+6.5}) -- (#1,#2) -- ({#1+6.5},#2) -- cycle;
  \node[font=\NumStyle] at ({#1+2.0},{#2+2.0}) {#4};
}
\def\HalfTR#1#2#3#4{%
  \fill[fill=#3,fill opacity=0.82] (#1,{#2+6.5}) -- ({#1+6.5},{#2+6.5}) -- ({#1+6.5},#2) -- cycle;
  \node[font=\NumStyle] at ({#1+4.5},{#2+4.5}) {#4};
}

\begin{scope}[cm={-1,0,0,1,(39,0)}]

\HalfBL{0}{6.5}{slopeoneA}{11}
\HalfTR{0}{6.5}{slopeoneA}{11}
\HalfBL{13}{19.5}{slopeoneA}{11}
\HalfTR{13}{19.5}{slopeoneA}{11}

\HalfBL{6.5}{13}{slopeoneB}{12}
\HalfTR{6.5}{13}{slopeoneB}{12}
\HalfBL{19.5}{26}{slopeoneB}{12}
\HalfTR{19.5}{26}{slopeoneB}{12}

\HalfTL{0}{0}{slopeoneC}{13}
\HalfTL{6.5}{6.5}{slopeoneC}{13}
\HalfTL{13}{13}{slopeoneC}{13}
\HalfTL{19.5}{19.5}{slopeoneC}{13}
\HalfTL{26}{26}{slopeoneC}{13}

\HalfBR{0}{0}{slopeoneD}{14}
\HalfBR{6.5}{6.5}{slopeoneD}{14}
\HalfBR{13}{13}{slopeoneD}{14}
\HalfBR{19.5}{19.5}{slopeoneD}{14}
\HalfBR{26}{26}{slopeoneD}{14}

\HalfBL{0}{-6.5}{slopeoneE}{15}
\HalfTR{0}{-6.5}{slopeoneE}{15}
\HalfBL{13}{6.5}{slopeoneE}{15}
\HalfTR{13}{6.5}{slopeoneE}{15}
\HalfBL{26}{19.5}{slopeoneE}{15}
\HalfTR{26}{19.5}{slopeoneE}{15}

\HalfBL{6.5}{0}{slopeoneF}{16}
\HalfTR{6.5}{0}{slopeoneF}{16}
\HalfBL{19.5}{13}{slopeoneF}{16}
\HalfTR{19.5}{13}{slopeoneF}{16}
\HalfBL{32.5}{26}{slopeoneF}{16}
\HalfTR{32.5}{26}{slopeoneF}{16}

\HalfTL{6.5}{-6.5}{slopeoneG}{17}
\HalfTL{13}{0}{slopeoneG}{17}
\HalfTL{19.5}{6.5}{slopeoneG}{17}
\HalfTL{26}{13}{slopeoneG}{17}
\HalfTL{32.5}{19.5}{slopeoneG}{17}

\HalfBR{6.5}{-6.5}{slopeoneH}{18}
\HalfBR{13}{0}{slopeoneH}{18}
\HalfBR{19.5}{6.5}{slopeoneH}{18}
\HalfBR{26}{13}{slopeoneH}{18}
\HalfBR{32.5}{19.5}{slopeoneH}{18}

\HalfBL{13}{-6.5}{slopeoneI}{19}
\HalfTR{13}{-6.5}{slopeoneI}{19}
\HalfBL{26}{6.5}{slopeoneI}{19}
\HalfTR{26}{6.5}{slopeoneI}{19}

\HalfBL{19.5}{0}{slopeoneJ}{20}
\HalfTR{19.5}{0}{slopeoneJ}{20}
\HalfBL{32.5}{13}{slopeoneJ}{20}
\HalfTR{32.5}{13}{slopeoneJ}{20}

\end{scope}

\draw [line width=2pt] (0.,32.5)-- (0.,-6.5);
\draw [line width=2pt] (13.,-6.5)-- (13.,32.5);
\draw [line width=2pt] (26.,32.5)-- (26.,-6.5);
\draw [line width=2pt] (0.,0.)-- (39.,0.);
\draw [line width=2pt] (0.,26.)-- (39.,26.);
\draw [line width=2pt] (39.,13.)-- (0.,13.);
\draw [line width=2pt] (39.,32.5)-- (39.,-6.5);
\draw [line width=1pt] (6.5,32.5)-- (6.5,-6.5);
\draw [line width=1pt] (19.5,-6.5)-- (19.5,32.5);
\draw [line width=1pt] (32.5,32.5)-- (32.5,-6.5);
\draw [line width=1pt] (0.,32.5)-- (39.,32.5);
\draw [line width=1pt] (39.,19.5)-- (0.,19.5);
\draw [line width=1pt] (0.,6.5)-- (39.,6.5);
\draw [line width=1pt] (39.,-6.5)-- (0.,-6.5);

\begin{scope}[cm={-1,0,0,1,(39,0)}]

\draw [line width=1pt] (0,13)--(6.5,6.5);        
\draw [line width=1pt] (13,26)--(19.5,19.5);     

\draw [line width=1pt] (6.5,19.5)--(13,13);      
\draw [line width=1pt] (19.5,32.5)--(26,26);     

\draw [line width=1pt] (0,0)--(6.5,-6.5);        
\draw [line width=1pt] (13,13)--(19.5,6.5);      
\draw [line width=1pt] (26,26)--(32.5,19.5);     

\draw [line width=1pt] (6.5,6.5)--(13,0);        
\draw [line width=1pt] (19.5,19.5)--(26,13);     
\draw [line width=1pt] (32.5,32.5)--(39,26);     

\draw [line width=1pt] (13,0)--(19.5,-6.5);      
\draw [line width=1pt] (26,13)--(32.5,6.5);      

\draw [line width=1pt] (19.5,6.5)--(26,0);       
\draw [line width=1pt] (32.5,19.5)--(39,13);     

\draw [line width=1pt] (0,0)--(6.5,6.5);         
\draw [line width=1pt] (6.5,6.5)--(13,13);       
\draw [line width=1pt] (13,13)--(19.5,19.5);     
\draw [line width=1pt] (19.5,19.5)--(26,26);     
\draw [line width=1pt] (26,26)--(32.5,32.5);     

\draw [line width=1pt] (6.5,-6.5)--(13,0);       
\draw [line width=1pt] (13,0)--(19.5,6.5);       
\draw [line width=1pt] (19.5,6.5)--(26,13);      
\draw [line width=1pt] (26,13)--(32.5,19.5);     
\draw [line width=1pt] (32.5,19.5)--(39,26);     

\end{scope}

\draw (0,-8.5) node[anchor=center] {$0$};
\draw (6.5,-8.5) node[anchor=center] {$\frac12p$};
\draw (13,-8.5) node[anchor=center] {$p$};
\draw (19.5,-8.5) node[anchor=center] {$\frac32p$};
\draw (26,-8.5) node[anchor=center] {$2p$};
\draw (32.5,-8.5) node[anchor=center] {$\frac52p$};
\draw (39,-8.5) node[anchor=center] {$3p$};
\draw (-2.5,-6.5) node[anchor=center] {$-\frac12p$};
\draw (-2.5,0) node[anchor=center] {$0$};
\draw (-2.5,6.5) node[anchor=center] {$\frac12p$};
\draw (-2.5,13) node[anchor=center] {$p$};
\draw (-2.5,19.5) node[anchor=center] {$\frac32p$};
\draw (-2.5,26) node[anchor=center] {$2p$};
\draw (-2.5,32.5) node[anchor=center] {$\frac52p$};
\end{tikzpicture}%
    \vspace{0.3cm}
		\caption{$m=-1$}
		\label{fig:slopepm1partition:b}
	\end{subfigure}%
	\caption{Decompositions of $\cG^+$ into families of $m$-equivalent regions for $m=\pm 1$, only showing families containing $\ge 4$ regions. For fixed $m$, each family consists of all triangular regions with the same number.}
	\label{fig:slopepm1partition}
\end{figure}

\begin{figure}[ht!]
	\centering
	\begin{minipage}{0.47\textwidth}
		\centering
		\definecolor{qqffff}{rgb}{0.,1.,1.}
\definecolor{ududff}{rgb}{0.30196078431372547,0.30196078431372547,1.}
\begin{tikzpicture}[line cap=round,line join=round,>=triangle 45,x=0.19cm,y=0.19cm, scale=0.85]
\clip(-4.,-10.0) rectangle (41.,34.5);
\fill[line width=0.pt,color=qqffff,fill=qqffff,fill opacity=1.0] (0.,0.) -- (6.5,0.) -- (6.5,-6.5) -- (0.,-6.5) -- cycle;
\fill[line width=0.pt,color=qqffff,fill=qqffff,fill opacity=1.0] (6.5,0.) -- (13,0.) -- (13,-6.5) -- (6.5,-6.5) -- cycle;
\fill[line width=0.pt,color=qqffff,fill=qqffff,fill opacity=1.0] (13,0.) -- (19.5,0.) -- (19.5,-6.5) -- (13,-6.5) -- cycle;
\fill[line width=0.pt,color=qqffff,fill=qqffff,fill opacity=1.0] (19.5,0.) -- (26,0.) -- (26,-6.5) -- (19.5,-6.5) -- cycle;
\fill[line width=0.pt,color=qqffff,fill=qqffff,fill opacity=1.0] (26,0.) -- (32.5,0.) -- (32.5,-6.5) -- (26,-6.5) -- cycle;
\fill[line width=0.pt,color=qqffff,fill=qqffff,fill opacity=1.0] (32.5,0.) -- (39,0.) -- (39,-6.5) -- (32.5,-6.5) -- cycle;
\fill[line width=0.pt,color=qqffff,fill=qqffff,fill opacity=1.0] (0.,32.5) -- (6.5,32.5) -- (6.5,26) -- (0.,26) -- cycle;
\fill[line width=0.pt,color=qqffff,fill=qqffff,fill opacity=1.0] (6.5,32.5) -- (13,32.5) -- (13,26) -- (6.5,26) -- cycle;
\fill[line width=0.pt,color=qqffff,fill=qqffff,fill opacity=1.0] (13,32.5) -- (19.5,32.5) -- (19.5,26) -- (13,26) -- cycle;
\fill[line width=0.pt,color=qqffff,fill=qqffff,fill opacity=1.0] (19.5,32.5) -- (26,32.5) -- (26,26) -- (19.5,26) -- cycle;
\fill[line width=0.pt,color=qqffff,fill=qqffff,fill opacity=1.0] (26,32.5) -- (32.5,32.5) -- (32.5,26) -- (26,26) -- cycle;
\fill[line width=0.pt,color=qqffff,fill=qqffff,fill opacity=1.0] (32.5,32.5) -- (39,32.5) -- (39,26) -- (32.5,26) -- cycle;
\fill[line width=0.pt,color=qqffff,fill=qqffff,fill opacity=1.0] (0.,6.5) -- (6.5,6.5) -- (6.5,0) -- (0.,0) -- cycle;
\fill[line width=0.pt,color=qqffff,fill=qqffff,fill opacity=1.0] (0.,13) -- (6.5,13) -- (6.5,6.5) -- (0.,6.5) -- cycle;
\fill[line width=0.pt,color=qqffff,fill=qqffff,fill opacity=1.0] (0.,19.5) -- (6.5,19.5) -- (6.5,13) -- (0.,13) -- cycle;
\fill[line width=0.pt,color=qqffff,fill=qqffff,fill opacity=1.0] (0.,26) -- (6.5,26) -- (6.5,19.5) -- (0.,19.5) -- cycle;
\fill[line width=0.pt,color=qqffff,fill=qqffff,fill opacity=1.0] (32.5,6.5) -- (39,6.5) -- (39,0) -- (32.5,0) -- cycle;
\fill[line width=0.pt,color=qqffff,fill=qqffff,fill opacity=1.0] (32.5,13) -- (39,13) -- (39,6.5) -- (32.5,6.5) -- cycle;
\fill[line width=0.pt,color=qqffff,fill=qqffff,fill opacity=1.0] (32.5,19.5) -- (39,19.5) -- (39,13) -- (32.5,13) -- cycle;
\fill[line width=0.pt,color=qqffff,fill=qqffff,fill opacity=1.0] (32.5,26) -- (39,26) -- (39,19.5) -- (32.5,19.5) -- cycle;
\fill[line width=0.pt,color=qqffff,fill=qqffff,fill opacity=1.0] (6.5,6.5) -- (13.,0.) -- (6.5,0.) -- cycle;
\fill[line width=0.pt,color=qqffff,fill=qqffff,fill opacity=1.0] (6.5,19.5) -- (13.,13) -- (6.5,13) -- cycle;
\fill[line width=0.pt,color=qqffff,fill=qqffff,fill opacity=1.0] (6.5,6.5)--(6.5,13)--(13,13)--cycle;
\fill[line width=0.pt,color=qqffff,fill=qqffff,fill opacity=1.0] (6.5,19.5)--(6.5,26)--(13,26)--cycle;
\fill[line width=0.pt,color=qqffff,fill=qqffff,fill opacity=1.0] (13,19.5)--(19.5,19.5)--(13,26)--cycle;
\fill[line width=0.pt,color=qqffff,fill=qqffff,fill opacity=1.0] (19.5,19.5)--(26,19.5)--(26,26)--cycle;
\fill[line width=0.pt,color=qqffff,fill=qqffff,fill opacity=1.0] (26,26)--(32.5,26)--(32.5,19.5)--cycle;
\fill[line width=0.pt,color=qqffff,fill=qqffff,fill opacity=1.0] (26,13)--(32.5,13)--(32.5,19.5)--cycle;
\fill[line width=0.pt,color=qqffff,fill=qqffff,fill opacity=1.0] (26,13)--(32.5,13)--(32.5,6.5)--cycle;
\fill[line width=0.pt,color=qqffff,fill=qqffff,fill opacity=1.0] (26,0)--(32.5,0)--(32.5,6.5)--cycle;
\fill[line width=0.pt,color=qqffff,fill=qqffff,fill opacity=1.0] (19.5,6.5)--(26,6.5)--(26,0)--cycle;
\fill[line width=0.pt,color=qqffff,fill=qqffff,fill opacity=1.0](13,6.5)--(19.5,6.5)--(13,0)--cycle;
\draw [line width=2pt] (0.,32.5)-- (0.,-6.5);
\draw [line width=2pt] (13.,-6.5)-- (13.,32.5);
\draw [line width=2pt] (26.,32.5)-- (26.,-6.5);
\draw [line width=2pt] (0.,0.)-- (39.,0.);
\draw [line width=2pt] (0.,26.)-- (39.,26.);
\draw [line width=2pt] (39.,13.)-- (0.,13.);
\draw [line width=2pt] (39.,32.5)-- (39.,-6.5);
\draw [line width=1pt] (6.5,32.5)-- (6.5,-6.5);
\draw [line width=1pt] (19.5,-6.5)-- (19.5,32.5);
\draw [line width=1pt] (32.5,32.5)-- (32.5,-6.5);
\draw [line width=1pt] (0.,32.5)-- (39.,32.5);
\draw [line width=1pt] (39.,19.5)-- (0.,19.5);
\draw [line width=1pt] (0.,6.5)-- (39.,6.5);
\draw [line width=1pt] (39.,-6.5)-- (0.,-6.5);
\draw [line width=1pt] (6.5,6.5)-- (13.,0.);
\draw [line width=1pt] (6.5,19.5)-- (13.,13);
\draw [line width=1pt] (6.5,6.5)-- (13.,13);
\draw [line width=1pt] (6.5,19.5)-- (13.,26);
\draw [line width=1pt] (13,26)--(19.5,19.5);
\draw [line width=1pt] (19.5,19.5)--(26,26);
\draw [line width=1pt] (26,26)--(32.5,19.5);
\draw [line width=1pt] (26,13)--(32.5,19.5);
\draw [line width=1pt] (26,13)--(32.5,6.5);
\draw [line width=1pt] (26,0)--(32.5,6.5);
\draw [line width=1pt] (19.5,6.5)--(26,0);
\draw [line width=1pt] (13,0)--(19.5,6.5);

\draw (0,-8.5) node[anchor=center] {$0$};
\draw (6.5,-8.5) node[anchor=center] {$\frac12p$};
\draw (13,-8.5) node[anchor=center] {$p$};
\draw (19.5,-8.5) node[anchor=center] {$\frac32p$};
\draw (26,-8.5) node[anchor=center] {$2p$};
\draw (32.5,-8.5) node[anchor=center] {$\frac52p$};
\draw (39,-8.5) node[anchor=center] {$3p$};
\draw (-2.5,-6.5) node[anchor=center] {$-\frac12p$};
\draw (-2.5,0) node[anchor=center] {$0$};
\draw (-2.5,6.5) node[anchor=center] {$\frac12p$};
\draw (-2.5,13) node[anchor=center] {$p$};
\draw (-2.5,19.5) node[anchor=center] {$\frac32p$};
\draw (-2.5,26) node[anchor=center] {$2p$};
\draw (-2.5,32.5) node[anchor=center] {$\frac52p$};
\end{tikzpicture}
        \captionsetup{skip=3pt}
		\captionof{figure}{The set $T_1$.\\~\\~}
		\label{fig:slope1_fig1}
	\end{minipage}
	\hfill
	\begin{minipage}{0.47\textwidth}
		\centering
		\definecolor{qqffff}{rgb}{0.,1.,1.}
\definecolor{ududff}{rgb}{0.30196078431372547,0.30196078431372547,1.}
\definecolor{lightmagenta}{RGB}{244,145,190}

\begin{tikzpicture}[line cap=round,line join=round,>=triangle 45,x=0.19cm,y=0.19cm, scale=0.85]
\clip(-4.,-10.0) rectangle (41.,34.5);

\fill[line width=0pt,color=lightmagenta,fill=lightmagenta,fill opacity=0.82] (13,32.5) -- (19.5,32.5) -- (19.5,26) -- (13,26) -- cycle; 

\fill[line width=0pt,color=lightmagenta,fill=lightmagenta,fill opacity=0.82] (19.5,32.5) -- (19.5,26) -- (26,26) -- cycle; 
\fill[line width=0pt,color=lightmagenta,fill=lightmagenta,fill opacity=0.82] (26,26) -- (26,32.5) -- (32.5,32.5) -- cycle; 
\fill[line width=0pt,color=lightmagenta,fill=lightmagenta,fill opacity=0.82] (32.5,32.5) -- (39,32.5) -- (39,26) -- cycle; 

\fill[line width=0pt,color=lightmagenta,fill=lightmagenta,fill opacity=0.82] (13,26) -- (13,19.5) -- (19.5,19.5) -- cycle; 
\fill[line width=0pt,color=lightmagenta,fill=lightmagenta,fill opacity=0.82] (19.5,19.5) -- (26,19.5) -- (26,26) -- cycle; 
\fill[line width=0pt,color=lightmagenta,fill=lightmagenta,fill opacity=0.82] (26,26) -- (32.5,26) -- (32.5,19.5) -- cycle; 
\fill[line width=0pt,color=lightmagenta,fill=lightmagenta,fill opacity=0.82] (32.5,26) -- (39,26) -- (39,19.5) -- (32.5,19.5) -- cycle; 

\fill[line width=0pt,color=lightmagenta,fill=lightmagenta,fill opacity=0.82] (0,6.5) -- (6.5,6.5) -- (6.5,0) -- (0,0) -- cycle; 

\fill[line width=0pt,color=lightmagenta,fill=lightmagenta,fill opacity=0.82] (6.5,6.5) -- (6.5,0) -- (13,0) -- cycle; 
\fill[line width=0pt,color=lightmagenta,fill=lightmagenta,fill opacity=0.82] (13,0) -- (13,6.5) -- (19.5,6.5) -- cycle; 
\fill[line width=0pt,color=lightmagenta,fill=lightmagenta,fill opacity=0.82] (19.5,6.5) -- (26,6.5) -- (26,0) -- cycle; 

\fill[line width=0pt,color=lightmagenta,fill=lightmagenta,fill opacity=0.82] (0,0) -- (0,-6.5) -- (6.5,-6.5) -- cycle; 
\fill[line width=0pt,color=lightmagenta,fill=lightmagenta,fill opacity=0.82] (6.5,-6.5) -- (13,-6.5) -- (13,0) -- cycle; 
\fill[line width=0pt,color=lightmagenta,fill=lightmagenta,fill opacity=0.82] (13,0) -- (19.5,0) -- (19.5,-6.5) -- cycle; 
\fill[line width=0pt,color=lightmagenta,fill=lightmagenta,fill opacity=0.82] (19.5,0) -- (26,0) -- (26,-6.5) -- (19.5,-6.5) -- cycle; 

\draw [line width=2pt] (0.,32.5)-- (0.,-6.5);
\draw [line width=2pt] (13.,-6.5)-- (13.,32.5);
\draw [line width=2pt] (26.,32.5)-- (26.,-6.5);
\draw [line width=2pt] (0.,0.)-- (39.,0.);
\draw [line width=2pt] (0.,26.)-- (39.,26.);
\draw [line width=2pt] (39.,13.)-- (0.,13.);
\draw [line width=2pt] (39.,32.5)-- (39.,-6.5);
\draw [line width=1pt] (6.5,32.5)-- (6.5,-6.5);
\draw [line width=1pt] (19.5,-6.5)-- (19.5,32.5);
\draw [line width=1pt] (32.5,32.5)-- (32.5,-6.5);
\draw [line width=1pt] (0.,32.5)-- (39.,32.5);
\draw [line width=1pt] (39.,19.5)-- (0.,19.5);
\draw [line width=1pt] (0.,6.5)-- (39.,6.5);
\draw [line width=1pt] (39.,-6.5)-- (0.,-6.5);

\draw [line width=1pt] (19.5,32.5)--(26,26);     
\draw [line width=1pt] (26,26)--(32.5,32.5);     
\draw [line width=1pt] (32.5,32.5)--(39,26);     

\draw [line width=1pt] (13,26)--(19.5,19.5);     
\draw [line width=1pt] (19.5,19.5)--(26,26);     
\draw [line width=1pt] (26,26)--(32.5,19.5);     

\draw [line width=1pt] (6.5,6.5)--(13,0);        
\draw [line width=1pt] (13,0)--(19.5,6.5);       
\draw [line width=1pt] (19.5,6.5)--(26,0);       

\draw [line width=1pt] (0,0)--(6.5,-6.5);        
\draw [line width=1pt] (6.5,-6.5)--(13,0);       
\draw [line width=1pt] (13,0)--(19.5,-6.5);      

\draw (0,-8.5) node[anchor=center] {$0$};
\draw (6.5,-8.5) node[anchor=center] {$\frac12p$};
\draw (13,-8.5) node[anchor=center] {$p$};
\draw (19.5,-8.5) node[anchor=center] {$\frac32p$};
\draw (26,-8.5) node[anchor=center] {$2p$};
\draw (32.5,-8.5) node[anchor=center] {$\frac52p$};
\draw (39,-8.5) node[anchor=center] {$3p$};
\draw (-2.5,-6.5) node[anchor=center] {$-\frac12p$};
\draw (-2.5,0) node[anchor=center] {$0$};
\draw (-2.5,6.5) node[anchor=center] {$\frac12p$};
\draw (-2.5,13) node[anchor=center] {$p$};
\draw (-2.5,19.5) node[anchor=center] {$\frac32p$};
\draw (-2.5,26) node[anchor=center] {$2p$};
\draw (-2.5,32.5) node[anchor=center] {$\frac52p$};
\end{tikzpicture}
        \captionsetup{skip=3pt}
		\captionof{figure}{The set $T_1^*$, consisting of those points of $T_1$ which have another congruent copy in $T_1$ falling on the same slope $+2$ line.}
		\label{fig:slope1_fig2}
	\end{minipage}
\end{figure}

Let us now define an initial set $T_1\subseteq \cG^{+}$ via \autoref{fig:slope1_fig1}, by taking all normal points that lie in the cyan-coloured area. This includes 20 full blocks and 12 half-blocks (that is, the intersection of a block with either $E$ or $F$, up to the prior remark on the boundary not being included). This area has been constructed so that it contains exactly $3$ regions from every family in \autoref{fig:slopepm1partition}(a) and \autoref{fig:slopepm1partition}(b), and therefore $T_1$ meets every $m$-equivalence class of $\cG^+$ in at most $3$ points for $m=\pm 1$.

We continue with the directions $m\in \left\{\pm 2, \pm\frac12\right\}$. Note that lines in these directions can only meet $\cG^+$ in at most $2$ points per congruence class, and in at most $4$ points altogether. So such a line can lead to a violation of the no-$4$-in-line property only if it meets $\cG^+\cap H(c,p)$ in four points forming two congruent pairs, and all four of these points are in $T_1$.

First we focus on slope $+2$ lines. In $T_1$, a pair of congruent points on the same slope $+2$ line must differ by the vector $(p,2p)$. The set of points that belong to such pairs within $T_1$ is shown by the pink areas of \autoref{fig:slope1_fig2}, and we call this set $T_1^*$. For every slope $+2$ line that contains two pairs of pink points in $H(c,p)$, we must remove at least one of those four points, which we will perform by partitioning these points of $T_1^*$ into families of four $(+2)$-equivalent regions, and removing one region from each family.

As mentioned in \autoref{obs:meet_classes_2}, for every congruence class containing two points of $T_1^*$, we may determine the other congruence class lying on the same slope $+2$ line via the mod $p$ mapping $(x,y)\mapsto \left(-\frac{y}{2}, -2x\right)$. However, the $(+2)$-equivalence class of these two points may have size less than $4$, or even if its size is $4$, it is not necessarily the case that the other two points are also elements of $T_1^*$. (Note that the two classes may coincide, but in those cases the corresponding $(+2)$-equivalence class has size at most $2$.)

Therefore, by using the aforementioned mapping, we need to locate the points of $\cG^+$ that are $(+2)$-equivalent to some element of $T_1^*$. As the congruence class pairing transformations make use of the mod $p$ parity of coordinates (cf. \autoref{obs:meet_classes_2}), we will use the following convention to illustrate these sets of points. Areas shaded with single, or double horizontal lines, are considered to include only those points whose $y$-coordinate is odd mod $p$, or even mod $p$, respectively. Single, or double vertical lines correspondingly indicate $x$-coordinates which are odd, or even, mod $p$.

\autoref{fig:slope2_fig1} illustrates the procedure of finding members of a $(+2)$-equivalence class. Here, as an example, we would like to find all points $(+2)$-equivalent to some point in the blue region $B$, defined as
$$B:=\left\{(x_0,y_0): \tfrac12p<x_0<\tfrac34p, ~0<y_0<\tfrac12p, ~ y_0\equiv 1\pmod{2}\right\}.$$

Note that $B$ is the set of all points in the left half of the block $D_{0,0}$ having odd $y$-coordinate.

\begin{figure}[h!!]
		\centering
		\definecolor{shadered}{RGB}{244,124,124}
\definecolor{pointred}{RGB}{170,45,45}

\begin{tikzpicture}[line cap=round,line join=round,>=triangle 45,x=4.0cm,y=4.0cm, scale=0.75]
\clip(-0.31,-0.81) rectangle (3.15,2.65);


\fill[pattern=horizontalodd, pattern color=blue] (0.5,0) rectangle (0.75,0.5);   
\fill[pattern=horizontalodd, pattern color=blue] (1.5,2) rectangle (1.75,2.5);    

\fill[pattern=horizontaleven, pattern color=OliveGreen] (1.25,1.5) rectangle (1.5,2);  
\fill[pattern=horizontaleven, pattern color=OliveGreen] (0.25,-0.5) rectangle (0.5,0);  

\draw [line width=2pt] (0.,2.5)-- (0.,-0.5);
\draw [line width=2pt] (1.,-0.5)-- (1.,2.5);
\draw [line width=2pt] (2.,2.5)-- (2.,-0.5);
\draw [line width=2pt] (0.,0.)-- (3.,0.);
\draw [line width=2pt] (0.,2.)-- (3.,2.);
\draw [line width=2pt] (3.,1.)-- (0.,1.);
\draw [line width=2pt] (3.,2.5)-- (3.,-0.5);
\draw [line width=1pt] (0.5,2.5)-- (0.5,-0.5);
\draw [line width=1pt] (1.5,-0.5)-- (1.5,2.5);
\draw [line width=1pt] (2.5,2.5)-- (2.5,-0.5);
\draw [line width=1pt] (0.,2.5)-- (3.,2.5);
\draw [line width=1pt] (3.,1.5)-- (0.,1.5);
\draw [line width=1pt] (0.,0.5)-- (3.,0.5);
\draw [line width=1pt] (3.,-0.5)-- (0.,-0.5);

\draw [line width=1pt] (0.75,0)--(0.75,0.5);     
\draw [line width=1pt] (1.75,2)--(1.75,2.5);     
\draw [line width=1pt] (1.25,1.5)--(1.25,2);     
\draw [line width=1pt] (0.25,-0.5)--(0.25,0);    

\definecolor{pointblue}{RGB}{28,55,140}
\definecolor{pointgreen}{RGB}{35,100,45}

\draw [line width=1pt, color=violet] (0.14,-0.55)--(1.69,2.55);

\fill[pointblue] (0.62,0.41) circle (0.03);
\fill[pointblue] (1.62,2.41) circle (0.03);

\draw (0.63,0.41) node[anchor=south,text=black,yshift=8] {$P$};
\draw (1.63,2.41) node[anchor=south,text=black,yshift=8] {$P'$};

\draw[color=pointblue] (0.8,0.25) node[anchor=center] {$B$};

\fill[pointgreen] (0.295,-0.24) circle (0.03);   
\fill[pointgreen] (1.295,1.76) circle (0.03);    
\fill[pointred]  (1.295,0.76) circle (0.03);     
\fill[pointred]  (0.295,0.76) circle (0.03);     

\draw (0.330,-0.24) node[anchor=east,text=black,xshift=-4] {$Q_{0,1}$};
\draw (1.330,1.76) node[anchor=east,text=black,xshift=-4] {$Q_{1,3}$};
\draw (1.305,0.76) node[anchor=south,text=black] {$Q_{1,2}$};
\draw (0.305,0.76) node[anchor=south,text=black] {$Q_{0,2}$};

\draw (0,-0.65) node[anchor=center] {$0$};
\draw (0.5,-0.65) node[anchor=center] {$\frac12p$};
\draw (1,-0.65) node[anchor=center] {$p$};
\draw (1.5,-0.65) node[anchor=center] {$\frac32p$};
\draw (2,-0.65) node[anchor=center] {$2p$};
\draw (2.5,-0.65) node[anchor=center] {$\frac52p$};
\draw (3,-0.65) node[anchor=center] {$3p$};
\draw (-0.19,-0.5) node[anchor=center] {$-\frac12p$};
\draw (-0.19,0) node[anchor=center] {$0$};
\draw (-0.19,0.5) node[anchor=center] {$\frac12p$};
\draw (-0.19,1) node[anchor=center] {$p$};
\draw (-0.19,1.5) node[anchor=center] {$\frac32p$};
\draw (-0.19,2) node[anchor=center] {$2p$};
\draw (-0.19,2.5) node[anchor=center] {$\frac52p$};

\end{tikzpicture}
        \captionsetup{skip=3pt}
		\captionof{figure}{The points $P$, $P'$, $Q_{0,1}$, $Q_{1,3}$ form a $(+2)$-equivalence class. Even though the two red points are also congruent to $Q_{0,1}$ and $Q_{1,3}$, they do not belong to this equivalence class. The two green and two blue regions form a $(+2)$-equivalent family. Note that each region only contains points whose $y$-coordinate has a fixed mod $p$ parity: odd for single, even for double horizontal shading.}
		\label{fig:slope2_fig1}
\end{figure}

Consider the point $P=(x_0,y_0)$, and its congruent copy $P'=(x_0+p, y_0+2p)$. Note that the congruence class paired to $(x_0,y_0)$ is $\left(-\frac{y_0}{2}, -2x_0\right)$, which consists of all points of the form $Q_{\ell_1, \ell_2}:=\left(\frac{p-y_0}{2}+\ell_1 p, -2x_0+\ell_2 p\right)$ with $\ell_1, \ell_2\in \zz$, keeping in mind that $y_0$ is odd. However not all such points are on the same slope $+2$ line as $P$ and $P'$. Indeed, the difference vector $\left(\frac{p-y_0}{2}+\ell_1 p-x_0, -2x_0+\ell_2 p-y_0\right)$ must be a scalar multiple of $(1,2)$, which holds if and only if $\ell_2=2\ell_1+1$. So $Q_{0,1}$ and $Q_{1,3}$ lie on this line, whereas other points of $\cG^{+}$ congruent to them, such as the two red points with $Q_{0,2}$ and $Q_{1,2}$, do not.

It is easy to see that as $P$ ranges over $B$, the points $Q_{0,1}$ and $Q_{1,3}$ range over the two green regions marked in the figure (attaining all points whose $y$-coordinates are even mod $p$). Therefore, the two green and two blue regions form a family of $(+2)$-equivalent regions.

Using a similar algebraic reasoning, one may obtain full families of $(+2)$-equivalent regions covering every block that intersects the set $T_1^*$ shown in \autoref{fig:slope1_fig2}: see \autoref{fig:slope2_fig34} and \autoref{fig:slope2_fig56}. Note that some regions, such as Region 5 in \autoref{fig:slope2_fig34}, are mapped to themselves, however, noting the change of orientation in \autoref{fig:class_pairing_slope2}, one can see that most points in such regions are actually mapped to incongruent points in the same region, except for one diagonal segment in every such region which is pointwise fixed.

Using \autoref{fig:slope2_fig34} and \autoref{fig:slope2_fig56}, as well as the change of orientation under the pairing map, we can now depict those points in $T_1^*$ which have four $(+2)$-equivalent copies in $\cG^+$ overall. These points and their regionwise equivalent images are shown in \autoref{fig:slope2_fig78}.

To obtain the $(+2)$-equivalence classes consisting of four points of $T_1$, we need to depict the overlap between the shaded regions in the two parts of \autoref{fig:slope2_fig78}. This is shown in \autoref{fig:slope2_fig9}, along with a full partitioning of all such areas into families of $(+2)$-equivalent regions, as obtained using the previous figures.

By following an analogous procedure for slope $+\frac12$ lines, we can similarly obtain a partitioning with respect to $\left(+\frac12\right)$-equivalence, see \autoref{fig:slopehalf_fig9}. Using symmetry in the line $x=\frac32p$, the partitionings with respect to $m$-equivalence for $m=-2$ and $m=-\frac12$ may be obtained by reflecting the partitionings for $m=+2$ and $m=+\frac12$ in this line.

To obtain our set $T$, we choose a set of points to be removed from $T_1$ so that for all four partitionings, at least one region is fully erased from each family of $m$-equivalent regions. A way to do this is to remove all points indicated in \autoref{fig:slope2_removed} and \autoref{fig:slopehalf_removed}, taking care of $m=\pm 2$ and $m=\pm \frac12$, respectively. In several places, we chose to remove overlapping areas in the two figures to reduce the number of points overall removed. The remaining set $T$ was depicted in \autoref{sec:k=3}, in \autoref{fig:finalset}.

The set $T_1$ contains $\frac{13}{2}p^2-O(p)$ normal points (see \autoref{fig:slope1_fig1}), whereas in \autoref{fig:slope2_removed} and \autoref{fig:slopehalf_removed}, the number of normal points in the removed regions is $\frac{11}{48}p^2+O(p)$ and $\frac{23}{60}p^2+O(p)$, respectively. However, the two regions intersect in $\frac{29}{896}p^2-O(p)$ normal points, so all in all, $T$ contains

$$\left(\frac{13}{2}-\frac{11}{48}-\frac{23}{60}+\frac{29}{896}\right)p^2-O(p)=\frac{26521}{4480}p^2-O(p)$$
normal points. As our no-$4$-in-line set $T'$ is constructed as $T':=T\cap H(c,p)$, there exists a value $c\in \F_p^{*}$ for which $$|T'|\ge \frac{26521}{4480}p-O(1)=\frac{26521}{13440}n-O(1)\ge 1.97328n-O(1).$$ This concludes the proof of \autoref{k=3}, as for sufficiently large $n$ we can pick an odd prime $p$ such that $3p$ asymptotically equals (but does not exceed) $n$.

\vfill

\begin{figure}[h!!]
	\centering
	\begin{minipage}{0.47\textwidth}
		\centering
		\definecolor{shadeone}{RGB}{200,70,70}     
\definecolor{shadefour}{RGB}{55,115,205}    
\definecolor{shadethree}{RGB}{60,150,85}   
\definecolor{shadetwo}{RGB}{220,145,45}   
\definecolor{shadefive}{RGB}{140,95,200}   
\definecolor{shadesix}{RGB}{45,165,155}    
\definecolor{shadeeight}{RGB}{185,165,35}  
\definecolor{shadeseven}{RGB}{200,85,145}  

\begin{tikzpicture}[line cap=round,line join=round,>=triangle 45,x=0.19cm,y=0.19cm,  scale=0.85, transform shape]

\clip(-4.,-10.0) rectangle (41.,34.5);

\def\NumStyle{\bfseries\fontsize{9}{9}\selectfont}

\def\HalfLeft#1#2#3#4{%
  \fill[pattern=horizontalodd, pattern color=#3] (#1,#2) rectangle ({#1+3.25},{#2+6.5});
  \node[font=\NumStyle] at ({#1+1.625},{#2+3.25}) {#4};
}
\def\HalfRight#1#2#3#4{%
  \fill[pattern=horizontalodd, pattern color=#3] ({#1+3.25},#2) rectangle ({#1+6.5},{#2+6.5});
  \node[font=\NumStyle] at ({#1+4.875},{#2+3.25}) {#4};
}


\HalfLeft{13}{26}{shadeone}{1}    
\HalfLeft{26}{26}{shadeone}{9}    
\HalfLeft{0}{0}{shadeone}{1}      
\HalfLeft{13}{0}{shadeone}{9}     

\HalfRight{13}{26}{shadetwo}{2}   
\HalfRight{26}{26}{shadetwo}{10}   
\HalfRight{0}{0}{shadetwo}{2}     
\HalfRight{13}{0}{shadetwo}{10}    

\HalfLeft{19.5}{26}{shadethree}{3}  
\HalfLeft{32.5}{26}{shadethree}{11}  
\HalfLeft{6.5}{0}{shadethree}{3}    
\HalfLeft{19.5}{0}{shadethree}{11}   

\HalfRight{19.5}{26}{shadefour}{4}  
\HalfRight{32.5}{26}{shadefour}{12}  
\HalfRight{6.5}{0}{shadefour}{4}    
\HalfRight{19.5}{0}{shadefour}{12}   

\HalfLeft{13}{19.5}{shadefive}{5}   
\HalfLeft{26}{19.5}{shadefive}{13}   
\HalfLeft{0}{-6.5}{shadefive}{5}    
\HalfLeft{13}{-6.5}{shadefive}{13}   

\HalfRight{13}{19.5}{shadesix}{6}   
\HalfRight{26}{19.5}{shadesix}{14}   
\HalfRight{0}{-6.5}{shadesix}{6}    
\HalfRight{13}{-6.5}{shadesix}{14}   

\HalfLeft{19.5}{19.5}{shadeseven}{7}  
\HalfLeft{32.5}{19.5}{shadeseven}{15}  
\HalfLeft{6.5}{-6.5}{shadeseven}{7}   
\HalfLeft{19.5}{-6.5}{shadeseven}{15}  

\HalfRight{19.5}{19.5}{shadeeight}{8} 
\HalfRight{32.5}{19.5}{shadeeight}{16} 
\HalfRight{6.5}{-6.5}{shadeeight}{8}  
\HalfRight{19.5}{-6.5}{shadeeight}{16} 

\draw [line width=2pt] (0.,32.5)-- (0.,-6.5);
\draw [line width=2pt] (13.,-6.5)-- (13.,32.5);
\draw [line width=2pt] (26.,32.5)-- (26.,-6.5);
\draw [line width=2pt] (0.,0.)-- (39.,0.);
\draw [line width=2pt] (0.,26.)-- (39.,26.);
\draw [line width=2pt] (39.,13.)-- (0.,13.);
\draw [line width=2pt] (39.,32.5)-- (39.,-6.5);
\draw [line width=1pt] (6.5,32.5)-- (6.5,-6.5);
\draw [line width=1pt] (19.5,-6.5)-- (19.5,32.5);
\draw [line width=1pt] (32.5,32.5)-- (32.5,-6.5);
\draw [line width=1pt] (0.,32.5)-- (39.,32.5);
\draw [line width=1pt] (39.,19.5)-- (0.,19.5);
\draw [line width=1pt] (0.,6.5)-- (39.,6.5);
\draw [line width=1pt] (39.,-6.5)-- (0.,-6.5);

\draw [line width=1pt] (3.25,-6.5)--(3.25,0);     
\draw [line width=1pt] (9.75,-6.5)--(9.75,0);     
\draw [line width=1pt] (16.25,-6.5)--(16.25,0);   
\draw [line width=1pt] (22.75,-6.5)--(22.75,0);   

\draw [line width=1pt] (3.25,0)--(3.25,6.5);      
\draw [line width=1pt] (9.75,0)--(9.75,6.5);      
\draw [line width=1pt] (16.25,0)--(16.25,6.5);    
\draw [line width=1pt] (22.75,0)--(22.75,6.5);    

\draw [line width=1pt] (16.25,19.5)--(16.25,26);  
\draw [line width=1pt] (22.75,19.5)--(22.75,26);  
\draw [line width=1pt] (29.25,19.5)--(29.25,26);  
\draw [line width=1pt] (35.75,19.5)--(35.75,26);  

\draw [line width=1pt] (16.25,26)--(16.25,32.5);  
\draw [line width=1pt] (22.75,26)--(22.75,32.5);  
\draw [line width=1pt] (29.25,26)--(29.25,32.5);  
\draw [line width=1pt] (35.75,26)--(35.75,32.5);  

\draw (0,-8.5) node[anchor=center] {$0$};
\draw (6.5,-8.5) node[anchor=center] {$\frac12p$};
\draw (13,-8.5) node[anchor=center] {$p$};
\draw (19.5,-8.5) node[anchor=center] {$\frac32p$};
\draw (26,-8.5) node[anchor=center] {$2p$};
\draw (32.5,-8.5) node[anchor=center] {$\frac52p$};
\draw (39,-8.5) node[anchor=center] {$3p$};
\draw (-2.5,-6.5) node[anchor=center] {$-\frac12p$};
\draw (-2.5,0) node[anchor=center] {$0$};
\draw (-2.5,6.5) node[anchor=center] {$\frac12p$};
\draw (-2.5,13) node[anchor=center] {$p$};
\draw (-2.5,19.5) node[anchor=center] {$\frac32p$};
\draw (-2.5,26) node[anchor=center] {$2p$};
\draw (-2.5,32.5) node[anchor=center] {$\frac52p$};
\end{tikzpicture}
	\end{minipage}%
	\begin{minipage}{0.06\textwidth}
    \centering
    \raisebox{0.2cm}{\scalebox{2.5}{$\rightarrow$}}
    \end{minipage}%
	\begin{minipage}{0.47\textwidth}
		\centering
		\definecolor{shadeone}{RGB}{200,70,70}     
\definecolor{shadefour}{RGB}{55,115,205}    
\definecolor{shadethree}{RGB}{60,150,85}   
\definecolor{shadetwo}{RGB}{220,145,45}   
\definecolor{shadefive}{RGB}{140,95,200}   
\definecolor{shadesix}{RGB}{45,165,155}    
\definecolor{shadeeight}{RGB}{185,165,35}  
\definecolor{shadeseven}{RGB}{200,85,145}  

\begin{tikzpicture}[line cap=round,line join=round,>=triangle 45,x=0.19cm,y=0.19cm, scale=0.85, transform shape]
\clip(-4.,-10.0) rectangle (41.,34.5);

\def\NumStyle{\bfseries\fontsize{9}{9}\selectfont}

\def\HalfLeftOdd#1#2#3#4{%
  \fill[pattern=horizontalodd, pattern color=#3] (#1,#2) rectangle ({#1+3.25},{#2+6.5});
  \node[font=\NumStyle] at ({#1+1.625},{#2+3.25}) {#4};
}
\def\HalfRightOdd#1#2#3#4{%
  \fill[pattern=horizontalodd, pattern color=#3] ({#1+3.25},#2) rectangle ({#1+6.5},{#2+6.5});
  \node[font=\NumStyle] at ({#1+4.875},{#2+3.25}) {#4};
}
\def\HalfLeftEven#1#2#3#4{%
  \fill[pattern=horizontaleven, pattern color=#3] (#1,#2) rectangle ({#1+3.25},{#2+6.5});
  \node[font=\NumStyle] at ({#1+1.625},{#2+3.25}) {#4};
}
\def\HalfRightEven#1#2#3#4{%
  \fill[pattern=horizontaleven, pattern color=#3] ({#1+3.25},#2) rectangle ({#1+6.5},{#2+6.5});
  \node[font=\NumStyle] at ({#1+4.875},{#2+3.25}) {#4};
}


\HalfRightOdd{0}{6.5}{shadeone}{1}      
\HalfRightOdd{13}{6.5}{shadeone}{9}     

\HalfRightOdd{13}{26}{shadetwo}{2}      
\HalfRightOdd{26}{26}{shadetwo}{10}      
\HalfRightOdd{0}{0}{shadetwo}{2}        
\HalfRightOdd{13}{0}{shadetwo}{10}       

\HalfRightEven{13}{19.5}{shadethree}{3} 
\HalfRightEven{26}{19.5}{shadethree}{11} 
\HalfRightEven{0}{-6.5}{shadethree}{3}  
\HalfRightEven{13}{-6.5}{shadethree}{11} 

\HalfRightEven{13}{13}{shadefour}{4}    
\HalfRightEven{26}{13}{shadefour}{12}    

\HalfLeftOdd{13}{19.5}{shadefive}{5}    
\HalfLeftOdd{26}{19.5}{shadefive}{13}    
\HalfLeftOdd{0}{-6.5}{shadefive}{5}     
\HalfLeftOdd{13}{-6.5}{shadefive}{13}    

\HalfLeftOdd{13}{13}{shadesix}{6}       
\HalfLeftOdd{26}{13}{shadesix}{14}       

\HalfLeftEven{13}{6.5}{shadeseven}{7}   
\HalfLeftEven{26}{6.5}{shadeseven}{15}   

\HalfLeftEven{26}{26}{shadeeight}{16}    
\HalfLeftEven{13}{0}{shadeeight}{8}     
\HalfLeftEven{26}{0}{shadeeight}{16}     

\draw [line width=2pt] (0.,32.5)-- (0.,-6.5);
\draw [line width=2pt] (13.,-6.5)-- (13.,32.5);
\draw [line width=2pt] (26.,32.5)-- (26.,-6.5);
\draw [line width=2pt] (0.,0.)-- (39.,0.);
\draw [line width=2pt] (0.,26.)-- (39.,26.);
\draw [line width=2pt] (39.,13.)-- (0.,13.);
\draw [line width=2pt] (39.,32.5)-- (39.,-6.5);
\draw [line width=1pt] (6.5,32.5)-- (6.5,-6.5);
\draw [line width=1pt] (19.5,-6.5)-- (19.5,32.5);
\draw [line width=1pt] (32.5,32.5)-- (32.5,-6.5);
\draw [line width=1pt] (0.,32.5)-- (39.,32.5);
\draw [line width=1pt] (39.,19.5)-- (0.,19.5);
\draw [line width=1pt] (0.,6.5)-- (39.,6.5);
\draw [line width=1pt] (39.,-6.5)-- (0.,-6.5);

\draw [line width=1pt] (3.25,-6.5)--(3.25,0);      
\draw [line width=1pt] (16.25,-6.5)--(16.25,0);    

\draw [line width=1pt] (3.25,0)--(3.25,6.5);       
\draw [line width=1pt] (16.25,0)--(16.25,6.5);     
\draw [line width=1pt] (29.25,0)--(29.25,6.5);     

\draw [line width=1pt] (16.25,6.5)--(16.25,13);    
\draw [line width=1pt] (29.25,6.5)--(29.25,13);    
\draw [line width=1pt] (3.25,6.5)--(3.25,13);      

\draw [line width=1pt] (16.25,13)--(16.25,19.5);   
\draw [line width=1pt] (29.25,13)--(29.25,19.5);   

\draw [line width=1pt] (16.25,19.5)--(16.25,26);   
\draw [line width=1pt] (29.25,19.5)--(29.25,26);   

\draw [line width=1pt] (16.25,26)--(16.25,32.5);   
\draw [line width=1pt] (29.25,26)--(29.25,32.5);   

\draw (0,-8.5) node[anchor=center] {$0$};
\draw (6.5,-8.5) node[anchor=center] {$\frac12p$};
\draw (13,-8.5) node[anchor=center] {$p$};
\draw (19.5,-8.5) node[anchor=center] {$\frac32p$};
\draw (26,-8.5) node[anchor=center] {$2p$};
\draw (32.5,-8.5) node[anchor=center] {$\frac52p$};
\draw (39,-8.5) node[anchor=center] {$3p$};
\draw (-2.5,-6.5) node[anchor=center] {$-\frac12p$};
\draw (-2.5,0) node[anchor=center] {$0$};
\draw (-2.5,6.5) node[anchor=center] {$\frac12p$};
\draw (-2.5,13) node[anchor=center] {$p$};
\draw (-2.5,19.5) node[anchor=center] {$\frac32p$};
\draw (-2.5,26) node[anchor=center] {$2p$};
\draw (-2.5,32.5) node[anchor=center] {$\frac52p$};
\end{tikzpicture}
	\end{minipage}
 \captionsetup{skip=3pt}
     \caption{The left figure shows regions covering all points of $T_1^*$ having odd $y$-coordinate mod $p$. The right figure shows the regions $(+2)$-equivalent to those with the same number on the left, belonging to the image of the congruence class under our pairing map.}
     \vspace{1.2cm}
     \label{fig:slope2_fig34}
\end{figure}
\begin{figure}[h!!]
	\centering
	\begin{minipage}{0.47\textwidth}
		\centering
		\definecolor{shadeone}{RGB}{200,70,70}     
\definecolor{shadefour}{RGB}{55,115,205}    
\definecolor{shadethree}{RGB}{60,150,85}   
\definecolor{shadetwo}{RGB}{220,145,45}   
\definecolor{shadefive}{RGB}{140,95,200}   
\definecolor{shadesix}{RGB}{45,165,155}    
\definecolor{shadeeight}{RGB}{185,165,35}  
\definecolor{shadeseven}{RGB}{200,85,145}  

\begin{tikzpicture}[line cap=round,line join=round,>=triangle 45,x=0.19cm,y=0.19cm, scale=0.85, transform shape]
\clip(-4.,-10.0) rectangle (41.,34.5);

\def\NumStyle{\bfseries\fontsize{9}{9}\selectfont}

\def\HalfLeft#1#2#3#4{%
  \fill[pattern=horizontaleven, pattern color=#3] (#1,#2) rectangle ({#1+3.25},{#2+6.5});
  \node[font=\NumStyle] at ({#1+1.625},{#2+3.25}) {#4};
}
\def\HalfRight#1#2#3#4{%
  \fill[pattern=horizontaleven, pattern color=#3] ({#1+3.25},#2) rectangle ({#1+6.5},{#2+6.5});
  \node[font=\NumStyle] at ({#1+4.875},{#2+3.25}) {#4};
}


\HalfLeft{13}{26}{shadeone}{1}    
\HalfLeft{26}{26}{shadeone}{9}    
\HalfLeft{0}{0}{shadeone}{1}      
\HalfLeft{13}{0}{shadeone}{9}     

\HalfRight{13}{26}{shadetwo}{2}   
\HalfRight{26}{26}{shadetwo}{10}   
\HalfRight{0}{0}{shadetwo}{2}     
\HalfRight{13}{0}{shadetwo}{10}    

\HalfLeft{19.5}{26}{shadethree}{3}  
\HalfLeft{32.5}{26}{shadethree}{11}  
\HalfLeft{6.5}{0}{shadethree}{3}    
\HalfLeft{19.5}{0}{shadethree}{11}   

\HalfRight{19.5}{26}{shadefour}{4}  
\HalfRight{32.5}{26}{shadefour}{12}  
\HalfRight{6.5}{0}{shadefour}{4}    
\HalfRight{19.5}{0}{shadefour}{12}   

\HalfLeft{13}{19.5}{shadefive}{5}   
\HalfLeft{26}{19.5}{shadefive}{13}   
\HalfLeft{0}{-6.5}{shadefive}{5}    
\HalfLeft{13}{-6.5}{shadefive}{13}   

\HalfRight{13}{19.5}{shadesix}{6}   
\HalfRight{26}{19.5}{shadesix}{14}   
\HalfRight{0}{-6.5}{shadesix}{6}    
\HalfRight{13}{-6.5}{shadesix}{14}   

\HalfLeft{19.5}{19.5}{shadeseven}{7}  
\HalfLeft{32.5}{19.5}{shadeseven}{15}  
\HalfLeft{6.5}{-6.5}{shadeseven}{7}   
\HalfLeft{19.5}{-6.5}{shadeseven}{15}  

\HalfRight{19.5}{19.5}{shadeeight}{8} 
\HalfRight{32.5}{19.5}{shadeeight}{16} 
\HalfRight{6.5}{-6.5}{shadeeight}{8}  
\HalfRight{19.5}{-6.5}{shadeeight}{16} 

\draw [line width=2pt] (0.,32.5)-- (0.,-6.5);
\draw [line width=2pt] (13.,-6.5)-- (13.,32.5);
\draw [line width=2pt] (26.,32.5)-- (26.,-6.5);
\draw [line width=2pt] (0.,0.)-- (39.,0.);
\draw [line width=2pt] (0.,26.)-- (39.,26.);
\draw [line width=2pt] (39.,13.)-- (0.,13.);
\draw [line width=2pt] (39.,32.5)-- (39.,-6.5);
\draw [line width=1pt] (6.5,32.5)-- (6.5,-6.5);
\draw [line width=1pt] (19.5,-6.5)-- (19.5,32.5);
\draw [line width=1pt] (32.5,32.5)-- (32.5,-6.5);
\draw [line width=1pt] (0.,32.5)-- (39.,32.5);
\draw [line width=1pt] (39.,19.5)-- (0.,19.5);
\draw [line width=1pt] (0.,6.5)-- (39.,6.5);
\draw [line width=1pt] (39.,-6.5)-- (0.,-6.5);

\draw [line width=1pt] (3.25,-6.5)--(3.25,0);     
\draw [line width=1pt] (9.75,-6.5)--(9.75,0);     
\draw [line width=1pt] (16.25,-6.5)--(16.25,0);   
\draw [line width=1pt] (22.75,-6.5)--(22.75,0);   

\draw [line width=1pt] (3.25,0)--(3.25,6.5);      
\draw [line width=1pt] (9.75,0)--(9.75,6.5);      
\draw [line width=1pt] (16.25,0)--(16.25,6.5);    
\draw [line width=1pt] (22.75,0)--(22.75,6.5);    

\draw [line width=1pt] (16.25,19.5)--(16.25,26);  
\draw [line width=1pt] (22.75,19.5)--(22.75,26);  
\draw [line width=1pt] (29.25,19.5)--(29.25,26);  
\draw [line width=1pt] (35.75,19.5)--(35.75,26);  

\draw [line width=1pt] (16.25,26)--(16.25,32.5);  
\draw [line width=1pt] (22.75,26)--(22.75,32.5);  
\draw [line width=1pt] (29.25,26)--(29.25,32.5);  
\draw [line width=1pt] (35.75,26)--(35.75,32.5);  

\draw (0,-8.5) node[anchor=center] {$0$};
\draw (6.5,-8.5) node[anchor=center] {$\frac12p$};
\draw (13,-8.5) node[anchor=center] {$p$};
\draw (19.5,-8.5) node[anchor=center] {$\frac32p$};
\draw (26,-8.5) node[anchor=center] {$2p$};
\draw (32.5,-8.5) node[anchor=center] {$\frac52p$};
\draw (39,-8.5) node[anchor=center] {$3p$};
\draw (-2.5,-6.5) node[anchor=center] {$-\frac12p$};
\draw (-2.5,0) node[anchor=center] {$0$};
\draw (-2.5,6.5) node[anchor=center] {$\frac12p$};
\draw (-2.5,13) node[anchor=center] {$p$};
\draw (-2.5,19.5) node[anchor=center] {$\frac32p$};
\draw (-2.5,26) node[anchor=center] {$2p$};
\draw (-2.5,32.5) node[anchor=center] {$\frac52p$};
\end{tikzpicture}
	\end{minipage}%
	\begin{minipage}{0.06\textwidth}
    \centering
    \raisebox{0.2cm}{\scalebox{2.5}{$\rightarrow$}}
    \end{minipage}%
	\begin{minipage}{0.47\textwidth}
		\centering
		\definecolor{shadeone}{RGB}{200,70,70}     
\definecolor{shadefour}{RGB}{55,115,205}    
\definecolor{shadethree}{RGB}{60,150,85}   
\definecolor{shadetwo}{RGB}{220,145,45}   
\definecolor{shadefive}{RGB}{140,95,200}   
\definecolor{shadesix}{RGB}{45,165,155}    
\definecolor{shadeeight}{RGB}{185,165,35}  
\definecolor{shadeseven}{RGB}{200,85,145}  

\begin{tikzpicture}[line cap=round,line join=round,>=triangle 45,x=0.19cm,y=0.19cm, scale=0.85, transform shape] 
\clip(-4.,-10.0) rectangle (41.,34.5);

\def\NumStyle{\bfseries\fontsize{9}{9}\selectfont}

\def\HalfLeftOdd#1#2#3#4{%
  \fill[pattern=horizontalodd, pattern color=#3] (#1,#2) rectangle ({#1+3.25},{#2+6.5});
  \node[font=\NumStyle] at ({#1+1.625},{#2+3.25}) {#4};
}
\def\HalfRightOdd#1#2#3#4{%
  \fill[pattern=horizontalodd, pattern color=#3] ({#1+3.25},#2) rectangle ({#1+6.5},{#2+6.5});
  \node[font=\NumStyle] at ({#1+4.875},{#2+3.25}) {#4};
}
\def\HalfLeftEven#1#2#3#4{%
  \fill[pattern=horizontaleven, pattern color=#3] (#1,#2) rectangle ({#1+3.25},{#2+6.5});
  \node[font=\NumStyle] at ({#1+1.625},{#2+3.25}) {#4};
}
\def\HalfRightEven#1#2#3#4{%
  \fill[pattern=horizontaleven, pattern color=#3] ({#1+3.25},#2) rectangle ({#1+6.5},{#2+6.5});
  \node[font=\NumStyle] at ({#1+4.875},{#2+3.25}) {#4};
}


\HalfRightOdd{6.5}{19.5}{shadeone}{1}     
\HalfRightOdd{19.5}{19.5}{shadeone}{9}    
\HalfRightOdd{6.5}{-6.5}{shadeone}{9}     

\HalfRightOdd{6.5}{13}{shadetwo}{2}       
\HalfRightOdd{19.5}{13}{shadetwo}{10}      

\HalfRightEven{6.5}{6.5}{shadethree}{3}   
\HalfRightEven{19.5}{6.5}{shadethree}{11}  

\HalfRightEven{19.5}{26}{shadefour}{4}    
\HalfRightEven{32.5}{26}{shadefour}{12}    
\HalfRightEven{6.5}{0}{shadefour}{4}      
\HalfRightEven{19.5}{0}{shadefour}{12}     

\HalfLeftOdd{6.5}{6.5}{shadefive}{5}      
\HalfLeftOdd{19.5}{6.5}{shadefive}{13}     

\HalfLeftOdd{19.5}{26}{shadesix}{6}       
\HalfLeftOdd{32.5}{26}{shadesix}{14}       
\HalfLeftOdd{6.5}{0}{shadesix}{6}         
\HalfLeftOdd{19.5}{0}{shadesix}{14}        

\HalfLeftEven{19.5}{19.5}{shadeseven}{7}  
\HalfLeftEven{32.5}{19.5}{shadeseven}{15}  
\HalfLeftEven{6.5}{-6.5}{shadeseven}{7}   
\HalfLeftEven{19.5}{-6.5}{shadeseven}{15}  

\HalfLeftEven{19.5}{13}{shadeeight}{8}    
\HalfLeftEven{32.5}{13}{shadeeight}{16}    

\draw [line width=2pt] (0.,32.5)-- (0.,-6.5);
\draw [line width=2pt] (13.,-6.5)-- (13.,32.5);
\draw [line width=2pt] (26.,32.5)-- (26.,-6.5);
\draw [line width=2pt] (0.,0.)-- (39.,0.);
\draw [line width=2pt] (0.,26.)-- (39.,26.);
\draw [line width=2pt] (39.,13.)-- (0.,13.);
\draw [line width=2pt] (39.,32.5)-- (39.,-6.5);
\draw [line width=1pt] (6.5,32.5)-- (6.5,-6.5);
\draw [line width=1pt] (19.5,-6.5)-- (19.5,32.5);
\draw [line width=1pt] (32.5,32.5)-- (32.5,-6.5);
\draw [line width=1pt] (0.,32.5)-- (39.,32.5);
\draw [line width=1pt] (39.,19.5)-- (0.,19.5);
\draw [line width=1pt] (0.,6.5)-- (39.,6.5);
\draw [line width=1pt] (39.,-6.5)-- (0.,-6.5);

\draw [line width=1pt] (9.75,-6.5)--(9.75,0);      
\draw [line width=1pt] (22.75,-6.5)--(22.75,0);    

\draw [line width=1pt] (9.75,0)--(9.75,6.5);       
\draw [line width=1pt] (22.75,0)--(22.75,6.5);     

\draw [line width=1pt] (9.75,6.5)--(9.75,13);      
\draw [line width=1pt] (22.75,6.5)--(22.75,13);    

\draw [line width=1pt] (9.75,13)--(9.75,19.5);     
\draw [line width=1pt] (22.75,13)--(22.75,19.5);   
\draw [line width=1pt] (35.75,13)--(35.75,19.5);   

\draw [line width=1pt] (9.75,19.5)--(9.75,26);     
\draw [line width=1pt] (22.75,19.5)--(22.75,26);   
\draw [line width=1pt] (35.75,19.5)--(35.75,26);   

\draw [line width=1pt] (22.75,26)--(22.75,32.5);   
\draw [line width=1pt] (35.75,26)--(35.75,32.5);   

\draw (0,-8.5) node[anchor=center] {$0$};
\draw (6.5,-8.5) node[anchor=center] {$\frac12p$};
\draw (13,-8.5) node[anchor=center] {$p$};
\draw (19.5,-8.5) node[anchor=center] {$\frac32p$};
\draw (26,-8.5) node[anchor=center] {$2p$};
\draw (32.5,-8.5) node[anchor=center] {$\frac52p$};
\draw (39,-8.5) node[anchor=center] {$3p$};
\draw (-2.5,-6.5) node[anchor=center] {$-\frac12p$};
\draw (-2.5,0) node[anchor=center] {$0$};
\draw (-2.5,6.5) node[anchor=center] {$\frac12p$};
\draw (-2.5,13) node[anchor=center] {$p$};
\draw (-2.5,19.5) node[anchor=center] {$\frac32p$};
\draw (-2.5,26) node[anchor=center] {$2p$};
\draw (-2.5,32.5) node[anchor=center] {$\frac52p$};
\end{tikzpicture}
	\end{minipage}
 \captionsetup{skip=3pt}
     \caption{The left figure shows regions covering all points of $T_1^*$ having even $y$-coordinate mod $p$. The right figure shows the regions $(+2)$-equivalent to those with the same number on the left, belonging to the image of the congruence class under our pairing map.}
     \label{fig:slope2_fig56}
\end{figure}

\begin{figure}[h!!]
	\centering
	\begin{minipage}{0.47\textwidth}
		\centering
		\input{figures/slope2_fig7}
	\end{minipage}%
	\begin{minipage}{0.06\textwidth}
    \centering
    \raisebox{0.2cm}{\scalebox{2.5}{$\rightarrow$}}
    \end{minipage}%
	\begin{minipage}{0.47\textwidth}
		\centering
		\definecolor{shadeone}{RGB}{200,70,70}     
\definecolor{shadefour}{RGB}{55,115,205}    
\definecolor{shadethree}{RGB}{60,150,85}   
\definecolor{shadetwo}{RGB}{220,145,45}   
\definecolor{shadefive}{RGB}{140,95,200}   
\definecolor{shadesix}{RGB}{45,165,155}    
\definecolor{shadeeight}{RGB}{185,165,35}  
\definecolor{shadeseven}{RGB}{200,85,145}  

\begin{tikzpicture}[line cap=round,line join=round,>=triangle 45,x=0.19cm,y=0.19cm, scale=0.85, transform shape]
\clip(-4.,-10.0) rectangle (41.,34.5);

\def\NumStyle{\bfseries\fontsize{9}{9}\selectfont}


\fill[pattern=horizontalodd, pattern color=shadeone]
  (22.75,26) -- (26,26) -- (24.375,19.5) -- (22.75,19.5) -- cycle;
\node[font=\NumStyle] at (24.0139,22.75) {9};

\fill[pattern=horizontalodd, pattern color=shadeone]
  (9.75,0) -- (13,0) -- (11.375,-6.5) -- (9.75,-6.5) -- cycle;
\node[font=\NumStyle] at (11.0139,-3.25) {9};

\fill[pattern=horizontalodd, pattern color=shadetwo]
  (16.25,26) rectangle (19.5,32.5);
\node[font=\NumStyle] at (17.875,29.25) {2};

\fill[pattern=horizontalodd, pattern color=shadetwo]
  (29.25,32.5) -- (29.25,26) -- (30.875,32.5) -- cycle;
\node[font=\NumStyle] at (29.9,33.35) {10};

\fill[pattern=horizontalodd, pattern color=shadetwo]
  (3.25,0) rectangle (6.5,6.5);
\node[font=\NumStyle] at (4.875,3.25) {2};

\fill[pattern=horizontalodd, pattern color=shadetwo]
  (16.25,6.5) -- (16.25,0) -- (17.875,6.5) -- cycle;
\node[font=\NumStyle] at (16.9,7.35) {10};

\fill[pattern=horizontaleven, pattern color=shadethree]
  (16.25,26) -- (19.5,26) -- (19.5,19.5) -- (17.875,19.5) -- cycle;
\node[font=\NumStyle] at (18.2361,22.75) {3};

\fill[pattern=horizontaleven, pattern color=shadethree]
  (29.25,26) -- (29.25,19.5) -- (30.875,19.5) -- cycle;
\node[font=\NumStyle] at (30.95,24) {11};

\fill[pattern=horizontaleven, pattern color=shadethree]
  (3.25,0) -- (6.5,0) -- (6.5,-6.5) -- (4.875,-6.5) -- cycle;
\node[font=\NumStyle] at (5.2361,-3.25) {3};

\fill[pattern=horizontaleven, pattern color=shadethree]
  (16.25,0) -- (16.25,-6.5) -- (17.875,-6.5) -- cycle;
\node[font=\NumStyle] at (17.95,-2) {11};

\fill[pattern=horizontaleven, pattern color=shadefour]
  (24.375,32.5) -- (26,32.5) -- (26,26) -- cycle;
\node[font=\NumStyle] at (25.35,31.2) {4};

\fill[pattern=horizontaleven, pattern color=shadefour]
  (35.75,32.5) -- (37.375,32.5) -- (39,26) -- (35.75,26) -- cycle;
\node[font=\NumStyle] at (37.0139,29.25) {12};

\fill[pattern=horizontaleven, pattern color=shadefour]
  (11.375,6.5) -- (13,6.5) -- (13,0) -- cycle;
\node[font=\NumStyle] at (12.35,5.2) {4};

\fill[pattern=horizontaleven, pattern color=shadefour]
  (22.75,6.5) -- (24.375,6.5) -- (26,0) -- (22.75,0) -- cycle;
\node[font=\NumStyle] at (24.0139,3.25) {12};

\fill[pattern=horizontalodd, pattern color=shadefive]
  (13,26) -- (16.25,26) -- (16.25,19.5) -- (14.625,19.5) -- cycle;
\node[font=\NumStyle] at (14.9861,22.75) {5};

\fill[pattern=horizontalodd, pattern color=shadefive]
  (26,26) -- (26,19.5) -- (27.625,19.5) -- cycle;
\node[font=\NumStyle] at (27.6,24) {13};

\fill[pattern=horizontalodd, pattern color=shadefive]
  (0,0) -- (3.25,0) -- (3.25,-6.5) -- (1.625,-6.5) -- cycle;
\node[font=\NumStyle] at (1.9861,-3.25) {5};

\fill[pattern=horizontalodd, pattern color=shadefive]
  (13,0) -- (13,-6.5) -- (14.625,-6.5) -- cycle;
\node[font=\NumStyle] at (14.6,-2) {13};

\fill[pattern=horizontalodd, pattern color=shadesix]
  (21.125,32.5) -- (22.75,32.5) -- (22.75,26) -- cycle;
\node[font=\NumStyle] at (22.1,31.2) {6};

\fill[pattern=horizontalodd, pattern color=shadesix]
  (32.5,32.5) -- (34.125,32.5) -- (35.75,26) -- (32.5,26) -- cycle;
\node[font=\NumStyle] at (33.7639,29.25) {14};

\fill[pattern=horizontalodd, pattern color=shadesix]
  (8.125,6.5) -- (9.75,6.5) -- (9.75,0) -- cycle;
\node[font=\NumStyle] at (9.1,5.2) {6};

\fill[pattern=horizontalodd, pattern color=shadesix]
  (19.5,6.5) -- (21.125,6.5) -- (22.75,0) -- (19.5,0) -- cycle;
\node[font=\NumStyle] at (20.7639,3.25) {14};

\fill[pattern=horizontaleven, pattern color=shadeseven]
  (21.125,19.5) -- (22.75,19.5) -- (22.75,26) -- cycle;
\node[font=\NumStyle] at (22.1,20.8) {7};

\fill[pattern=horizontaleven, pattern color=shadeseven]
  (32.5,19.5) rectangle (35.75,26);
\node[font=\NumStyle] at (34.125,22.75) {15};

\fill[pattern=horizontaleven, pattern color=shadeseven]
  (8.125,-6.5) -- (9.75,-6.5) -- (9.75,0) -- cycle;
\node[font=\NumStyle] at (9.1,-5.2) {7};

\fill[pattern=horizontaleven, pattern color=shadeseven]
  (19.5,-6.5) rectangle (22.75,0);
\node[font=\NumStyle] at (21.125,-3.25) {15};

\fill[pattern=horizontaleven, pattern color=shadeeight]
  (26,26) -- (27.625,32.5) -- (29.25,32.5) -- (29.25,26) -- cycle;
\node[font=\NumStyle] at (27.9861,29.25) {8};

\fill[pattern=horizontaleven, pattern color=shadeeight]
  (13,0) -- (14.625,6.5) -- (16.25,6.5) -- (16.25,0) -- cycle;
\node[font=\NumStyle] at (14.9861,3.25) {8};

\draw [line width=2pt] (0.,32.5)-- (0.,-6.5);
\draw [line width=2pt] (13.,-6.5)-- (13.,32.5);
\draw [line width=2pt] (26.,32.5)-- (26.,-6.5);
\draw [line width=2pt] (0.,0.)-- (39.,0.);
\draw [line width=2pt] (0.,26.)-- (39.,26.);
\draw [line width=2pt] (39.,13.)-- (0.,13.);
\draw [line width=2pt] (39.,32.5)-- (39.,-6.5);
\draw [line width=1pt] (6.5,32.5)-- (6.5,-6.5);
\draw [line width=1pt] (19.5,-6.5)-- (19.5,32.5);
\draw [line width=1pt] (32.5,32.5)-- (32.5,-6.5);
\draw [line width=1pt] (0.,32.5)-- (39.,32.5);
\draw [line width=1pt] (39.,19.5)-- (0.,19.5);
\draw [line width=1pt] (0.,6.5)-- (39.,6.5);
\draw [line width=1pt] (39.,-6.5)-- (0.,-6.5);

\draw [line width=1pt] (3.25,0)--(3.25,6.5);       
\draw [line width=1pt] (9.75,0)--(9.75,6.5);       
\draw [line width=1pt] (16.25,0)--(16.25,6.5);     
\draw [line width=1pt] (22.75,0)--(22.75,6.5);     

\draw [line width=1pt] (3.25,-6.5)--(3.25,0);      
\draw [line width=1pt] (9.75,-6.5)--(9.75,0);      
\draw [line width=1pt] (16.25,-6.5)--(16.25,0);    
\draw [line width=1pt] (22.75,-6.5)--(22.75,0);    

\draw [line width=1pt] (16.25,19.5)--(16.25,26);   
\draw [line width=1pt] (22.75,19.5)--(22.75,26);   
\draw [line width=1pt] (29.25,19.5)--(29.25,26);   
\draw [line width=1pt] (35.75,19.5)--(35.75,26);   

\draw [line width=1pt] (16.25,26)--(16.25,32.5);   
\draw [line width=1pt] (22.75,26)--(22.75,32.5);   
\draw [line width=1pt] (29.25,26)--(29.25,32.5);   
\draw [line width=1pt] (35.75,26)--(35.75,32.5);   


\draw [line width=1pt] (24.375,19.5)--(26,26);

\draw [line width=1pt] (11.375,-6.5)--(13,0);

\draw [line width=1pt] (29.25,26)--(30.875,32.5);

\draw [line width=1pt] (16.25,0)--(17.875,6.5);

\draw [line width=1pt] (16.25,26)--(17.875,19.5);

\draw [line width=1pt] (29.25,26)--(30.875,19.5);

\draw [line width=1pt] (3.25,0)--(4.875,-6.5);

\draw [line width=1pt] (16.25,0)--(17.875,-6.5);

\draw [line width=1pt] (24.375,32.5)--(26,26);

\draw [line width=1pt] (37.375,32.5)--(39,26);

\draw [line width=1pt] (11.375,6.5)--(13,0);

\draw [line width=1pt] (24.375,6.5)--(26,0);

\draw [line width=1pt] (13,26)--(14.625,19.5);

\draw [line width=1pt] (26,26)--(27.625,19.5);

\draw [line width=1pt] (0,0)--(1.625,-6.5);

\draw [line width=1pt] (13,0)--(14.625,-6.5);

\draw [line width=1pt] (21.125,32.5)--(22.75,26);

\draw [line width=1pt] (34.125,32.5)--(35.75,26);

\draw [line width=1pt] (8.125,6.5)--(9.75,0);

\draw [line width=1pt] (21.125,6.5)--(22.75,0);

\draw [line width=1pt] (21.125,19.5)--(22.75,26);

\draw [line width=1pt] (8.125,-6.5)--(9.75,0);

\draw [line width=1pt] (26,26)--(27.625,32.5);

\draw [line width=1pt] (13,0)--(14.625,6.5);

\draw (0,-8.5) node[anchor=center] {$0$};
\draw (6.5,-8.5) node[anchor=center] {$\frac12p$};
\draw (13,-8.5) node[anchor=center] {$p$};
\draw (19.5,-8.5) node[anchor=center] {$\frac32p$};
\draw (26,-8.5) node[anchor=center] {$2p$};
\draw (32.5,-8.5) node[anchor=center] {$\frac52p$};
\draw (39,-8.5) node[anchor=center] {$3p$};
\draw (-2.5,-6.5) node[anchor=center] {$-\frac12p$};
\draw (-2.5,0) node[anchor=center] {$0$};
\draw (-2.5,6.5) node[anchor=center] {$\frac12p$};
\draw (-2.5,13) node[anchor=center] {$p$};
\draw (-2.5,19.5) node[anchor=center] {$\frac32p$};
\draw (-2.5,26) node[anchor=center] {$2p$};
\draw (-2.5,32.5) node[anchor=center] {$\frac52p$};
\end{tikzpicture}
	\end{minipage}%
 \captionsetup{skip=3pt}
    \caption{The left figure shows all points of $T_1^*$ that have four $(+2)$-equivalent copies in $\cG^+$, and the right figure shows their respective regionwise images under the mappings of \autoref{fig:slope2_fig34} and \autoref{fig:slope2_fig56}.}
    \vspace{1.2cm}
     \label{fig:slope2_fig78}
\end{figure}

\begin{figure}[h!!]
	\centering
	\begin{minipage}{0.47\textwidth}
		\centering
		\input{figures/slope2_fig9}
  \captionsetup{skip=3pt}
        \caption{Full partitioning of the union of all $(+2)$-equivalence classes consisting of four elements of $T_1$ into families of $(+2)$-equivalent regions.}
		\label{fig:slope2_fig9}
	\end{minipage}
    \hfill
	\begin{minipage}{0.47\textwidth}
		\centering
		\input{figures/slopehalf_fig9}
  \captionsetup{skip=3pt}
        \caption{Full partitioning of the union of all $\left(+\frac12\right)$-equivalence classes consisting of four elements of $T_1^*$ into families of $\left(+\frac12\right)$-equivalent regions.}
		\label{fig:slopehalf_fig9}
	\end{minipage}
\end{figure}
\begin{figure}[h!!]
	\begin{minipage}{0.47\textwidth}
		\centering
		\definecolor{removal}{rgb}{1, 0, 0}
\begin{tikzpicture}[line cap=round,line join=round,>=triangle 45,x=2.47cm,y=2.47cm, scale=0.85, transform shape]
\clip(-0.31,-0.81) rectangle (3.15,2.65);

\draw [line width=2pt] (0.,2.5)-- (0.,-0.5);
\draw [line width=2pt] (1.,-0.5)-- (1.,2.5);
\draw [line width=2pt] (2.,2.5)-- (2.,-0.5);
\draw [line width=2pt] (0.,0.)-- (3.,0.);
\draw [line width=2pt] (0.,2.)-- (3.,2.);
\draw [line width=2pt] (3.,1.)-- (0.,1.);
\draw [line width=2pt] (3.,2.5)-- (3.,-0.5);
\draw [line width=1pt] (0.5,2.5)-- (0.5,-0.5);
\draw [line width=1pt] (1.5,-0.5)-- (1.5,2.5);
\draw [line width=1pt] (2.5,2.5)-- (2.5,-0.5);
\draw [line width=1pt] (0.,2.5)-- (3.,2.5);
\draw [line width=1pt] (3.,1.5)-- (0.,1.5);
\draw [line width=1pt] (0.,0.5)-- (3.,0.5);
\draw [line width=1pt] (3.,-0.5)-- (0.,-0.5);

\draw [line width=1pt] (0.5,2.5)--(1,2);
\draw [line width=1pt] (0.75,2)--(0.75,2.5);
\draw [line width=1pt] (0.5,2)--(0.75,2.5);
\fill[pattern=horizontalodd, pattern color=removal](0.6667,2.3333)--(0.75,2.5)--(0.75,2.25)--cycle;

\draw [line width=1pt] (2.25,2.5)--(2.25,2);
\draw [line width=1pt] (2,2)--(2.5,2.5);
\draw [line width=1pt] (2.25,2.5)--(2.5,2);
\fill[pattern=horizontalodd, pattern color=removal](2.25,2.5)--(2.3333,2.3333)--(2.25,2.25)--cycle;

\draw [line width=1pt] (0.25, 1.5)--(0.5, 2);
\fill[pattern=horizontaleven, pattern color=removal](0.25, 1.5)--(0.5,1.5)--(0.5, 2)--cycle;

\draw [line width=1pt] (0.5,1.5)--(1,2);
\draw [line width=1pt] (0.75,1.5)--(0.75,2);
\draw [line width=1pt] (0.625,1.5)--(0.75,2);
\fill[pattern=horizontaleven, pattern color=removal](0.75,1.75)--(0.75,2)--(0.6667,1.6667)--cycle;

\draw [line width=1pt] (1,2)--(1.125,1.5);
\draw [line width=1pt] (1,2)--(1.25,1.5);
\fill[pattern=horizontalodd, pattern color=removal](1,2)--(1.125,1.5)--(1.25,1.5)--cycle;

\draw [line width=1pt] (1.75,1.5)--(2,2);
\draw [line width=1pt] (1.875,1.5)--(2,2);
\fill[pattern=horizontalodd, pattern color=removal](1.75,1.5)--(1.875,1.5)--(2,2)--cycle;

\draw [line width=1pt] (2.25,1.5)--(2.25,2);
\draw [line width=1pt] (2,2)--(2.5,1.5);
\draw [line width=1pt] (2.375,1.5)--(2.25,2);
\fill[pattern=horizontaleven, pattern color=removal](2.25,1.75)--(2.3333,1.6667)--(2.25,2)--cycle;

\draw [line width=1pt] (2.5,2)--(2.75,1.5);
\fill[pattern=horizontaleven, pattern color=removal](2.5,1.5)--(2.75,1.5)--(2.5,2)--cycle;

\draw [line width=1pt] (0.25,0.5)--(0.5,0);
\fill[pattern=horizontalodd, pattern color=removal](0.25,0.5)--(0.5,0.5)--(0.5,0)--cycle;

\draw [line width=1pt] (0.625,0.5)--(0.75,0);
\draw [line width=1pt] (0.75,0.5)--(0.75,0);
\draw [line width=1pt] (0.5,0.5)--(1,0);
\fill[pattern=horizontalodd, pattern color=removal](0.6667,0.3333)--(0.75,0)--(0.75,0.25)--cycle;

\draw [line width=1pt] (1,0)--(1.125,0.5);
\draw [line width=1pt] (1,0)--(1.25,0.5);
\fill[pattern=horizontaleven, pattern color=removal](1,0)--(1.125,0.5)--(1.25,0.5)--cycle;

\draw [line width=1pt] (1.75,0.5)--(2,0);
\draw [line width=1pt] (1.875,0.5)--(2,0);
\fill[pattern=horizontaleven, pattern color=removal](1.75,0.5)--(1.875,0.5)--(2,0)--cycle;

\draw [line width=1pt] (2.25,0)--(2.25,0.5);
\draw [line width=1pt] (2,0)--(2.5,0.5);
\draw [line width=1pt] (2.25,0)--(2.375,0.5);
\fill[pattern=horizontalodd, pattern color=removal](2.25,0)--(2.3333,0.3333)--(2.25,0.25)--cycle;

\draw [line width=1pt] (2.5,0)--(2.75,0.5);
\fill[pattern=horizontalodd, pattern color=removal](2.5,0)--(2.5,0.5)--(2.75,0.5)--cycle;

\draw [line width=1pt] (0.5,0)--(0.75,-0.5);
\draw [line width=1pt] (0.75,-0.5)--(0.75,0);
\draw [line width=1pt] (0.5,-0.5)--(1,0);
\fill[pattern=horizontaleven, pattern color=removal](0.6667,-0.3333)--(0.75,-0.25)--(0.75,-0.5)--cycle;

\draw [line width=1pt] (2.25,-0.5)--(2.25,0);
\draw [line width=1pt] (2,0)--(2.5,-0.5);
\draw [line width=1pt] (2.25,-0.5)--(2.5,0);
\fill[pattern=horizontaleven, pattern color=removal](2.25,-0.5)--(2.25,-0.25)--(2.3333,-0.3333)--cycle;

\draw (0,-0.65) node[anchor=center] {$0$};
\draw (0.5,-0.65) node[anchor=center] {$\frac12p$};
\draw (1,-0.65) node[anchor=center] {$p$};
\draw (1.5,-0.65) node[anchor=center] {$\frac32p$};
\draw (2,-0.65) node[anchor=center] {$2p$};
\draw (2.5,-0.65) node[anchor=center] {$\frac52p$};
\draw (3,-0.65) node[anchor=center] {$3p$};
\draw (-0.19,-0.5) node[anchor=center] {$-\frac12p$};
\draw (-0.19,0) node[anchor=center] {$0$};
\draw (-0.19,0.5) node[anchor=center] {$\frac12p$};
\draw (-0.19,1) node[anchor=center] {$p$};
\draw (-0.19,1.5) node[anchor=center] {$\frac32p$};
\draw (-0.19,2) node[anchor=center] {$2p$};
\draw (-0.19,2.5) node[anchor=center] {$\frac52p$};
\end{tikzpicture}
        \captionsetup{skip=3pt}
		\caption{Points removed from $T_1$ because of the partitionings for $m=\pm 2$.}
		\label{fig:slope2_removed}
	\end{minipage}
	\hfill
	\begin{minipage}{0.47\textwidth}
		\centering
		
		\definecolor{removal}{rgb}{1, 0, 0}
\begin{tikzpicture}[line cap=round,line join=round,>=triangle 45,x=2.47cm,y=2.47cm, scale=0.85, transform shape]
\clip(-0.31,-0.81) rectangle (3.15,2.65);

\draw [line width=2pt] (0.,2.5)-- (0.,-0.5);
\draw [line width=2pt] (1.,-0.5)-- (1.,2.5);
\draw [line width=2pt] (2.,2.5)-- (2.,-0.5);
\draw [line width=2pt] (0.,0.)-- (3.,0.);
\draw [line width=2pt] (0.,2.)-- (3.,2.);
\draw [line width=2pt] (3.,1.)-- (0.,1.);
\draw [line width=2pt] (3.,2.5)-- (3.,-0.5);
\draw [line width=1pt] (0.5,2.5)-- (0.5,-0.5);
\draw [line width=1pt] (1.5,-0.5)-- (1.5,2.5);
\draw [line width=1pt] (2.5,2.5)-- (2.5,-0.5);
\draw [line width=1pt] (0.,2.5)-- (3.,2.5);
\draw [line width=1pt] (3.,1.5)-- (0.,1.5);
\draw [line width=1pt] (0.,0.5)-- (3.,0.5);
\draw [line width=1pt] (3.,-0.5)-- (0.,-0.5);

\draw [line width=1pt] (0,2)--(0.5,2.5);
\draw [line width=1pt] (0,2)--(1,2.5);
\fill[pattern=verticaleven, pattern color=removal](0,2)--(0.5,2.5)--(1,2.5)--cycle;

\draw [line width=1pt] (2,2.5)--(3,2);
\draw [line width=1pt] (2.5,2.5)--(3,2);
\fill[pattern=verticalodd, pattern color=removal](2,2.5)--(2.5,2.5)--(3,2)--cycle;

\draw [line width=1pt] (1,2)--(0,1.75);
\draw [line width=1pt] (1,2)--(0.5,1.75);
\draw [line width=1pt] (1,1.75)--(0,1.75);
\draw [line width=1pt] (0.5,1.5)--(0,2);
\draw [line width=1pt] (0.5,1.75)--(0,1.5);
\draw [line width=1pt] (1,1.5)--(0,1.75);
\draw [line width=1pt] (1,2)--(0.5,1.5);
\fill[pattern=verticalodd, pattern color=removal](1,2)--(0.5,1.75)--(0.75,1.75)--cycle;
\fill[pattern=verticalodd, pattern color=removal](0.5,1.875)--(0.2,1.8)--(0.25,1.75)--(0.5,1.75)--cycle;
\fill[pattern=verticalodd, pattern color=removal](0.5,1.75)--(0.5,1.625)--(0.3333,1.6667)--cycle;

\draw [line width=1pt] (2,2)--(3,1.75);
\draw [line width=1pt] (2,2)--(2.5,1.75);
\draw [line width=1pt] (2,1.75)--(3,1.75);
\draw [line width=1pt] (2.5,1.5)--(3,2);
\draw [line width=1pt] (2.5,1.75)--(3,1.5);
\draw [line width=1pt] (2,1.5)--(3,1.75);
\draw [line width=1pt] (2,2)--(2.5,1.5);
\fill[pattern=verticaleven, pattern color=removal](2,2)--(2.5,1.75)--(2.25,1.75)--cycle;
\fill[pattern=verticaleven, pattern color=removal](2.5,1.875)--(2.8,1.8)--(2.75,1.75)--(2.5,1.75)--cycle;
\fill[pattern=verticaleven, pattern color=removal](2.5,1.75)--(2.5,1.625)--(2.6667,1.6667)--cycle;

\draw [line width=1pt] (1,0)--(0,0.25);
\draw [line width=1pt] (1,0)--(0.5,0.25);
\draw [line width=1pt] (1,0.25)--(0,0.25);
\draw [line width=1pt] (0.5,0.5)--(0,0);
\draw [line width=1pt] (0.5,0.25)--(0,0.5);
\draw [line width=1pt] (1,0.5)--(0,0.25);
\draw [line width=1pt] (1,0)--(0.5,0.5);
\fill[pattern=verticalodd, pattern color=removal](1,0)--(0.5,0.25)--(0.75,0.25)--cycle;
\fill[pattern=verticalodd, pattern color=removal](0.5,0.125)--(0.2,0.2)--(0.25,0.25)--(0.5,0.25)--cycle;
\fill[pattern=verticalodd, pattern color=removal](0.5,0.25)--(0.5,0.375)--(0.3333,0.3333)--cycle;

\draw [line width=1pt] (2,0)--(3,0.25);
\draw [line width=1pt] (2,0)--(2.5,0.25);
\draw [line width=1pt] (2,0.25)--(3,0.25);
\draw [line width=1pt] (2.5,0.5)--(3,0);
\draw [line width=1pt] (2.5,0.25)--(3,0.5);
\draw [line width=1pt] (2,0.5)--(3,0.25);
\draw [line width=1pt] (2,0)--(2.5,0.5);
\fill[pattern=verticaleven, pattern color=removal](2,0)--(2.5,0.25)--(2.25,0.25)--cycle;
\fill[pattern=verticaleven, pattern color=removal](2.5,0.125)--(2.8,0.2)--(2.75,0.25)--(2.5,0.25)--cycle;
\fill[pattern=verticaleven, pattern color=removal](2.5,0.25)--(2.5,0.375)--(2.6667,0.3333)--cycle;

\draw [line width=1pt] (0,0)--(0.5,-0.5);
\draw [line width=1pt] (0,0)--(1,-0.5);
\fill[pattern=verticaleven, pattern color=removal](0,0)--(0.5,-0.5)--(1,-0.5)--cycle;

\draw [line width=1pt] (2,-0.5)--(3,0);
\draw [line width=1pt] (2.5,-0.5)--(3,0);
\fill[pattern=verticalodd, pattern color=removal](2,-0.5)--(2.5,-0.5)--(3,0)--cycle;

\draw (0,-0.65) node[anchor=center] {$0$};
\draw (0.5,-0.65) node[anchor=center] {$\frac12p$};
\draw (1,-0.65) node[anchor=center] {$p$};
\draw (1.5,-0.65) node[anchor=center] {$\frac32p$};
\draw (2,-0.65) node[anchor=center] {$2p$};
\draw (2.5,-0.65) node[anchor=center] {$\frac52p$};
\draw (3,-0.65) node[anchor=center] {$3p$};
\draw (-0.19,-0.5) node[anchor=center] {$-\frac12p$};
\draw (-0.19,0) node[anchor=center] {$0$};
\draw (-0.19,0.5) node[anchor=center] {$\frac12p$};
\draw (-0.19,1) node[anchor=center] {$p$};
\draw (-0.19,1.5) node[anchor=center] {$\frac32p$};
\draw (-0.19,2) node[anchor=center] {$2p$};
\draw (-0.19,2.5) node[anchor=center] {$\frac52p$};
\end{tikzpicture}
  \captionsetup{skip=3pt}
		\caption{Points removed from $T_1$ because of the partitionings for $m=\pm\frac12$.}
		\label{fig:slopehalf_removed}
	\end{minipage}
\end{figure}

\newpage
\section*{Acknowledgments} 
The authors are grateful to the anonymous referees for their useful comments and suggestions that helped us improve the presentation of the manuscript.



\begin{aicauthors}
\begin{authorinfo}[benedek]
  Benedek Kov{\'a}cs\\
  Eötvös Loránd University\\
  Budapest, Hungary\\
benoke981\imageat{}gmail\imagedot{}com
\end{authorinfo}
\begin{authorinfo}[zoli]
  Zolt{\'a}n L{\'o}r{\'a}nt Nagy\\
  Eötvös Loránd University\\
  Budapest, Hungary\\
zoltan\imagedot{}lorant\imagedot{}nagy\imageat{}ttk\imagedot{}elte\imagedot{}hu
\end{authorinfo}
\begin{authorinfo}[david]
  D{\'a}vid R. Szab{\'o}\\
  Eötvös Loránd University \& Alfr{\'e}d R{\'e}nyi Institute of Mathematics\\
  Budapest, Hungary\\
szabo\imagedot{}r\imagedot{}david\imageat{}gmail\imagedot{}com
\end{authorinfo}
\end{aicauthors}

\end{document}